\documentclass[a4paper]{amsart}

\usepackage[all]{xy}
\usepackage{graphicx}
\usepackage{amsmath}
\usepackage{amssymb}
\usepackage{amsthm}
\usepackage{latexsym}
\usepackage{enumerate}
\usepackage{prettyref}
\usepackage{hyperref}
\usepackage{mathrsfs}
\usepackage{youngtab}

\newtheorem{theorem}{Theorem}[section]
\newrefformat{thm}{\hyperref[{#1}]{Theorem~\ref*{#1}}}
\newtheorem{definition}[theorem]{Definition}
\newrefformat{def}{\hyperref[{#1}]{Definition~\ref*{#1}}}
\newtheorem{lemma}[theorem]{Lemma}
\newrefformat{lem}{\hyperref[{#1}]{Lemma~\ref*{#1}}}
\newtheorem{proposition}[theorem]{Proposition}
\newrefformat{prop}{\hyperref[{#1}]{Proposition~\ref*{#1}}}
\newtheorem{corollary}[theorem]{Corollary}
\newrefformat{cor}{\hyperref[{#1}]{Corollary~\ref*{#1}}}
\newtheorem{question}[theorem]{Question}

\theoremstyle{remark}
\newtheorem{nremark}[theorem]{Remark}
\newenvironment{remark}
{\par\pushQED{\qed}\nremark \small}
{\popQED\endnremark}

{%\theorembodyfont{\upshape}
\newtheorem{examplecore}[theorem]{Example}}
\newrefformat{ex}{\hyperref[{#1}]{Example~\ref*{#1}}}

\numberwithin{equation}{section}

\newenvironment{example}{\begin{examplecore}}{\hspace*{\fill}
$\square$\par\vspace{.1cm}\end{examplecore}}

\newcommand{\op}{\operatorname}

\newcommand{\om}{\omega}

\newcommand{\coker}{\op{coker}}

\newcommand{\inj}{\hookrightarrow}
\newcommand{\iso}{\cong}

\newcommand{\Hom}{\op{Hom}}
\newcommand{\Fl}{\op{Fl}}
\newcommand{\Gr}{\op{Gr}}

\newcommand{\id}{\op{id}}
\newcommand{\CH}{\op{CH}}
\newcommand{\Ch}{\op{Ch}}
\newcommand{\GL}{\op{GL}}
\newcommand{\PP}{\mathbb{P}}
\newcommand{\A}{\mathbb{A}}
\newcommand{\Q}{\mathbb{Q}}
\newcommand{\R}{\mathbb{R}}

\newcommand{\Sq}{\op{Sq}}

\newcommand{\Z}{\mathbb{Z}}
\newcommand{\C}{\mathbb{C}}

\newcommand{\stb}{,\ldots, }
\newcommand{\se}{\subseteq}
\newcommand{\bra}{\langle}
\newcommand{\ket}{\rangle}
\newcommand{\al}{\alpha}
\newcommand{\be}{\beta}
\newcommand{\ga}{\gamma}
\newcommand{\Ga}{\Gamma}
\newcommand{\De}{\Delta}
\newcommand{\de}{\delta}

\newcommand{\ka}{\kappa}
\newcommand{\la}{\lambda}
\newcommand{\La}{\Lambda}
\newcommand{\si}{\sigma}
\newcommand{\Si}{\Sigma}

\newcommand{\su}{\backslash}

\newcommand{\D}{D}
\renewcommand{\L}{\mathscr{L}}

\newcommand{\EE}{\mathcal{E}}
\newcommand{\FF}{\mathcal{F}}
\renewcommand{\SS}{\mathcal{S}}
\newcommand{\DD}{\mathcal{D}}
\newcommand{\QQ}{\mathcal{Q}}

\newcommand{\Ann}{\op{Ann}}
\newcommand{\e}{{\rm e}}
\newcommand{\p}{{\rm p}}
\newcommand{\cc}{{\rm c}}
\newcommand{\symmdiff}{\triangle}

\begin{document}

\title{Chow--Witt rings and topology of flag varieties}

\author{Thomas Hudson, \'Akos K.\ Matszangosz and Matthias Wendt}

\date{August 2024}

\address{Thomas Hudson, College of Transdisciplinary Studies, DGIST, 
%333, Techno jungang-daero
%Hyeonpung-eup, Dalseong-gun,
Daegu, 42988, Republic of Korea}
\email{hudson@digst.ac.kr}

\address{\'Akos K.\ Matszangosz, HUN-REN Alfr\'ed R\'enyi Institute of Mathematics, Re\'altanoda utca 13-15, 1053 Budapest, Hungary}
\email{matszangosz.akos@gmail.com}

\address{Matthias Wendt, Fachgruppe Mathematik und Informatik, Bergische Universit\"at Wuppertal, Gaussstrasse 20, 42119 Wuppertal, Germany}
\email{m.wendt.c@gmail.com}

\thanks{\'A. K. M. is supported by the Hungarian National Research, Development and Innovation Office, NKFIH K 138828 and NKFIH PD 145995. Part of this work was done while \'A. K. M. was at the University of Wuppertal.}

\subjclass[2010]{14M15,14F25,19G12,57T15}
\keywords{Witt-sheaf cohomology, Chow--Witt rings, singular cohomology, (real) flag varieties, characteristic classes, refined enumerative geometry}

\begin{abstract}
  The paper computes the Witt-sheaf cohomology rings of partial flag varieties in type~A in terms of the Pontryagin classes of the subquotient bundles. The proof is based on a Leray--Hirsch-type theorem for Witt-sheaf cohomology for the maximal rank cases, and a detailed study of cohomology ring presentations and annihilators of characteristic classes for the general case. The computations have consequences for the topology of real flag manifolds: we show that all torsion in the integral cohomology is 2-torsion, which was not known in full generality previously. This allows for example to compute the Poincar\'e polynomials of complete flag varieties for cohomology with twisted integer coefficients. The computations also allow to describe the Chow--Witt rings of flag varieties, and we sketch an enumerative application to counting flags satisfying multiple incidence conditions to given hypersurfaces.
\end{abstract}

\maketitle
\setcounter{tocdepth}{1}
\tableofcontents

\section{Introduction}

The goal of the paper is to establish formulas describing the Witt-sheaf cohomology and Chow--Witt rings of partial flag varieties in type A, generalizing previous computations for Grassmannians in \cite{realgrassmannian}. The resulting formulas (and some of the methods used to establish them) are very close to the corresponding results for integral cohomology of real flag manifolds, cf.\ in particular~\cite{matszangosz}. \newpage

\textbf{Real vs.\ complex.} There are two significantly different settings to study flag varieties: over the complex and over the real numbers. On the complex side, one can consider the Chow ring or the singular cohomology of the complex flag varieties. These turn out to be isomorphic via the complex cycle class map \cite{Fulton}, and additively freely generated by the Bruhat cells. In this situation, both the characteristic class picture and the Schubert calculus picture are very well studied and understood, with strong links and applications to enumerative geometry and geometric representation theory. On the other hand, real flag varieties have been much less studied. Only recently, algebraic cohomology theories such as Chow--Witt groups or Witt-sheaf cohomology have been developed \cite{fasel:memoir} which can be compared to the integral singular cohomology of real flag manifolds via Jacobson's real cycle class map \cite{4real}. Even the structure of integral singular cohomology itself is complicated and (as far as we are aware) not completely known. Well-known computations exist with mod 2 or $\mathbb{Z}[1/2]$ coefficients, e.g.~\cite{he}, as well as in specific cases such as Grassmannians \cite{realgrassmannian} and maximal rank flag varieties \cite{matszangosz} or infinite and hence ``stable'' cases \cite{markl}. These results demonstrate that the integral singular cohomology of real flag varieties has a lot of complicated 2-torsion related to the integral Stiefel--Whitney classes of the tautological subquotient bundles. But while the known computations suggest that for real flag manifolds in general all the torsion in the integral singular cohomology is 2-torsion, this appears not to be known in full generality.

\textbf{Characteristic classes and Schubert calculus.} For an explicit description of the cohomology of flag varieties, as well as applications, there are (at least) two relevant pictures that can be considered. On the one hand, there is the characteristic class picture, where characteristic classes of the tautological subquotient bundles provide cohomology classes, with relations coming from the Whitney sum formula and other universal properties of characteristic classes. In favourable circumstances, such as for the Chow rings of flag varieties, one obtains a presentation of the cohomology ring with characteristic classes as generators and very explicit (though combinatorially complicated) relations. On the other hand, there is the Schubert calculus picture, where the Schubert cells (alternatively called Bruhat cells) of the flag variety provide additive generators of cohomology, and the intersection numbers for Schubert varieties are the structure constants which describe the multiplicative structure. The latter picture in particular gives rise to applications in enumerative geometry, where counting problems related to (flags of) subspaces are translated into computation of intersection numbers in Grassmannians or flag varieties.

\textbf{Summary of main results.} The present paper focuses on the characteristic class picture on the real flag varieties side. The main results in the paper describe the cohomology theories on the algebraic side, such as the Witt-sheaf cohomology or, ultimately, the Chow--Witt rings of partial flag varieties, cf.~Theorems~\ref{thm:untwisted} and \ref{thm:main} below. The short informal description of the results is that the Witt-sheaf cohomology is generated by the characteristic classes (Euler classes and Pontryagin classes) of the tautological subquotient bundles, plus additional exterior classes which are transgressions of Euler classes on related flag varieties. One of the principal motivations for our results -- besides the intrinsic value of obtaining a presentation of cohomology rings of such fundamental a class of examples as flag varieties -- is the long-term goal of establishing a version of Schubert calculus in the algebraic setting of Witt-sheaf cohomology or Witt rings. This, however, is beyond the scope of the present paper, and we will undertake the development of a Schubert calculus for Chow--Witt rings of flag varieties and subsequent refined enumerative geometry of degeneracy loci in sequels to the present paper.

The more immediate payoff from these computations is a complete description of the integral singular cohomology of real flag manifolds. As mentioned above, the remaining open question is a 2-primary torsion question. Using the real cycle class map isomorphism of \cite{4real}, we can deduce that all torsion in the integral singular cohomology of partial flag manifolds of type A is 2-torsion, cf.~Theorem~\ref{thm:main-2torsion} below. The reason why the detour through algebraic-geometric cohomology theories is helpful at all is that, as with Hodge structures on the cohomology of complex varieties, the algebraic theories have additional ``weight'' information which provides more structure than is visible in just singular cohomology. More concretely, the Witt-sheaf cohomology of the flag varieties encodes an analogue of ``integral singular cohomology modulo the image of the Bockstein map'', which cannot be formulated on the topological side. As direct applications of our cohomology computations, we can compute the cohomology groups of the complete flag manifolds and, in particular, determine the ranks of the 2-torsion components, cf.~Theorem~\ref{thm:completeflag}. On the algebraic side, the computations of Witt-sheaf cohomology also extend to a description of the Chow--Witt rings (with a slight indeterminacy in the reductions of the exterior classes), cf.\ the discussion in Section~\ref{sec:relationship}. 

\subsection{Tools for computing Witt-sheaf cohomology}
The key computation in the paper concerns the Witt-sheaf cohomology ${\rm H}^*\bigl(X,\mathbf{W}(\L)\bigr)$, i.e., the Nisnevich cohomology of a smooth scheme $X$ with coefficients in the sheaf $\mathbf{W}(\L)$ of Witt rings twisted by a line bundle $\L$ over $X$, cf.~Section~\ref{sec:prelims} for background. From this, computations of $\mathbf{I}$-cohomology or Chow--Witt groups follow in a rather straightforward manner. The properties of Witt-sheaf cohomology as outlined in Section~\ref{sec:prelims} also make it more amenable to computations, most of the time very similar to computations of singular cohomology. In particular, for the purposes of the present paper, we develop $\mathbf{W}$-cohomology analogues of two particularly useful computational tools from singular cohomology. 

On the one hand, we establish in Proposition~\ref{prop:kuenneth} \emph{a K\"unneth formula for $\mathbf{W}$-cohomology}, which expresses the (twisted) $\mathbf{W}$-cohomology of a product $X\times Y$ as a tensor product
\[
  {\rm H}^*\bigl(X,\mathbf{W}(\L_X)\bigr)\otimes_{{\rm W}(F)}{\rm H}^*\bigl(Y,\mathbf{W}(\L_Y)\bigr)\cong {\rm H}^*\bigl(X\times Y,\mathbf{W}(\L_X\boxtimes \L_Y)\bigr)
\]
of the $\mathbf{W}$-cohomologies of the factors, provided one of the factors is cellular. Based on this, we also establish a \emph{$\mathbf{W}$-cohomology version of the Leray--Hirsch theorem}, expressing the (twisted) $\mathbf{W}$-cohomology of the total space of a Zariski-locally trivial fiber bundle with cellular fiber as tensor product of the $\mathbf{W}$-cohomologies of base and fiber, cf.~Proposition~\ref{prop:leray-hirsch}. 

The Leray--Hirsch theorem can be used to compute $\mathbf{W}$-cohomology of many Grassmannian bundles. For a rank $n$ vector bundle $\EE$ over a smooth scheme $X$, the Leray--Hirsch theorem applies to the Grassmannian bundle $\mathscr{G}(k,\EE)$ of rank $k$ subbundles of $\EE$ over $X$ as long as the dimension $\dim{\rm Gr}(k,n)=k(n-k)$ of the fiber is even. This is the principal tool for the computation of $\mathbf{W}$-cohomology of maximal rank flag varieties. Moreover, the Leray--Hirsch theorem in fact applies to maximal rank flag bundles as well, a fact which we think will be very useful for the future study of refined enumerative geometry of degeneracy loci.

\subsection{$\mathbf{W}$-cohomology: multiplicative description}
For partial flag varieties of type A, we will now describe the multiplicative structure of the $\mathbf{W}$-cohomology, as computed in the paper. They are generated by cohomology classes arising from the tautological subquotient bundles $\DD_i$ -- Pontryagin classes $\op{p}_j(\DD_i)$, and Euler classes ${\rm e}(\DD_i)$ in appropriate twisted cohomology groups --, and certain additional classes $\op{R}_l$, which are not characteristic classes, but are related to Euler classes on other flag varieties.

Before making precise statements, let us introduce some notation, cf.\ also Section~\ref{sec:flags}. Let $\D=(d_1\stb d_m)$ be a sequence of nonnegative integers, and let $\Fl(\D)$ denote the variety of flags 
$$V_\bullet=\Bigl((0)=V_0\se V_1\se \ldots\se V_m=\mathbb{A}^N\Bigr),$$
such that $\dim V_i/V_{i-1}=d_i$ and $N=\sum_{i=1}^md_i$. Denote by $\DD_i$ the subquotient bundle whose fiber over $V_\bullet$ is $V_i/V_{i-1}$. For a vector bundle $\EE$ of rank $n$ over $X$, let $\op{p}_j(\EE)\in \op{H}^{2j}(X, \mathbf{W})$ denote its $j$th Pontryagin class and let
$$\op{p}(\EE)=1+\op{p}_1(\EE)+\cdots+\op{p}_{n}(\EE)$$ 
denote the total Pontryagin class. Note that the indexing of Pontryagin classes used here follows the unconventional conventions set out in \cite{chowwitt,realgrassmannian} and includes the odd Pontryagin classes, without reindexing or additional signs, cf.~Remark~\ref{rem:pontryagin}. Note also that the odd Pontryagin classes are trivial in Witt-sheaf cohomology. The multiplicative structure of the Witt-sheaf cohomology of $\Fl(\D)$ is then given by the following theorem.

\begin{theorem}\label{thm:untwisted}
  Let $F$ be a perfect field of characteristic $\neq 2$. Then the $\mathbf{W}$-cohomology ring of $\Fl(\D)$ (with untwisted coefficients) is:
  $$
  \op{H}^*\bigl(\Fl(\D),\mathbf{W}\bigr)\cong
  \frac{\op{W}(F)[\op{p}_{2j}(\DD_i) ]}{\prod_{i=1}^m \op{p}(\DD_i)=1}
  \otimes \bigwedge\nolimits_{l=q+1}^n\langle \op{R}_l\rangle,
  $$
  where for all $i=1,\dots,m$, the $\op{p}_{2j}(\DD_i)$, $j=1\stb \lfloor\frac{d_i}{2}\rfloor$, are the relevant even Pontryagin classes of the subquotient bundle $\DD_i$. 
  In the exterior algebra part, $q=\sum_{i=1}^m\lfloor \frac{d_i}{2}\rfloor$, $n=\lfloor \frac N 2\rfloor$ and the degree of $R_l$ is $4l-1$, with one exception: when $N$ is even, $R_n$ has degree $N-1$.
\end{theorem}

This formula is analogous to the description of $\mathbb{Z}[1/2]$-cohomology of real flag manifolds, cf.\ e.g.\ \cite{he}, \cite{matszangosz}. For the case of complete flags, additive splittings of Milnor--Witt motives have been obtained by Yang \cite{yang}. One could expect that versions of Cartan's and Borel's formulas hold for projective homogeneous varieties of other types, but for that we would first need to know the Chow--Witt rings for reductive groups of the relevant types (which is not the case at this point). 

Our proof of Theorem~\ref{thm:untwisted} actually computes the total $\mathbf{W}$-cohomology ring including twisted local coefficients. This has the structure of a $\Z\oplus \op{Pic}\bigl(\Fl(\D)\bigr)/2$-graded $\op{W}(F)$-algebra, so let us recall that the Picard group of $\Fl(\D)$, for $\D=(d_1\stb d_m)$ equals
$$ \op{Pic}\bigl(\Fl(\D)\bigr)=\Z\bra \ell_1\stb \ell_m\ket\bigg/ \sum \ell_i,$$
where $\ell_i$ denotes the first Chern class ${\rm c}_1(\DD_i)$ of the bundle $\DD_i$. The full description of the $\mathbf{W}$-cohomology ring with twisted coefficients is then given by the following theorem. See also Definition~\ref{def:WD} for a slightly different way of writing out this presentation of the Witt-sheaf cohomology ring of $\Fl(\D)$.

\begin{theorem}
  \label{thm:main}
With the notation of Theorem~\ref{thm:untwisted}, the complete $\Z\oplus \op{Pic}\bigl(\Fl(\D)\bigr)/2$-graded $\mathbf{W}$-cohomology ring equals:
$$
\bigoplus_\L \op{H}^*\bigl(\op{Fl}(\D),\mathbf{W}(\L)\bigr)=
\op{H}^*\bigl(\op{Fl}(\D),\mathbf{W}\bigr)\bigotimes 
\op{W}(F)[\op{e}(\DD_i) \mid i=1\stb m]\Big/ \sim ,
$$
where the Euler classes ${\rm e}(\DD_i)$ satisfy the following relations $\sim$:
$$ \begin{cases}
  \op e(\DD_i)=0, \qquad &\text{if } d_i \text{ is odd, and}\\
  \op e(\DD_i)^2=\op{p}_{d_i}(\DD_i), \qquad &\text{if } d_i \text{ is even,}
\end{cases}$$
with an additional relation in the case when all $d_i$ are even:
$$ \prod_{i=1}^m\op{e}(\DD_i)=0.$$
\end{theorem}
For the orientation covers of maximal rank flag varieties and $\eta$-invertible cohomology theories, similar formulas have been obtained in \cite{ananyevskiy}. Our strategy to prove Theorem~\ref{thm:main} for $\Fl(\D)$ is by induction and consists of the following steps:
\begin{itemize}
\item The base case $\D=(k,l)$ of Grassmannians was completely settled in \cite{realgrassmannian}.
\item When $\D$ is of maximal rank, i.e.,\ when the number of odd $d_i$'s in $\D$ is at most one, we compute the $\mathbf{W}$-cohomology of $\Fl(\D)$ as a tower of iterated Grassmannian bundles in Section~\ref{sec:maximal}, using the Leray--Hirsch theorem for $\mathbf{W}$-cohomology, cf.\ Proposition~\ref{prop:leray-hirsch}. 
\item To obtain the general case, we show in Section~\ref{sec:sadykov} how the $\mathbf{W}$-cohomology of $\Fl(\D)$ changes under the operation
  $$\D=(d_1\stb d_m)\mapsto \D'=(d_1\stb d_i-1,d_{i+1}+1\stb d_m)$$
  and the permutation action of the symmetric group $\Sigma_m$ on $\D$. It is elementary to see that any sequence $\D$ can be obtained from a maximal rank sequence by a repeated use of these two operations.
\end{itemize}

\subsection{Witt groups of flag varieties}
  Presumably, similar methods allow to compute the twisted Witt groups of flag varieties (which also do not seem to be known in full generality at present, but see \cite{zibrowius}). Conjecturally, the result would be a decomposition 
  \[
  \op{W}^i\bigl(\op{Fl}(\D),\L\bigr)\xrightarrow{\cong} \bigoplus_{j\equiv i\bmod 4}\op{H}^j\bigl(\op{Fl}(\D),\mathbf{W}(\L)\bigr)
  \]
  as in the Grassmannian case, cf.~\cite{schubert}. In some cases, such a decomposition already follows from our Witt-sheaf cohomology computation in Theorem~\ref{thm:main} using the Gersten--Witt spectral sequence
  \[
  {\rm H}^*(X,{\bf W}(\mathscr{L}))\Rightarrow {\rm W}^*(X,\mathscr{L})
  \]
  of Balmer and Walter \cite{balmer:walter}. Its differentials all have cohomological degrees congruent to 1 mod 4, so as long as ${\rm H}^*(X,{\bf W})$ is concentrated in even degrees, the spectral sequence necessarily degenerates. This is the case for maximal rank flag varieties, i.e., $\Fl(D)$  where $\D=(d_1,\dots,d_m)$ has at most one odd entry. The spectral sequence degenerates for degree reasons in a couple more cases, such as some low-dimensional full flag varieties $\Fl(N)$.
  
  In light of the computations in \cite{zibrowius} (which primarily deals with the full flag variety case in which we only have the exterior algebra part), one should probably expect that the Gersten--Witt spectral sequence always degenerates. We will investigate this and the question of ring structures in future work on Schubert calculus for Chow--Witt rings of flag varieties. 

  The computations of hermitian K-theory of Grassmannians in \cite{huang:xie}, parallel to the Witt-sheaf computations in \cite{schubert,realgrassmannian},  also suggest that the hermitian K-theory of flag varieties (at least in type A) should have an additive decomposition into hermitian summands corresponding to the non-torsion summands in Witt-sheaf cohomology, plus algebraic K-theory summands corresponding to the 2-torsion in ${\bf I}$-cohomology, i.e., the image of the Bockstein.

  \subsection{Chow--Witt rings}
  While the main bulk of the paper is concerned with computations of Witt-sheaf cohomology of flag varieties, we can also deduce a description of the Chow--Witt rings $\widetilde{\op{CH}}^*\bigl(\Fl(\D)\bigr)$ of flag varieties. There are a couple of standard decompositions and reductions known which reduce the computation of Chow--Witt rings to $\mathbf{W}$-cohomology. First, the total Chow--Witt ring is a fiber product of a subring of the Chow ring and the $\mathbf{I}$-cohomology ring, cf.\ Proposition~\ref{prop:fiberprod}. In a second step, the $\mathbf{I}$-cohomology ring can be completely recovered from the $\mathbf{W}$-cohomology ring and the image of the Bockstein maps $\beta_{\L}$ provided the $\mathbf{W}$-cohomology is free as a module over the Witt ring $\op{W}(F)$ of the base field, cf.\ Lemma~\ref{lem:wsplit}. The freeness of the $\mathbf{W}$-cohomology is one of the consequences of the main theorems~\ref{thm:untwisted} and \ref{thm:main} formulated above. The description of the Chow--Witt ring thus reduces to two main inputs: on the one hand, the integral and mod 2 Chow rings of flag varieties together with the action of the Steenrod squares $\op{Sq}^2_{\L}$ on the latter, which is essentially known (albeit complicated); on the other hand, the $\mathbf{W}$-cohomology of flag varieties which is the central computation in our paper. There is one subtlety that makes the above strategy a little less straightforward than the above description might suggest. To describe the Chow--Witt rings, knowledge of the reduction morphism ${\rm H}^n\bigl(\Fl(\D),\mathbf{I}^n\bigr)\to {\rm Ch}^n\bigl(\Fl(\D)\bigr)$ from $\mathbf{I}$-cohomology to mod 2 Chow groups is necessary, cf.\ the computations in Section~\ref{sec:chowwittsingular}. 

\subsection{Application: integer coefficient cohomology of real flag manifolds}
The computation above has a specific topological consequence. The result in Theorem~\ref{thm:main} states in particular that the $\mathbf{W}$-cohomology of the flag varieties $\op{Fl}(\D)$ is free, hence $\mathbf{I}$-cohomology splits as direct sum of the $\mathbf{W}$-cohomology and the image of $\beta$, cf.\ Theorem~\ref{thm:icohomology}. Over the real numbers, we can then use that the real cycle class map 
$$
cy\colon{\rm H}^n\bigl(\Fl(\D), \mathbf{I}^n(\mathcal{L})\bigr)\to {\rm H}^n\bigl(\Fl(\D,\mathbb{R}\bigr), \mathbb{Z}\bigl(\mathcal{L})\bigr)
$$
is an isomorphism for cellular varieties, cf.\ \cite{4real}. As a consequence, we get that all torsion in the integral cohomology of real flag manifolds is 2-torsion, cf.~Theorem~\ref{thm:realflag}:

\begin{theorem}
  \label{thm:main-2torsion}
  All torsion in $\op{H}^*\bigl(\op{Fl}(\D,\mathbb{R}),\mathbb{Z}(\L)\bigr)$ is of order 2, for any $\D=(d_1,\dots,d_m)$ and for any line bundle $\L$ over the real flag manifold $\op{Fl}(\D,\mathbb{R})$.
\end{theorem}

This extends known results in the maximal rank case, cf.\ \cite{matszangosz}, and is established in Theorem~\ref{thm:realflag}. For complete real flag varieties, this result was obtained by splitting Milnor--Witt motives by Yang \cite{yang}.  Even in the maximal rank case, the statement for the cohomologies with non-trivial local coefficient systems appears to be new. It is interesting to note that the structure of $\mathbf{W}$-cohomology is not immediately apparent in the topological context. The results also imply a complete description of the Poincar\'e polynomials of the free part and of the 2-torsion part of complete flag manifolds, cf.~Theorem~\ref{thm:completeflag}:

\begin{theorem}
  Let ${\rm P}_2(X,t)$, ${\rm P}_0(X,t)$ and ${\rm P}_{\op{Tor}}(X,t)$ denote the Poincar\'e polynomials of $\op{H}^*(X;\mathbb F_2)$, $\op{H}^*(X;\mathbb \Q)$ and $\op{Tor}\bigl(\op{H}^*(X;\mathbb Z)\bigr)$, respectively. Denoting by $\Fl(N)$ the variety of complete flags in $\mathbb{A}^N$, we have
  \begin{eqnarray*}
    {\rm P}_2\bigl(\Fl(N),t\bigr)&=&\frac{\prod_{j=1}^{N} (1-t^j)}{(1-t)^{N}},\\
    {\rm P}_0\bigl(\Fl(N),t\bigr)&=&(1+t^{n})\cdot \prod_{i=1}^{\lfloor\frac{N}{2}\rfloor-1} (1+t^{4i-1}),\\
    {\rm P}_{\op{Tor}}\bigl(\Fl(N),t\bigr)&=&\frac{t}{t+1}(P_2-P_0),
    \end{eqnarray*} 
  where $n=N-1$ if $N$ is even and $n=2N-3$ if $N$ is odd.
\end{theorem}

\subsection{Application: refined enumerative geometry}
A classical problem in enumerative geometry is counting the number of lines on a smooth cubic surface. This problem can be formulated over the complex field as computation of the top Chern class of $\op{Sym}^3(S^\vee)$ over the Grassmannian $\Gr_2(4)$, whose integral is 27, the number of lines. Over the reals, the number of lines depends on the surface. However, a similar characteristic class computation can be formulated in the real case, namely computing the integral of the Euler class of the same bundle over the real Grassmannian, which is 3. This number can be interpreted as a signed count of the lines as shown by \cite{finashinkharlamov1}, \cite{okonekteleman}. These results can be extended to the arithmetic setting, as was demonstrated recently by \cite{kasswickelgren}.

The above problems can be formulated on Grassmannians, since they ask about numbers of subspaces incident to hypersurfaces. Extending this, we can consider the following flag variety analogue of the above problem. 

\begin{question}
  Given a partition $\D=(d_1,\dots,d_m)$ of $N=\sum_{i=1}^md_i$ and general hypersurfaces $H_1,\dots,H_m$ in $\mathbb{P}^{N-1}$, how many flags $V_\bullet$ in $\mathbb{A}^{N}$ of type $\D$ exist, so that $\mathbb{P}(V_1)\se H_1$, $\mathbb{P}(V_2)\se H_2\stb \mathbb{P}(V_m)\se H_m$?
\end{question}

In order to obtain a lower bound, one has to compute the Euler class of an even rank orientable bundle over an orientable flag variety. We discovered that finding such a bundle is less immediate than in the case of Grassmannians: the orientability conditions impose severe numerical conditions. In particular, the smallest example we found was the following one:
\begin{question}
  Given a degree $3$ hypersurface $H_1$ and a degree $5$ hypersurface $H_2$ in $\mathbb{P}^{17}$, how many two-step flags $V_1\se V_2\se \mathbb{A}^{18}$ with $\dim V_1=2, \dim V_2=4$ exist, s.t.\ $\mathbb{P}(V_1)\subseteq H_1$ and $\mathbb{P}(V_2)\subseteq H_2$?
\end{question}
Using our computation of the $\mathbf{W}$-cohomology of flag varieties, we show that there are at least $24681637575$ such flags, cf.\ Theorem~\ref{thm:enumerativeexample}.

\subsection{Structure of the paper} In Section~\ref{sec:notation} we collect some of the notation we use in the paper. The first part of the paper is used to set up the computational tools used to compute the Witt-sheaf cohomology of flag manifolds.  An overview of the relevant standard basics of Chow--Witt theory and related cohomology theories is provided in Section~\ref{sec:prelims}, with a particular focus on various compatibilities and the Gysin sequence for sphere bundles. Then Section~\ref{sec:kuenneth} establishes a K\"unneth formula and Leray--Hirsch theorem for $\mathbf{W}$-cohomology of cellular varieties. In the second half of the paper, we describe the characteristic class picture for Witt-sheaf cohomology of flag varieties. In Section~\ref{sec:flags} we recall some classical information on geometry and Chow rings of flag varieties and introduce some notation. Then we turn to the description of the ring structure of $\mathbf{W}$-cohomology. In Section~\ref{sec:ann-euler} we provide a detailed study of the algebraic structure of the cohomology rings, establishing a description of annihilators of Euler classes necessary for the subsequent computations. In Section~\ref{sec:maximal} we deal with the maximal rank case using the Leray--Hirsch theorem, and in Section~\ref{sec:sadykov} we provide a spherical bundle argument extending the results to the general case.  The final two sections provide applications of our results on Witt-sheaf cohomology of flag varieties: In Section~\ref{sec:chowwittsingular} we deduce structure results for Chow--Witt rings as well as consequences for the singular cohomology of flag manifolds. We conclude the paper by sketching an enumerative application in Section~\ref{sec:enumerative}.

\subsection{Acknowledgements} We would like to thank the anonymous referee at J.~Topol. for their careful reading of the paper and many helpful comments which improved the presentation of the paper. MW would like to thank Jan Hennig for comments on a preliminary version of Section~\ref{sec:kuenneth}. \'AKM would like to thank L\'aszl\'o M.\ Feh\'er for several discussions concerning the cohomology of real flag manifolds. 

\section{Notation and conventions}\label{sec:notation}

\subsection{Notation}

For convenience, we collect some of the common notation used in this paper.

\begin{itemize}
\item $\D$ denotes a sequence of nonnegative integers $(d_1\stb d_m)$, i.e., a partition of $N=\sum_{i=1}^md_i$, 
\item $\Fl(\D)$ denotes a flag variety for the partition $\D$, cf.\ Section~\ref{sec:flags}, 
\item $\EE, \SS,\QQ, \DD$ denote vector bundles, $\L$ usually line bundles, $\mathscr{O}$ the trivial line bundle,
\item for two vector bundles $\EE\to X$ and $\FF\to Y$, $\EE\boxtimes \FF:=\pi_1^*\EE\otimes \pi_2^*\FF$ denotes their exterior product over $X\times Y$,
\item for $X$ a smooth scheme, $\op{CH}^*(X)$ and $\op{Ch}^*(X)$ denote the Chow rings of $X$ with integral and mod 2 coefficients, respectively,
\item $\widetilde{\op{CH}}^*(X, \L)$ denotes the Chow--Witt groups of $X$ with twist $\L/X$,
\item $\op{H}^*\bigl(X,\mathbf{I}^q(\L)\bigr)$, $\op{H}^*\bigl(X,\mathbf{W}(\L)\bigr)$ the $\mathbf{I}^q$ and $\mathbf{W}$-cohomology of $X$ with twist $\L/X$,
\item in certain proofs to simplify notation, for a vector bundle $\EE$ we will denote, cf.~ \eqref{eq:notation}:
  \[
  \op{H}_{\EE}^*(X):=\op{H}^*\bigl(X;\mathbf{W}(\det \EE)\bigr),\quad
  \op{H}_\oplus^*(X):= \bigoplus_{\L}\op{H}_{\L}^\bullet(X),\quad
  \op{H}^*(X):=\op{H}_{\mathscr{O}}^*(X).
  \]
\end{itemize}

Twisting is always with line bundles. In some situations, like for the normal bundle $\mathscr{N}$ of a closed immersion, it is convenient to write ${\bf W}(\mathscr{N})$ for ${\bf W}(\det\mathscr{N})$.

\subsection{Conventions}

Throughout the paper, we consider schemes over a perfect base field $F$ of characteristic $\neq 2$. We always consider finite type separated schemes over $F$ which we will call varieties. Most of the time, varieties will be smooth. Our indexing of Pontryagin classes used here follows the unconventional conventions set out in \cite{chowwitt,realgrassmannian} and includes the odd Pontryagin classes, without reindexing or additional signs, cf.~Remark~\ref{rem:pontryagin}.

\subsection{Remark on perfect base fields}
Throughout the paper, we will work over a perfect base field as required by a couple of the computational tools from motivic homotopy or generic smoothness results.

For the end result, this restriction is, however, not necessary. By results of Elmanto and Khan \cite{elmanto:khan:perfect}, for a field $F$ of characteristic $p>0$ with perfect closure $F^{\rm perf}$, the pullback functor $\mathcal{SH}(F)[p^{-1}]\to \mathcal{SH}(F^{\rm perf})[p^{-1}]$ on stable homotopy categories is an equivalence. Note that Chow--Witt groups as well as $\mathbf{I}$- and $\mathbf{W}$-cohomology are all representable in the stable homotopy category over a perfect base field. In particular, there is up to stable weak equivalence a unique spectrum over $F$ which base-changes to the spectrum representing the respective cohomology over $F^{\rm perf}$. With this as definition, the relevant cohomology of the flag variety over $F$ is isomorphic to the cohomology of the flag variety over $F^{\rm perf}$ after inverting $p$. In particular, all results concerning 2-torsion phenomena in the present paper hold without the restriction to perfect base fields. 

We would like to point out that our assumption on perfect base fields is only due our use of motivic homotopy tools for the computation of Witt-sheaf cohomology. In a previous version of our argument for the Leray--Hirsch theorem in Proposition~\ref{prop:leray-hirsch}, we used some generic smoothness argument, which we can now replace by an argument which doesn't use perfect base field assumptions any more. In conclusion, while the results in Section~\ref{sec:kuenneth} are stated with the assumption of perfect base fields, the reader can remove these assumptions if they believe that the assumption of perfectness is unnecessary for the motivic homotopy tools used.

\section{Recollection on Chow--Witt groups and \texorpdfstring{$\mathbf{W}$}{W}-cohomology}
\label{sec:prelims}

In this section, we provide a basic recollection on cohomology theories related to Chow--Witt groups, such as $\mathbf{I}$- or $\mathbf{W}$-cohomology. Similar such recollections with varying focus can be found in \cite{chowwitt} or \cite{realgrassmannian}. 
\subsection{Key diagram and standard reductions} 

In this paper we consider Nisnevich cohomology of various sheaves on smooth varieties. The relevant sheaves appearing are the Milnor--Witt K-theory sheaves $\mathbf{K}^{\op{MW}}_n$, the sheaf of Witt rings $\mathbf{W}$ and the sheaves of powers of fundamental ideals $\mathbf{I}^n\subseteq \mathbf{W}$. 

Recall that the Chow--Witt groups can be defined as sheaf cohomology groups (in the Nisnevich topology) of the Milnor--Witt K-theory sheaves $\mathbf{K}^{\op{MW}}_n$ which sit in a pullback square
\[
\xymatrix{
\mathbf{K}^{\op{MW}}_n\ar[r] \ar[d] & \mathbf{K}^{\op{M}}_n \ar[d] \\
\mathbf{I}^n\ar[r] & \mathbf{K}^{\op{M}}_n/2.
}
\]
Taking cohomology induces the following key diagram, which appeared in \cite{chowwitt} and \cite{realgrassmannian} in its untwisted and twisted versions.
\[
  \label{keydiagram}
\xymatrix{
&\op{CH}^n(X)\ar[r]^= \ar[d] & \op{CH}^n(X)\ar[d]^2 \\
\op{H}^n\bigl(X,\mathbf{I}^{n+1}(\L)\bigr)\ar[r]\ar[d]_=& \widetilde{\op{CH}}^n(X,\L)\ar[r]\ar[d]&\op{CH}^n(X)\ar[r]^(.4){\partial_{\L}}\ar[d]^{\bmod 2}& \op{H}^{n+1}\bigl(X,\mathbf{I}^{n+1}(\L)\bigr)\ar[d]^=\\
\op{H}^n\bigl(X,\mathbf{I}^{n+1}(\L)\bigr)\ar[r]_\eta& \op{H}^{n}\bigl(X,\mathbf{I}^{n}(\L)\bigr)\ar[r]_\rho\ar[d]& \op{Ch}^n(X)\ar[r]^(.4){\beta_{\L}}\ar[rd]_{\op{Sq}^2_{\L}}\ar[d]& \op{H}^{n+1}\bigl(X,\mathbf{I}^{n+1}(\L)\bigr)\ar[d]^\rho \\
&0\ar[r]&0&\op{Ch}^{n+1}(X)
}
\]
In this diagram, the two middle rows and the two middle columns are exact, they are pieces of the long exact sequences associated to short exact sequences of coefficient sheaves from the pullback presentation of Milnor--Witt K-theory. The lower horizontal sequence is the B\"ar sequence which is an analogue of the Bockstein sequence in singular cohomology. The commutativity of the lower-right triangle was established for $\L=\mathscr{O}$ by Totaro and extended to the twisted case by Asok and Fasel \cite{AsokFaselEuler}. For more details, cf.\ \cite{chowwitt} or \cite{realgrassmannian}. 

The key point relevant for computations of Chow--Witt rings is that for a perfect field $F$ of characteristic unequal to 2 and a smooth $F$-scheme $X$ the canonical map 
\[
\widetilde{\op{CH}}^*(X;\L)\to \op{H}^*\bigl(X,\mathbf{I}^*(\L)\bigr)\times_{\op{Ch}^*(X)} \ker\partial_{\L}
\]
is an isomorphism if $\op{CH}^*(X)$ has no non-trivial 2-torsion, cf.\ \cite[Proposition 2.11]{chowwitt}. Under this assumption, the computation of the total Chow--Witt ring $\bigoplus_{q,\L}\widetilde{\op{CH}}^q(X,\L)$ is reduced to the knowledge of the Chow ring $\op{CH}^*(X)$ together with the (twisted) Steenrod square $\op{Sq}^2_{\L}$ on $\op{Ch}^*(X)$, and the knowledge of the $\mathbf{I}$-cohomology ring $\bigoplus_{q,\L}\op{H}^q\bigl(X,\mathbf{I}^q(\L)\bigr)$. As is well-known, the projective homogeneous varieties for split reductive groups are cellular, hence their Chow rings are torsion-free, and we get the following result: 

\begin{proposition}
  \label{prop:fiberprod}
  Let $F$ be a field of characteristic $\neq 2$, and let $G/P$ be a projective homogeneous variety for a split reductive group $G$. For any line bundle $\L$ on $G/P$, there is a cartesian square
\[
\xymatrix{
&\widetilde{\op{CH}}^q(G/P,\L) \ar[r] \ar[d] 
& \ker \partial_{\L}\subseteq \op{CH}^q(G/P)\ar[d] &
\\
&\op{H}^q\bigl(G/P,\mathbf{I}^q(\L)\bigr) \ar[r] &
\op{Ch}^q(G/P)&\phantom{a}\hspace{-2 cm},
}
\]
where $\op{Ch}:=\op{CH}/2$ denotes mod 2 Chow groups and the morphism 
\[
\partial_{\L}\colon \op{CH}^q(G/P) \to \op{Ch}^q(G/P)\to \op{H}^{q+1}\bigl(G/P,\mathbf{I}^{q+1}(\L)\bigr)
\]
is the twisted integral Bockstein operation.

If we consider the total Chow--Witt ring $\bigoplus_{q,\L}\widetilde{\op{CH}}^q(G/P,\L)$, the above is a fiber product diagram in the category of $\langle-1\rangle$-graded-commutative $\op{GW}(F)$-algebras. In particular, the multiplicative structure of the  total Chow--Witt ring is known once we know the multiplication on the Chow ring and the total $\mathbf{I}$-cohomology ring.
\end{proposition}

\subsection{$\mathbf{W}$-cohomology and the $\op{Im}\beta$-$\mathbf{W}$-decomposition}

In sufficiently nice instances, the portion 
\[
\op{Ch}^{n-1}(X)\cong \op{H}^{n-1}(X,\mathbf{K}^{\op{M}}_{n-1}/2)\xrightarrow{\beta_{\L}} \op{H}^n\bigl(X,\mathbf{I}^n(\L)\bigr)\to \op{H}^n\bigl(X,\mathbf{W}(\L)\bigr)\to 0
\]
of the B\"ar sequence will split, which leads to the following result concerning splitting of $\mathbf{I}$-cohomology, cf.~\cite[Lemma 2.3]{realgrassmannian}.

\begin{lemma}
  \label{lem:wsplit}
  Let $X$ be a smooth scheme over a field $F$ of characteristic $\neq 2$, and let $\L$ be a line bundle on $X$. If $\op{H}^n\bigl(X,\mathbf{W}(\L)\bigr)$ is free as a $\op{W}(F)$-module, then we have a splitting
  \[
  \op{H}^n\bigl(X,\mathbf{I}^n(\L)\bigr)\cong \op{Im}\beta_{\L}\oplus \op{H}^n\bigl(X,\mathbf{W}(\L)\bigr).
  \]
  In this case, the reduction morphism $\rho\colon \op{H}^n\bigl(X,\mathbf{I}^n(\L)\bigr)\to \op{Ch}^n(X)$ is injective on the image of $\beta_{\L}$. 
\end{lemma}

\begin{remark}
  The freeness of $\mathbf{W}$-cohomology in this lemma  plays the role of an algebraic replacement of the classical statement that ``all torsion in the cohomology of the Grassmannians is 2-torsion'', as is true e.g.\ for maximal rank flag varieties of type A, cf.\ \cite[Theorem 6.1]{matszangosz}. The above lemma will allow us to give, via an algebraic approach, a generalization of the torsion result in loc.~cit. 
\end{remark}

When applicable, the above lemma reduces the computation of $\mathbf{I}$-cohomology to the computation of $\mathbf{W}$-cohomology and a description of the image of $\beta$. For the latter, we only need a description of the action of $\op{Sq}^2$ on the mod 2 Chow ring, cf.~e.g.~\cite{lenart}, \cite{duan:zhao}, \cite[Theorem 4.2]{schubert}, \cite[Proposition 6.3]{matszangosz}.  %, cf.~Section~\ref{sec:chow-sq2}.

\subsection{$\mathbf{W}$-cohomology and its formal properties}

We provide a short recollection on $\mathbf{W}$-cohomology and the formal properties that we will use in the course of our computations. All other cohomology theories we will consider have similar formal properties which are detailed e.g.\ in \cite{chowwitt} or \cite{realgrassmannian}. The reader should consult these references (or the original literature cited in these references) for further details as well as facts on (twisted) cohomology operations like $\beta_{\L}$ or $\op{Sq}^2_{\L}$. The following is a quick recollection of definitions and basic properties, mostly from the detailed source \cite{4real}. 

For a smooth scheme $X$ over a field $F$ of characteristic $\neq 2$, \emph{the sheaf $\mathbf{W}$ of unramified Witt groups} is the Zariski sheafification of the presheaf $U\mapsto{\rm W}(U)$ mapping each open $U\subseteq X$ to the Witt group ${\rm W}(U)$ of quadratic modules over $U$. For a line bundle $\L$ on $X$, we can consider the corresponding twisted sheaf $\mathbf{W}(\L)$ which is defined as the sheafification of the presheaf
\[
U\mapsto {\rm W}(U)\otimes_{\mathbb{Z}[\mathscr{O}_X^\times(U)]}\mathbb{Z}[\L(U)^\times]. 
\]
The tensor product is over the group ring of invertible functions on $U$, the action of invertible functions $f\in\mathscr{O}_X^\times(U)$ on the Witt ring is via multiplication by rank one forms $\langle f\rangle$, and $\mathbb{Z}[\L(U)^\times]$ is the free abelian group of invertible sections of $\L$ over $U$. Alternatively, the sheaf $\mathbf{W}(\L)$ can be described as the sheafification of the presheaf $U\mapsto{\rm W}(U,\L)$ of Witt groups of quadratic modules for the duality $\mathscr{H}om_{\mathscr{O}_U}(-,\L|_U)$, cf.\ the discussion in \cite[Section 2.B]{4real}.

\emph{The (twisted) Witt-sheaf cohomology} ${\rm H}^i\bigl(X,\mathbf{W}(\L)\bigr)$ is then the sheaf cohomology (for Zariski- or equivalently Nisnevich-topology) of $\mathbf{W}(\L)$ on $X$. There are particularly nice and useful Gersten-type resolutions for $\mathbf{W}(\L)$, giving rise to the so-called \emph{Rost--Schmid-complexes} computing the Witt-sheaf cohomology: for a smooth scheme $X$ of Krull dimension $n$ and a line bundle $\L$ over $X$, the complex is concentrated in degrees $[0,n]$, with the degree $i$ part given by
\[
C_{\rm RS}^i\bigl(X,\mathbf{W}(\L)\bigr)=\bigoplus_{x\in X^{(i)}}{\rm W}\bigl(\kappa(x),\omega_{\kappa(x)/F}\otimes\L\bigr).
\]
The direct sum runs over all codimension $i$ points $x$ of $X$, and the coefficients are Witt groups of the residue field $\kappa(x)$ twisted by the tensor product of the relative canonical bundle of $\kappa(x)/F$ and the global line bundle $\L$. For a discretely valued field $L$ with valuation ring $\mathcal{O}$, residue field $K$ and chosen uniformizer $\pi$, there is a residue morphism $\partial^\pi\colon{\rm W}(L)\to {\rm W}(K)$ of Witt groups essentially induced by mapping
\[
\langle\pi,u_1,\dots,u_n\rangle\mapsto \langle \overline{u}_1,\dots,\overline{u}_n\rangle \qquad \textrm{ and }\qquad\langle u_1,\dots,u_n\rangle\mapsto 0
\]
for units $u_i\in\mathcal{O}^\times$. The differential
\[
\bigoplus_{x\in X^{(i)}}{\rm W}\bigl(\kappa(x),\omega_{\kappa(x)/F}\otimes\L\bigr)\xrightarrow{\sum\partial} \bigoplus_{y\in X^{(i+1)}}{\rm W}\bigl(\kappa(y),\omega_{\kappa(y)/F}\otimes\L\bigr)
\]
for the Rost--Schmid-complex is induced from these residue morphisms with the twisting making the resulting morphism independent of the choice of uniformizer. A convenient summary of the Rost--Schmid complex for $\mathbf{W}$-cohomology with more details concerning the differential can be found in \cite[Section A.A]{4real}. The Rost--Schmid-complex computes Witt-sheaf cohomology in the sense that for any natural number $q$, any smooth scheme $X$ and any line bundle $\L$ on $X$ we have natural isomorphisms
\[
{\rm H}^q\bigl(X,\mathbf{W}(\L)\bigr)\cong {\rm H}^q\bigl(C^\ast_{\rm RS}\bigl(X,\mathbf{W}(\L)\bigr)\bigr).
\]

There is also a version of $\mathbf{W}$-cohomology with supports in a closed subscheme $Z\subseteq X$, denoted by ${\rm H}^\ast_Z\bigl(X,\mathbf{W}(\L)\bigr)$. This can be defined either via deriving the global sections with supports, or on the level of Rost--Schmid complexes by only considering cycles which have non-zero coefficients for points contained in the support, cf.~\cite[Sections 2.A.2 and 2.C.2]{4real} for definitions and comparison results.

For a morphism $f\colon X\to Y$ of smooth schemes with closed subsets $V\subseteq X$ and $W\subseteq Y$ such that $f^{-1}(W)\subseteq V$ and a line bundle $\L$ over $Y$, there are functorial pullback morphisms in $\mathbf{W}$-cohomology with supports
\[
{\rm H}^\ast_W\bigl(Y,\mathbf{W}(\L)\bigr)\to {\rm H}^\ast_V\bigl(X,\mathbf{W}(f^\ast\L)\bigr),
\]
which can be defined as restriction morphisms in sheaf cohomology or via Rost--Schmid complexes. Note that the definition of pullback via Rost--Schmid-complexes is straightforward for smooth morphisms but involves intersection with Cartier divisors and deformation to the normal cone for regular embeddings. For definitions and comparison results, we again refer to \cite[Section 2.C.3]{4real}. 

For a closed immersion $\iota\colon Z\hookrightarrow X$ of codimension $c$ of smooth schemes with normal bundle $\mathscr{N}$, there are pushforward maps
\[
\iota_\ast\colon {\rm H}^q\bigl(Z,\mathbf{W}(\det\mathscr{N}\otimes \iota^\ast\L)\bigr)\to {\rm H}^{q+c}_Z\bigl(X,\mathbf{W}(\L)\bigr)\to {\rm H}^{q+c}\bigl(X,\mathbf{W}(\L)\bigr)
\]
obtained by composing a Thom isomorphism for the normal bundle, an isomorphism obtained from deformation to the normal cone and the forgetting-of-support. Again, this can be defined explicitly via a d\'evissage isomorphism on the level of Rost--Schmid complexes. For definitions and comparison, cf.\ \cite[Theorem 2.41 and Section 2.C.5]{4real}. 

\subsection{Product structures}

The cup product (or intersection product) on $\mathbf{W}$-cohomology is a product of the form
\[
{\rm H}^i\bigl(X,\mathbf{W}(\L_1)\bigr)\times{\rm H}^j\bigl(X,\mathbf{W}(\L_2)\bigr)\to {\rm H}^{i+j}\bigl(X,\mathbf{W}(\L_1\otimes\L_2)\bigr).
\]
With this product structure, the total $\mathbf{W}$-cohomology ring
\[
\bigoplus_{\L\in\op{Pic}(X)/2, q\in\mathbb{N}}{\rm H}^q\bigl(X,\mathbf{W}(\L)\bigr) 
\]
is a $(-1)$-graded-commutative associative algebra over the Witt ring ${\rm W}(F)$ of the base field, cf.\ e.g.\ \cite[Section 2.C.4]{4real}. There are technical coherence problems that arise because the identifications ${\rm H}^\ast\bigl(X,\mathbf{W}(\L)\bigr)\cong{\rm H}^\ast\bigl(X,\mathbf{W}(\L')\bigr)$ depend on the choice of isomorphisms $\L\cong\L'$. These can be dealt with by the methods of Balmer--Calm\`es in \cite{balmer:calmes}. 

There is also an exterior product structure on $\mathbf{W}$-cohomology which plays a role in the K\"unneth formula and Leray--Hirsch theorem established below in Section~\ref{sec:kuenneth}. For a further discussion of exterior products, cf.\ \cite[Section 3.1]{fasel:chow-witt-lectures} and \cite[Sections 2.A.3 and 2.C.4]{4real}. 

\begin{definition}
  \label{def:ext-prod}
  Let $F$ be a perfect field of characteristic $\neq 2$, let $X$ and $Y$ be two smooth $F$-schemes, and let $\L_X$ and $\L_Y$ be line bundles on $X$ and $Y$, respectively. Then there is an exterior product bundle
  \[
  \L_X\boxtimes\L_Y:=\op{pr}_1^\ast(\L_X)\otimes\op{pr}_2^\ast(\L_Y)
  \]
  on $X\times Y$, where $\op{pr}_1\colon X\times Y\to X$ and $\op{pr}_2\colon X\times Y\to Y$ are the respective projections.

  In the above situation, there is an exterior product map in $\mathbf{W}$-cohomology which has the form
  \[
  {\rm H}^p\bigl(X,\mathbf{W}(\L_X)\bigr)\times{\rm H}^q\bigl(Y,\mathbf{W}(\L_Y)\bigr) \to {\rm H}^{p+q}\bigl(X\times Y,\mathbf{W}(\L_X\boxtimes\L_Y)\bigr).
  \]
  The exterior product can be defined from the cup product by mapping $\alpha\in {\rm H}^p\bigl(X,\mathbf{W}(\L_X)\bigr)$ and $\beta\in {\rm H}^q\bigl(Y,\mathbf{W}(\L_Y)\bigr)$ to 
  \[
  \alpha\boxtimes\beta=\op{pr}_1^\ast(\alpha)\cup\op{pr}_2^\ast(\beta)\in {\rm H}^{p+q}\bigl(X\times Y,\mathbf{W}(\L_X\boxtimes\L_Y)\bigr).
  \]
  Conversely, the cup product can be defined in terms of the exterior product by mapping $\alpha\in {\rm H}^p\bigl(X,\mathbf{W}(\L_1)\bigr)$ and $\beta\in {\rm H}^q\bigl(X,\mathbf{W}(\L_2)\bigr)$ to 
  \[
  \alpha\cup\beta=\Delta^\ast(\alpha\boxtimes\beta)\in {\rm H}^{p+q}\bigl(X,\mathbf{W}(\L_1\otimes\L_2)\bigr),
  \]
  i.e., it is obtained from the exterior product on $X\times X$ by restriction along the diagonal $\Delta\colon X\to X\times X$. 
\end{definition}

The exterior product can alternatively be defined on the level of Gersten complexes: for a point $x\in X$ and a point $y\in Y$ the tensor product $\kappa(x)\otimes_F\kappa(y)$ of residue fields decomposes as direct sum of residue fields of finitely many points $u\in X\times Y$, and the exterior product can then be defined using the product structure on Witt groups (coming from tensor product of quadratic forms). This is explained in more detail in  \cite[Section~3.1]{fasel:chow-witt-lectures}. 

For a morphism $f\colon X\to Y$ of smooth schemes, the pullback maps $f^\ast\colon {\rm H}^\ast\bigl(Y,\mathbf{W}(\L)\bigr)\to {\rm H}^\ast\bigl(X,\mathbf{W}\bigl(f^\ast(\L)\bigr)\bigr)$ for line bundles $\L\in\op{Pic}(Y)/2$ assemble into a homomorphism of $\mathbb{Z}\oplus\op{Pic}/2$-graded ${\rm W}(F)$-algebras. As a particular special case, the total $\mathbf{W}$-cohomology ring is an algebra over the Witt ring ${\rm W}(F)$ of the base field. Moreover, the projection formula \cite[Theorem 3.19]{fasel:chow-witt-lectures} and derivation property \cite[Proposition 3.4]{fasel:chow-witt-lectures} imply that these ${\rm W}(F)$-module structures are compatible with pushforwards and boundary maps in localization sequences.

\subsection{Localization sequence}
As for all cohomology theories arising from motivic homotopy, there is a localization sequence for $\mathbf{W}$-cohomology which takes a simplified form (as compared to the localization sequences for, say, Chow--Witt-groups or $\mathbf{I}$-cohomology). To formulate it, assume that $X$ is a smooth scheme, $Z\subseteq X$ a smooth closed subscheme of pure codimension $c$ with open complement $U=X\setminus Z$, and $\L$ is a line bundle on $X$. Denote the inclusions by $i\colon Z\hookrightarrow X$ and $j\colon U\hookrightarrow X$, and denote by $\mathscr{N}$ the normal bundle for $Z$ in $X$. In this situation, there is the following localization sequence for $\mathbf{W}$-cohomology:
\begin{equation}\label{eq:localization}
\cdots\xrightarrow{j^\ast} \op{H}^{i-1}(U,\mathbf{W}(\L))\xrightarrow{\partial} \op{H}^{i-c}(Z,\mathbf{W}(\L\otimes \det\mathscr{N})) \xrightarrow{i_\ast}\op{H}^{i}(X,\mathbf{W}(\L)) \xrightarrow{j^\ast} \cdots% \op{H}^{i}(U,\mathbf{W}(\L))\to\cdots\ .
\end{equation}

Compared to the localization sequences for Chow--Witt-groups or $\mathbf{I}$-cohomology, this form has the distinct advantage that there are no index shifts in the coefficients and we really get an honest long exact sequence as opposed to only a piece of a long exact sequence containing the ``geometric bidegrees''. This way, computations of $\mathbf{W}$-cohomology can follow their classical topology counterparts much more closely than is possible for $\mathbf{I}$-cohomology. 

From the motivic homotopy point of view, the localization sequences for representable cohomology theories arise from a cofiber sequence $U\to X\to {\rm Th}(\mathscr{N}_Z)$, where $X$ is a smooth scheme, $\iota\colon Z\hookrightarrow X$ a smooth closed subscheme of $X$ with open complement $U\subseteq X$, and ${\rm Th}(\mathscr{N}_Z)$ the Thom space of the normal bundle of the inclusion $\iota$. This implies that the localization sequences are compatible with all relevant structures of representable cohomology theories such as, in particular, pullbacks and exterior products. For our computations, we will mostly use the $\mathbf{W}$-cohomology, i.e., Nisnevich cohomology with coefficients in the sheaf $\mathbf{W}$ of unramified Witt rings, which is representable in motivic homotopy over a base field of characteristic $\neq 2$ (since it satisfies Nisnevich descent and homotopy invariance). 

For later use we formulate the compatibility of the localization sequence  with pullbacks. Compatibility with exterior product will be discussed in Section~\ref{sec:kuenneth} where it will be relevant for the K\"unneth formula for $\mathbf{W}$-cohomology of cellular varieties.

\begin{lemma}
  \label{lem:compat-locseq-pullback}
  Let $F$ be a perfect field of characteristic $\neq 2$ and let $f\colon Y\to X$ be a morphism between smooth schemes. Let $Z\subseteq X$ be a smooth closed subscheme of codimension $c$ with normal bundle $\mathscr{N}_Z$ and complement $U=X\setminus Z$, and denote by $Y_Z=f^{-1}(Z)$ and $Y_U=f^{-1}(U)$ the respective preimages. Assume that  $Y_Z$ is smooth and that $f\colon Y\to X$ is transverse to $Z$ in the sense of \cite[Theorem 2.12]{AsokFaselEuler}, i.e., that there is an isomorphism $f^*\mathscr{N}_Z\iso \mathscr{N}(Y_Z\to Y)$ of normal bundles.\footnote{Note that this is satisfied for example when $f$ is smooth.} Then for any line bundle $\L$ on $X$ there is a ladder of long exact localization sequences
  \begin{center}
  \scalebox{0.9}{
  \xymatrix{
    \cdots\ar[r] & {\rm H}^{q-c}\bigl(Z,\mathbf{W}(\L\otimes\mathscr{N}_Z)\bigr) \ar[r] \ar[d]& {\rm H}^q\bigl(X,\mathbf{W}(\L)\bigr)\ar[r] \ar[d] & {\rm H}^q\bigl(U,\mathbf{W}(\L)\bigr) \ar[r] \ar[d] & \cdots\ \phantom{.}\\
    \cdots\ar[r] & {\rm H}^{q-c}\bigl(Y_Z,\mathbf{W}\bigl(f^\ast(\L\otimes\mathscr{N}_Z)\bigr)\bigr) \ar[r]& {\rm H}^q\bigl(Y,\mathbf{W}(f^\ast\L)\bigr)\ar[r] & {\rm H}^q\bigl(Y_U,\mathbf{W}(f^\ast\L)\bigr) \ar[r] & \cdots\ .\\
  }
  }
  \end{center}
\end{lemma}

\begin{proof}
  There are various ways to prove this. As mentioned above, one way is via motivic homotopy. For the localization situation on $X$ with closed subscheme $Z$ and open subscheme $U$, we have an associated cofiber sequence $U\to X\to X/(X\setminus Z)$, where the last term can be identified with the Thom space ${\rm Th}(\mathscr{N})$ of the normal bundle for the closed immersion $Z\hookrightarrow X$; that is the homotopy purity theorem of Morel and Voevodsky. The same thing works for $Y$, and the morphism $f$ - being compatible with everything - induces a commutative diagram of cofiber sequences
  \[
  \xymatrix{
    Y_U=f^{-1}(U)\ar[d] \ar[r] & Y \ar[r] \ar[d] & Y/(Y\setminus Y_Z)\ \phantom{.} \ar[d] \\
    U\ar[r] & X \ar[r] & X/(X\setminus Z)\ .
  }
  \]
  The $\mathbf{W}$-cohomology is represented by the Eilenberg--Mac~Lane spectrum $\mathbb{H}\mathbf{W}$ for the homotopy module $\mathbf{W}$. The localization sequence arises from the above cofiber sequence by mapping into the Eilenberg--Mac~Lane spectrum. The above commutative diagram then induces the required ladder of localization sequences. 

  Another possibility is the following. The  $\mathbf{W}$-cohomology localization sequence on $X$ (with open subscheme $U$ and closed subscheme $Z$) can be obtained from the exact sequence
  \[
  0\to i_\ast i^!\mathbf{W}_X\to \mathbf{W}_X\to j_\ast j^\ast\mathbf{W}_X\to 0
  \]
  of sheaves on the small (Zariski or Nisnevich) site of $X$, that follows from \cite[Proposition 2.39]{4real} and the remark on localization sequences after it. There is a similar exact sequence of sheaves on $Y$ yielding the localization sequence for $Y$ with open subscheme $Y_U=f^{-1}(U)$ and $Y_Z=f^{-1}(Z)$. Applying $f_\ast$ to the sequence on $Y$ and using \cite[Lemma 2.7]{4real} yields a commutative diagram of exact sequences
  \[
  \xymatrix{
        0\ar[r] & i_\ast i^!\mathbf{W}_X \ar[r] \ar[d] &  \mathbf{W}_X \ar[r]\ar[d]  & j_\ast j^\ast \mathbf{W}_X\ar[r] \ar[d] & 0 \ \phantom{,}\\
    0\ar[r] & i_\ast i^!f_\ast\mathbf{W}_Y \ar[r] &  f_\ast\mathbf{W}_Y \ar[r] & j_\ast j^\ast f_\ast\mathbf{W}_Y\ar[r] & 0\ ,
    }
  \]
  whose lower horizontal sequence is identified with the exact sequence
  \[
  0\to f_\ast (i_Y)_\ast (i_Y)^!\mathbf{W}_Y\to f_\ast\mathbf{W}_Y\to f_\ast(j_Y)_\ast(j_Y)^\ast\mathbf{W}_Y\to 0.
  \]
  The vertical morphisms are the morphisms of coefficient data yielding the pullback as in \cite[Definition 2.33]{4real}, cf.\ also Proposition 2.45 of loc.~cit. The ladder of exact localization sequences then follows from \cite[Proposition 2.12]{4real}.   
\end{proof}

\subsection{Special case: the Gysin sequence for sphere bundles}
\label{sec:Gysin}

We will use the following shorthand notation for the total $\mathbf{W}$-cohomology ring:
$$\op{H}^*_\oplus(X):=\op \bigoplus_{\mathscr L \in \op{Pic}(X)/2}\op{H}^*\bigl(X,\mathbf{W}(\L)\bigr).$$

Let $\EE$ be a rank $n$ vector bundle with Euler class $e:={\rm e}(\EE)$, and let $$q\colon \EE_0:=\EE\su 0\to X$$ denote its sphere bundle. The localization sequence, cf.~\eqref{eq:localization} on page \pageref{eq:localization}, applied to $X\inj \EE\hookleftarrow \EE_0$ translates to the following \emph{Gysin sequence}:
\begin{equation}\label{eq:gysin}
  \xymatrix{
    \cdots\ar[r]^-{\de}&	\op{H}^{*-n}_{\oplus}(X) \ar[r]^-{(-)\cdot e} &
    \op{H}^{*}_{\oplus}(X) \ar[r]^-{q^*} &
    \op{H}_\oplus^{*}(\EE_0) \ar[r]^-{\de} &
    \op{H}^{* +1-n}_{\oplus}(X)\ar[r]^-{(-)\cdot e} &
    \cdots\ .}
\end{equation}
In this sequence, the map $(-)\cdot e$ is multiplication with the Euler class of $\EE$ and can alternatively be understood as composition of a homotopy invariance isomorphism and the Thom isomorphism with the forgetting of support. The pullback morphism $q^*$ is a homomorphism of the total $\mathbf{W}$-cohomology rings. 

The cohomology of a sphere bundle can often be computed using the following part of the exact sequence
\begin{equation}\label{eq:Gysinkercoker}
  \xymatrix{0 \ar[r] &
    \op{coker} e\ar[r]^-{q^*} &
    \op{H}_\oplus^{*}(\EE_0) \ar[r]^-{\de} &
    \ker e\ar[r] &
    0}.
\end{equation}
Under some further assumptions, we can make some general statements.

\begin{proposition}
  \label{prop:spherebundle}
  Let $\EE$ be a vector bundle over $X$ with Euler class $e={\rm e}(\EE)$ and let $\EE_0$ denote its sphere bundle. 
  \begin{enumerate}
  \item The exact sequence~\eqref{eq:Gysinkercoker} is an exact sequence of right $C$-modules. 
  \item Assume that, with the right $C$-module structure from (1),  $\ker e$ is free of rank 1. Then the sequence~\eqref{eq:Gysinkercoker} splits, i.e., the $\mathbf{W}$-cohomology of $\EE_0$ is
  $$
  {\rm H}^*_\oplus(\EE_0)\cong \op{coker} e \oplus R\cdot \op{coker} e,
  $$
  as $\op{coker}e$-modules, where $R$ is an element satisfying $\de R=x$.
  \end{enumerate}
\end{proposition}

\begin{proof}
(1) We first note that all objects in the exact sequence \eqref{eq:Gysinkercoker}  have a module structure over $C:=\coker e$. For ${\rm H}^*_\oplus(\EE_0)$, the module structure is induced from $q^*$, making $q^*$ a $C$-module homomorphism. For $\ker e$ we note more generally that for a ring $R$ and any $y\in R$, $\ker y$ is a $\coker y$-module: an element $z\in \coker y$ acts by scalar multiplication  with a lift $\bar{z}\in R$ of $z$, and this is well-defined since the difference of two lifts is a multiple of $y$ and thus annihilated by $\ker y$.

  Next we want to show that $\de$ is a $C$-module homomorphism. For this, let $\alpha\in {\rm H}^*_\oplus(X)$ and $\beta\in{\rm H}^*_\oplus(\EE_0)$. The image of $\alpha$ in $\coker e$ will still be denoted by $\alpha$. Then the derivation property of the boundary map $\delta$, cf.\ \cite[Proposition 2.6]{chowwitt} implies
  \[
  \delta\bigl(\beta\cup q^*(\alpha)\bigr)=\delta(\beta)\cup q^*(\alpha) + (-1)^{\deg\alpha} \beta\cup \delta\bigl(q^*(\alpha)\bigr)=\delta(\beta)\cup q^*(\alpha). 
  \]
  The left-hand side is given by the $C$-module structure of ${\rm H}^*_\oplus(\EE_0)$ via $q^*$, and the right-hand side is the module structure of $\ker e$ explained above (since $\alpha$ is the lift of its image in $\coker e$). The first equality is the derivation property, and the second equality follows since $\delta\circ q^*=0$ from the exactness of the Gysin sequence. We conclude that the exact sequence \eqref{eq:Gysinkercoker} is an exact sequence of $C$-modules. 

(2) Since $\ker e=C\bra x\ket$ is a free $C$-module on one generator, the exact sequence~\eqref{eq:Gysinkercoker} is a split exact sequence of $C$-modules as claimed. 
\end{proof}

\begin{remark}
  \begin{enumerate}
  \item 
    In the above proposition, we could alternatively consider left module structures. This, however, requires introducing some signs depending on the degrees of elements, to be compatible with the derivation property of the boundary map in step (1) of the above proof.
  \item In later applications of point (2) above, the statement that $\ker e$ is free of rank one will follow from $\ker e$ being a principal ideal plus a Poincar\'e duality statement.
  \end{enumerate}
\end{remark}

\subsection{Characteristic classes in Witt-sheaf cohomology}

In this section, we provide a brief recollection on characteristic classes in Chow--Witt theory and Witt-sheaf cohomology, including some of the basic properties we will need in our later computations. 

For a smooth $F$-scheme $X$ and a rank $n$ vector bundle $p\colon \EE\to X$ with zero section $s\colon X\to \EE$, the Euler class in Chow--Witt theory is defined as follows, cf.~\cite{fasel:memoir}:
\[
  {\rm e}(\EE):=(p^*)^{-1}s_\ast(\langle 1\rangle)\in\widetilde{\rm CH}^n(X,\det \EE^\vee).
\]
This definition extends to all the other cohomology theories we consider; what is needed is a multiplicative cohomology theory with suitable pushforwards (such as MSL-oriented or MSp-oriented representable cohomology theories in motivic stable homotopy).

An orientation of a vector bundle $\EE$ on $X$ is given by a line bundle $\L$ and an isomorphism $\omega\colon\det \EE\xrightarrow{\cong}\L^{\otimes 2}$. Any orientation, i.e., choice of such isomorphism, induces isomorphisms
\[
\widetilde{\rm CH}^n(X,\det \EE^\vee)\cong\widetilde{\rm CH}^n\bigl(X,(\L^\vee)^{\otimes 2}\bigr)\cong\widetilde{\rm CH}^n(X,\mathscr{O}).
\]
It should be noted that the resulting Euler class in $\widetilde{\rm CH}^n(X)$ depends on the choice of orientation. 

The algebraic Euler classes are compatible with their topological counterparts, in the sense that the real cycle class map ${\rm H}^n\bigl(X,\mathbf{I}^n(\det \EE^\vee)\bigr)\to {\rm H}^n\bigl(X(\mathbb{R}),\mathbb{Z}(\det \EE^\vee)\bigr)$ takes the algebraic Euler classes in $\mathbf{I}$-cohomology to the topological Euler classes in singular cohomology, cf.~\cite[Proposition~6.1]{4real}.

Similar to their topological counterparts, the algebraic Euler classes satisfy the classical basic properties:
\begin{itemize}
\item Euler classes are \emph{functorial} for morphisms of smooth schemes, cf.~\cite[Proposition 3.1]{AsokFaselEuler}: for a morphism $f\colon X\to Y$ and a vector bundle $\EE$ over $Y$, we have
  \[
  f^\ast\bigl({\rm e}(\EE)\bigr)={\rm e}\bigl(f^\ast(\EE)\bigr).
  \]
\item There is a \emph{Whitney sum formula}, cf.~\cite[Lemma 2.8]{chowwitt}: for smooth schemes $X_1,X_2$ and vector bundles $\EE_i\to X_i$, we have ${\rm e}(\EE_1\boxplus \EE_2)={\rm e}(\EE_1)\boxtimes{\rm e}(\EE_2)$.
\item If a vector bundle $\EE/X$ has a \emph{nowhere vanishing section}, i.e., splits off a trivial line bundle, then ${\rm e}(\EE)=0$. For smooth affine schemes this is \cite[Corollaire 13.3.3]{fasel:memoir}. Alternatively, it is a consequence of \cite[Theorem 1]{AsokFaselEuler} which compares the above definition with the obstruction-theoretic definition. 
\end{itemize}

As a final property, the topological Euler class changes sign when \emph{switching the orientation} of a vector bundle. The algebraic situation is a bit more complicated, since the set of orientations of an $F$-vector space is a $F^\times/(F^\times)^2$-torsor. However, the behaviour of Euler classes under change of orientation is sort of the expected one which we detail in the following two lemmas. 

\begin{lemma}
  \label{lem:euler-iso}
  Let $X$ be a smooth $F$-scheme, and let $\phi\colon \EE\cong \EE'$ be an isomorphism of rank~$n$ vector bundles over $X$. Then the induced isomorphism
  \[
  \phi\colon \widetilde{\rm CH}^n(X,\det \EE^\vee)\cong \widetilde{\rm CH}^n\bigl(X,\det(\EE')^\vee\bigr)
  \]
  maps ${\rm e}(\EE)$ to ${\rm e}(\EE')$. If the vector bundles are orientable, and we choose orientations such that the diagram
  \[
  \xymatrix{
    \det \EE\ar[rr]^\phi\ar[rd] &&\det \EE' \ar[ld]\\
    &\L^{\otimes 2}&	%\phantom{.}	\	\qquad,
  }
  \]
  commutes, then the corresponding Euler classes in $\widetilde{\rm CH}^n(X)$ are equal. 
\end{lemma}

\begin{proof}
  We denote the projection maps of $\EE$ and $\EE'$ by $p$ and $p'$, and the zero sections by $s_0$ and $s_0'$,  respectively. We have a transversal square
  \[
  \xymatrix{
   & X \ar[r]^= \ar[d]_{s_0'} & X \ar[d]^{s_0}&\\
  & \EE'\ar[r]^\cong_\phi & \EE&%\phantom{a}\hspace{-1.8 cm } .		
  }
  \]
  By base change for Chow--Witt groups, cf.\ \cite[Theorem 3.18]{fasel:chow-witt-lectures} and the discussion before, we get a commutative square
  \[
  \xymatrix{
    \widetilde{\rm CH}^0(X) \ar[d]_{(s_0')_\ast} \ar[r]^{(s_0)_\ast} & \widetilde{\rm CH}^n\bigl(\EE,p^\ast(\det \EE^\vee)\bigr)    \ar[d]_\cong^{\phi^*}\\  
    \widetilde{\rm CH}^n\bigl(\EE',(p')^\ast\bigl(\det (\EE')^\vee\bigr)\bigr) \ar[r]_\cong & \widetilde{\rm CH}^n\bigl(\EE',(p')^\ast(\det \EE^\vee)\bigr)
  }
  \]
in which the bottom horizontal morphism is induced from the isomorphism $\phi\colon \EE\cong \EE'$ by taking determinant and pulling back along $p'$ (and the isomorphisms $\alpha$ from loc.~cit.~are suppressed).
  Pullback functoriality provides another commutative square in which all maps are isomorphisms
  \[
  \xymatrix{
&    \widetilde{\rm CH}^n\bigl(\EE',(p')^\ast(\det (\EE')^\vee)\bigr) && \widetilde{\rm CH}^n\bigl(\EE,p^\ast(\det \EE^\vee)\bigr) \ar[ll]_{\phi^*} &\\
  &  \widetilde{\rm CH}^n\bigl(X,\det (\EE')^\vee\bigr) \ar[u]^{(p')^*} && \widetilde{\rm CH}^n(X,\det \EE^\vee) \ar[u]_{p^*} \ar[ll]^\phi& %\phantom{a} \hspace{- 2 cm}.
  }
  \]
  The upper horizontal map is actually the composition of the two isomorphisms in the previous diagram, and the lower horizontal map is again induced from the isomorphism $\phi\colon \EE\xrightarrow{\cong}\EE'$. Combining the two commutative diagrams shows that the bottom isomorphism $\phi\colon\widetilde{\rm CH}^n(X,\det \EE^\vee)\to \widetilde{\rm CH}^n\bigl(X,\det (\EE')^\vee\bigr)$ induced by $\phi$ maps the Euler class ${\rm e}(\EE)$ to the Euler class ${\rm e}(\EE')$.

  In the orientable case, the compatible choices of orientations lead to a commutative triangle of isomorphisms
  \[
  \xymatrix{
  &  \widetilde{\rm CH}^n(X,\det \EE^\vee)\ar[rr]^\phi\ar[rd] &&\widetilde{\rm CH}^n\bigl(X,\det (\EE')^\vee\bigr) \ar[ld]&\\
  &  &\widetilde{\rm CH}^n(X,\mathscr{O})&&%\phantom{a}\hspace{-2 cm},
  }
  \]
  which proves the claim.
\end{proof}

\begin{lemma}
  \label{lem:change-orientation}
  Let $X$ be a smooth scheme over $F$ and let $p\colon \EE\to X$ be an orientable vector bundle of rank $n$, equipped with two orientations $\rho_1,\rho_2\colon\det \EE\xrightarrow{\cong} \mathscr{O}_X$. Then the composition $\rho_1\circ\rho_2^{-1}\colon \mathscr{O}_X\to\mathscr{O}_X$ is an isomorphism given by multiplication with a global unit $u\in{\rm H}^0(X,\mathscr{O}_X^\times)$. The corresponding isomorphism
  \[
  \widetilde{\rm CH}^n(X)\xrightarrow{\rho_2^{-1}}\widetilde{\rm CH}^n(X,\det \EE)\xrightarrow{\rho_1}\widetilde{\rm CH}^n(X)
  \]
  of Chow--Witt groups is given by multiplication with the class $\langle u\rangle\in\widetilde{\rm CH}^0(X)$.
\end{lemma}

\begin{proof}
  This can be seen directly from the Rost--Schmid complex. The map will change the coefficients of cycles in ${\rm W}(\kappa(x),\omega_{\kappa(x)/F})$ by multiplication with $\langle u\rangle$. 
\end{proof}

\begin{remark}
  Note that for the most commonly appearing unit, we have $\langle-1\rangle=-1$ in ${\rm W}(\mathbb{Z})$. 
\end{remark}

In the later computations, we will also need to relate Euler classes of vector bundles and their duals. The corresponding formula can be found in \cite[Theorem 7.1]{levine:aspects}. To state the formula, we note that changing the twist by the square of a line bundle induces isomorphisms. In particular, the isomorphism $\det \EE^\vee\cong \det \EE\otimes\det (\EE^\vee)^{\otimes 2}$ induces an isomorphism of Chow--Witt groups: 
  \[
  \psi\colon \widetilde{\rm CH}^n(X,\det \EE)\xrightarrow{\cong} \widetilde{\rm CH}^n(X,\det \EE^\vee).
  \]

\begin{proposition}
  \label{prop:euler-dual}
  Let $X$ be a smooth quasi-projective $F$-scheme, and let $\EE$ be a rank $n$ vector bundle over $X$. Then
  \[
  \psi\bigl({\rm e}(\EE^\vee)\bigr)=(-1)^n{\rm e}(\EE).
  \]
\end{proposition}

Starting with the Euler class, one can define characteristic classes, the Borel classes, for symplectic bundles by stabilization. Pontryagin classes for vector bundles are then obtained as Borel classes of the symplectification of a vector bundle, cf.~\cite{chowwitt}. This is (almost) the algebraic analogue of the definition of topological Pontryagin classes of real vector bundles as Chern classes of the complexified bundles. As a result, for a vector bundle $\EE$  on a smooth $F$-scheme $X$, there are Pontryagin classes
\[
{\rm p}_i(\EE)\in\widetilde{\rm CH}^{2i}(X),
\]
for $i=1,\dots,\op{rk}\EE$. These satisfy a Whitney sum formula and are compatible with stabilization, cf.~\cite[Theorem 4.10]{chowwitt} or \cite[Proposition 3.28]{realgrassmannian}. From the computations in \cite{realgrassmannian}, it follows that the Witt-sheaf cohomology ring of ${\rm B}{\rm GL}_n$ is generated by even Pontryagin classes, and the Euler class in case $n$ is even, cf.~\cite[Theorem 1.1, or Theorem 6.4]{realgrassmannian}.

\begin{remark}
  \label{rem:pontryagin}
  The definition used here is the one used in \cite{chowwitt} and \cite{realgrassmannian} and it differs from the classical topological definition of Pontryagin classes in that it doesn't discard the odd Pontryagin classes by reindexing and doesn't introduce a sign. This makes several formulas such as the Whitney sum formula more readable.

  Because of this, we should include the explicit reminder that, in $\mathbf{W}$-cohomology, the odd Pontryagin classes are actually trivial because they are in the image of the Bockstein map, cf.~\cite{realgrassmannian}, in particular Theorem 3.23.
\end{remark}

\section{K\"unneth formula and Leray--Hirsch theorems for \texorpdfstring{$\mathbf{W}$}{W}-cohomology}
\label{sec:kuenneth}

In this section, we will set up the K\"unneth formula and deduce a Leray--Hirsch-type theorem for $\mathbf{W}$-cohomology, following the arguments for higher Chow groups in \cite{krishna}. This will allow to extend the formulas from \cite{realgrassmannian,schubert} to describe the $\mathbf{W}$-cohomology of Grassmannian bundles. These formulas will be used in the computation of $\mathbf{W}$-cohomology of flag varieties. Moreover, once we have formulas for ${\bf W}$-cohomology of flag varieties, we can use the Leray--Hirsch theorem to easily produce to descriptions of the $\mathbf{W}$-cohomology rings of \emph{flag bundles} (whenever the cohomology of the relevant flag variety is generated by characteristic classes). This has possible relevance to refined enumerative geometry of degeneracy loci. 

\subsection{$\mathbf{W}$-cohomology of cellular spaces}
Both the K\"unneth formula and the Leray--Hirsch theorem involve \emph{cellular spaces}:
\begin{definition}
	\label{def:cellular}
	Let $F$ be a field of characteristic $\neq 2$ and let $X$ be a smooth variety over $F$. We call $X$ \emph{cellular} if there exists a filtration
	\[
	\emptyset=X_{n+1}\subset X_n\subset X_{n-1}\subset\cdots\subset X_{1}\subset X_0=X
	\]
	of $X$ by closed subschemes $X_i$ of codimension $i$ such that for each $0\leq i\leq n$ we have $X_i\setminus X_{i+1}=\bigsqcup_{c_i}\mathbb{A}^{n-i}$, where $c_i$ denotes the number of cells of \emph{codimension} $i$.
\end{definition}

\begin{lemma}
  \label{lem:cell-vanish}
  In the above setup, $\op{H}^q\bigl(X\setminus X_i,\mathbf{W}(\L)\bigr)=0$ for $q\geq i$.
\end{lemma}

\begin{proof}
  The proof is by induction on $i$. The base case is $i=1$, in which case $X\setminus X_1=\bigsqcup_{c_0}\mathbb{A}^n$. The required vanishing follows from $\mathbb{A}^1$-invariance of $\mathbf{W}$-cohomology. Now for $i>1$, we consider the localization  situation
  \[
  X\setminus X_{i-1}\hookrightarrow X\setminus X_i\hookleftarrow X_{i-1}\setminus X_i=\bigsqcup_{c_{i-1}}\mathbb{A}^{n-i+1}.
  \]
  The associated localization sequence looks as follows
  \[
  \op{H}^{q-i+1}\bigl(X_{i-1}\setminus X_i,\mathbf{W}(\L)\bigr)\to \op{H}^q\bigl(X\setminus X_i,\mathbf{W}(\L)\bigr)\to \op{H}^q\bigl(X\setminus X_{i-1},\mathbf{W}(\L)\bigr).\]
  The first map has degree $i-1$ because this is the codimension of $X_{i-1}$ in $X$. By the induction hypothesis, the last group vanishes for $q\geq i-1$. By $\mathbb{A}^1$-invariance of $\mathbf{W}$-cohomology, the first group vanishes for $q\geq i$. Consequently, the middle group vanishes for $q\geq i$ which is what we wanted to show. 
\end{proof}

\begin{lemma}
  \label{lem:w-free}
  Let $F$ be a perfect field of characteristic $\neq 2$, let $X$ be a smooth cellular $F$-scheme with cellular filtration
  \[
  \emptyset=X_{n+1}\subset X_n\subset X_{n-1}\subset\cdots\subset X_1\subset X_0=X,
  \]
  as in Definition~\ref{def:cellular}. If the total $\mathbf{W}$-cohomology ring  $\bigoplus_{q,\L}{\rm H}^q\bigl(X,\mathbf{W}(\L)\bigr)$ is a free ${\rm W}(F)$-module, then the same is true for the open smooth subschemes $X\setminus X_i$ for $1\leq i\leq n+1$.
\end{lemma}

\begin{proof}
  We first note that for Chow groups, the restriction maps along open immersions are surjective. This applies in particular to ${\rm Ch}^1(X)\to {\rm Ch}^1(X\setminus X_i)$, hence any mod 2 line bundle class $[\L]\in{\rm Ch}^1(X\setminus X_i)$ extends to $X$. This implies that for all open smooth subschemes $X\setminus X_i$, all relevant line bundle twists for $\mathbf{W}$-cohomology are restricted from $X$. Hence, the inductive argument below actually captures the total $\mathbf{W}$-cohomology rings of $X\setminus X_i$.
  
  The proof now proceeds by induction on the dimension of $X_i$. The base case is $i={n+1}$ where $X_{n+1}=\emptyset$ and we have the freeness of the total $\mathbf{W}$-cohomology ring by assumption. For the inductive step, assume we have the freeness for $X\setminus X_{i+1}$ and we want to prove it for $X\setminus X_i$. We consider the localization sequence for the scheme $U_{i+1}=X\setminus X_{i+1}$ with open subscheme $U_{i}=X\setminus X_{i}$ and closed subscheme $Z=X_i\setminus X_{i+1}\cong \bigsqcup_{c_i}\mathbb{A}^{n-i}$. Since 
  \[
    {\rm H}^{p}_Z\bigl(U_{i+1},\mathbf{W}(\L)\bigr)\cong{\rm H}^{p-i}\bigl(Z,\mathbf{W}(\L\otimes\mathscr{N}_Z)\bigr)\cong\left\{\begin{array}{ll}{\rm W}(F)^{\oplus c_i} & p=i,\\0&\textrm{otherwise},\end{array}\right.
  \]
  the localization sequence yields isomorphisms ${\rm H}^j\bigl(U_{i+1},\mathbf{W}(\L)\bigr) \to {\rm H}^j\bigl(U_i,\mathbf{W}(\L)\bigr)$ for $j\neq i-1,i$ and all line bundles $\L$ on $U_{i+1}$. By the inductive assumption, we only need to deal with freeness in degrees $i-1$ and $i$. The relevant part of the localization sequences for these degrees has the following form:
  \[
  0 \to {\rm H}^{i-1}\bigl(U_{i+1},\mathbf{W}(\L)\bigr)\to {\rm H}^{i-1}\bigl(U_i,\mathbf{W}(\L)\bigr) \to {\rm W}(F)^{\oplus c_i}\to
  \]
  \[
  \to {\rm H}^i\bigl(U_{i+1},\mathbf{W}(\L)\bigr) \to {\rm H}^i\bigl(U_i,\mathbf{W}(\L)\bigr)\to 0.
  \]
  By Lemma~\ref{lem:cell-vanish}, we have ${\rm H}^i\bigl(U_i,\mathbf{W}(\L)\bigr)=0$. Since ${\rm H}^i\bigl(U_{i+1},\mathbf{W}(\L)\bigr)$ is free by the inductive assumption, ${\rm W}(F)^{\oplus c_i}$ splits as direct sum of ${\rm H}^i\bigl(U_{i+1},\mathbf{W}(\L)\bigr)$ and a projective module $B:=\op{Im}\left({\rm H}^{i-1}\bigl(U_i,\mathbf{W}(\L)\bigr) \to {\rm W}(F)^{\oplus c_i}\right)$. By \cite{fitzgerald}, projective ${\rm W}(F)$-modules split as direct sums of a free module and an invertible ideal. Since in the decomposition
  \[
    {\rm W}(F)^{\oplus c_i}\cong {\rm H}^i\bigl(U_{i+1},\mathbf{W}(\L)\bigr)\oplus B
  \]
  the first summand is free and the direct sum is also free, the summand $B$ must also be free. Consequently, the short exact sequence
  \[
  0 \to {\rm H}^{i-1}\bigl(U_{i+1},\mathbf{W}(\L)\bigr)\to {\rm H}^{i-1}\bigl(U_i,\mathbf{W}(\L)\bigr) \to B\to 0
  \]
  splits and ${\rm H}^{i-1}\bigl(U_i,\mathbf{W}(\L)\bigr)$ is the direct sum of two free modules, hence free.
\end{proof}

\begin{remark}
  The conclusion of the lemma can also be obtained for Grassmannians directly: in this case, we have a basis of the $\mathbf{W}$-cohomology in terms of the even Schubert cells -- removing cells then only possibly removes generators of the $\mathbf{W}$-cohomology.
\end{remark}

\subsection{The ladder lemma}
The proof of the K\"unneth formula proceeds in rather close analogy to the topological case as proved e.g.\ in \cite[Theorem 3.15]{hatcher}. The following Lemma will be used both in the proof of the K\"unneth theorem and of the Leray--Hirsch theorem, it is a variation on the compatibility of localization sequences, cf.~Lemma~\ref{lem:compat-locseq-pullback}.

\begin{lemma}\label{lem:ladder}
  Let $F$ be a perfect field of characteristic $\neq 2$, and consider the following diagram of schemes:
  \[\xymatrix{
  &  Y_Z\ar[r]^-{j_Z}\ar[d]_-{f_Z}&	Y\ar[d]^-{f}&		Y_U\ar[l]_-{j_U}\ar[d]^-{f_U}&\\
   & Z\ar[r]_-{i_Z}&	X&		U\ar[l]^-{i_U}& \phantom{a} \hspace{-2 cm}.
  }
  \]
  Here $Z\se X$ is a closed smooth subscheme with complement $U=X\su Z$, and $f\colon Y\to X$ is a morphism of smooth $F$-schemes. Denote by $Y_Z=f^{-1}(Z)$ and $Y_U=f^{-1}(U)$ the preimages, and assume that $Y_Z$ is smooth and $f\colon Y\to X$ is transversal to $Z$.
  Let $\L_X$ and $\L_Y$ be line bundles on $X$ and $Y$, respectively.
  Let $s\colon \La\to {\rm H}^*\bigl(Y, \mathbf{W}(\L_Y)\bigr)$ be a homomorphism of ${\rm W}(F)$-modules. 
  Then the following ladder of $\rm{W}(F)$-module homomorphisms commutes:
  \[
%  \xymatrix@C+3em@R+.5pc{
    \xymatrix{
&    \vdots \ar[d]&& \vdots\ar[d] &\\
  &  \La\otimes_{{\rm W}(F)} \op{H}^*\bigl(U,\mathbf{W}(\L_X)\bigr) \ar[rr]^-{\cup(s_U\otimes f_U^*)} \ar[d]_-{\id\otimes \partial} &&\op{H}^*\bigl(Y_U, \mathbf{W}(\L)\bigr) \ar[d]^-{\partial} &\\
   & \La\otimes_{{\rm W}(F)} \op{H}^*\bigl(Z,\mathbf{W}(\L_X\otimes \mathscr{N}_Z)\bigr) \ar[rr]^-{\cup(s_Z\otimes f_Z^*)} \ar[d]_-{\id\otimes (i_Z)_*}&& \op{H}^*\bigl(Y_Z, \mathbf{W}(\L\otimes f^*\mathscr{N}_Z)\bigr) \ar[d]^-{(j_Z)_*}& \\
  &  \La\otimes_{{\rm W}(F)} \op{H}^*\bigl(X,\mathbf{W}(\L_X)\bigr) \ar[d]_-{\id\otimes i_U^*} \ar[rr]^-{\cup(s\otimes f^*)} && \op{H}^*\bigl(Y, \mathbf{W}(\L)\bigr)\ar[d]^-{j_U^*} &\\
&	\vdots &&\vdots& %\phantom{a}\hspace{-4 cm}
}  
\]
Here $s_U=j_U^*\circ s$, $s_Z=j_Z^*\circ s$, $\L=f^*\L_X\otimes \L_Y$ and $\cup\colon {\rm H}^*(\cdot)\otimes {\rm H}^*(\cdot)\to {\rm H}^*(\cdot)$ denotes the intersection/cup product multiplication. If $\La$ is a free ${\rm W}(F)$-module, or if all the $\mathbf{W}$-cohomologies of $U, Z$ and $X$ are free ${\rm W}(F)$-modules then both vertical sequences are exact.
\end{lemma}

\begin{proof}
  The left column arises from tensoring the localization sequence (for the scheme $X$ with open subscheme $U$ and closed subscheme $Z$) with $\La$.  The right column is the localization sequence for the scheme $Y$ with closed subscheme $Y_Z$ and open subscheme $Y_U$. 
The horizontal maps all arise from compositions of the form
\[\xymatrix@C+1em
{
		\La\otimes_{{\rm W}(F)}{\rm H}^*(R, \L_X)\ar[r]^-{s_R\otimes f_R^*}&   {\rm H}^*(Y_R,\L_Y)\otimes_{{\rm W}(F)}{\rm H}^*(Y_R,\L_X)\ar[r]^-{\cup}&{\rm H}^*(Y_R,\L)
	}
\]
for $R=X, U$ or $Z$ (and we omitted $\mathbf{W}$ for typographical reasons). 

The commutativity of the ladder consists of the commutativity of three squares. The square involving pullbacks follows from $s_U=j_U^*\circ s$, $i_U\circ f_U=f\circ j_U$ and naturality of the cup product.\footnote{Alternatively, it could be deduced on the level of Gersten complexes: the contributions for points $x\in U$ are unchanged by restriction, and the contributions for points $x\not\in U$ map to $0$.}  The square involving pushforwards states that for all $v\in \La$ and $z\in {\rm H}^*(Z)$:
\[(j_Z)_*\bigl(j_Z^*s(v)\cdot f_Z^*(z)\bigr)=s(v)\cdot (j_Z)_*f_Z^*(z)=s(v)\cdot f^*(i_Z)_*(z),\]
where the first equality follows from the projection formula.\footnote{Again, this can also be done using Gersten complexes. Points on $Z$ and $Y_Z$ can be considered as points on $X$ and $Y_X$, respectively, and there is only some adjustment necessary because of the change of twist with the normal bundle.}  
The second equality holds by base change, using transversality of $f$ to $Z$, i.e., $f^*\mathscr{N}_Z\iso \mathscr{N}_{Y_Z}$, cf.\ \cite[Theorem 2.12]{AsokFaselEuler}.

Finally, commutativity of the square involving boundary maps states that for all $v\in \La$ and $u\in {\rm H}^*(U)$:
\[  \partial\left(j_U^*s(v)\cdot f_U^*(u)\right)=j_Z^*s(v)\cdot \partial f_U^*(u)=s_Z(v)\cdot f_Z^*\partial(u),
\]
where the first equality follows from the derivation property $\partial (j_U^*(x)\cdot y)=j_Z^*(x)\cdot \partial y$ of the boundary map, cf.~\cite[Proposition 3.4]{fasel:chow-witt-lectures} or \cite[Proposition 2.6]{chowwitt}.  The second equality is naturality of the boundary map in the localization sequence, see Lemma~\ref{lem:compat-locseq-pullback}.

Finally, concerning exactness: if $\La$ is a free ${\rm W}(F)$-module, then tensoring the exact localization sequence with $\La$ preserves exactness, which proves the claim in this case. If on the other hand, all the $\mathbf{W}$-cohomology groups of $U$, $Z$ and $X$ are free ${\rm W}(F)$, then the exact localization sequence is a bounded acyclic complex of free ${\rm W}(F)$-modules. In this case, the tensor product with any $\La$ will preserve the acyclicity. 
\end{proof}

\begin{remark}
  We will use this lemma when $f\colon Y\to X$ is a fiber bundle, or more simply a product projection. There are two typical choices for $s$: when $Y$ is a product $X\times X'$, then the ${\rm W}(F)$-module homomorphism $s\colon {\rm H}^*(X')\to {\rm H}^*(Y)$ is the pullback via projection to the second factor $s=p_{X'}^*$. In the situation of the Leray--Hirsch theorem, the ${\rm W}(F)$-module homomorphism $s\colon {\rm H}^*(X')\to {\rm H}^*(Y)$ is a section of $j_{X'}^*$, where $j_{X'}\colon X'\se Y$ is the inclusion of the fiber of the projection map $f$.
\end{remark}

\subsection{K\"unneth formula}

\begin{proposition}[K\"unneth formula]
  \label{prop:kuenneth}
  Let $F$ be a perfect field of characteristic $\neq 2$, let $X, Y$ be smooth $F$-schemes, and let $\L_X$ and $\L_Y$ be line bundles over $X$ and $Y$, respectively. Assume that $X$ is cellular. Assume moreover that for $X$ or $Y$ all $\mathbf{W}$-cohomology groups ${\rm H}^q\bigl(X,\mathbf{W}(\L)\bigr)$ are free ${\rm W}(F)$-modules. Then the exterior product map
  \[
    {\rm H}^*\bigl(X,\mathbf{W}(\L_X)\bigr)\otimes_{{\rm W}(F)}{\rm H}^*\bigl(Y,\mathbf{W}(\L_Y)\bigr)\to {\rm H}^*\bigl(X\times Y,\mathbf{W}(\L_X\boxtimes \L_Y)\bigr)
    \]
    from Definition~\ref{def:ext-prod} is an isomorphism.
\end{proposition}

\begin{proof}
  Denote by
  \[
  \emptyset=X_{n+1}\subset X_n\subset X_{n-1}\subset\cdots\subset X_1\subset X_0=X
  \]
  the cellular filtration of $X$ as in Definition~\ref{def:cellular}.
  
  We apply Lemma~\ref{lem:ladder} for the scheme $U_{j+1}=X\su X_{j+1}$ with open subscheme $U_j=X\su X_j$ and closed subscheme $Z=X_j\su X_{j+1}$ and the map $f\colon U_j\times Y\to U_j$. Take the ${\rm W}(F)$-module homomorphism $s$ to be the pullback
  \[
  s:=p^*\colon \La^q={\rm H}^q\bigl(Y, \mathbf{W}(\L_Y)\bigr)\to {\rm H}^q\bigl(U_j\times Y, \mathbf{W}(\L)\bigr),
  \]
  where $\L=\L_X\boxtimes\L_Y$. 
  Taking direct sums for all $q$ with $p+q=i$, we obtain from Lemma~\ref{lem:ladder} a commutative ladder of long exact sequences:
  \[
  \xymatrix{
    &		\vdots \ar[d]& \vdots\ar[d]& \\
    &		\bigoplus_{p+q=i}\op{H}^{p-j}\bigl(Z,\mathbf{W}(\L_X\otimes\mathscr{N}_Z)\bigr)\otimes \La^q \ar[r] \ar[d] &\op{H}^{i-j}\bigl(Z\times Y,\mathbf{W}(\L\otimes\op{pr}_1^*\mathscr{N}_Z)\bigr) \ar[d]& \\
    &		\bigoplus_{p+q=i}\op{H}^p\bigl(U_{j+1},\mathbf{W}(\L_X)\bigr)\otimes \La^q \ar[r] \ar[d]& \op{H}^i\bigl(U_{j+1}\times Y,\mathbf{W}(\L)\bigr) \ar[d] &\\
    &		\bigoplus_{p+q=i}\op{H}^p\bigl(U_{j},\mathbf{W}(\L_X)\bigr)\otimes \La^q\ar[d] \ar[r] & \op{H}^i\bigl(U_{j}\times Y,\mathbf{W}(\L)\bigr)\ar[d] &\\
    &		\vdots &\vdots &\phantom{a} %\hspace{-2 cm}.
  }
  \]
  
  Now we prove the claim by induction on $j$. The base case $j=0$ is $U_0=X\setminus X_0=\emptyset$ is empty, the case $U_1=X\setminus X_1$ follows from homotopy invariance since $X\setminus X_1$ is a disjoint union of copies of affine space. For the inductive step, assume that we have the claim for $U_j$, we want to prove it for $U_{j+1}$. By induction, the morphisms
  \[
  \bigoplus_{p+q=i}\op{H}^p\bigl(U_{j},\mathbf{W}(\L_X)\bigr)\otimes \La^q\to \op{H}^i\bigl(U_{j}\times Y,\mathbf{W}(\L)\bigr)
  \]
  are isomorphisms. 
  Since $Z=X_j\setminus X_{j+1}$ is a disjoint union of $c_j$ copies of affine spaces $\mathbb{A}^{n-j}$, we have
  \[
    {\rm H}^{p-j}\bigl(Z,\mathbf{W}(\L_X\otimes\mathscr{N}_Z)\bigr)=\left\{\begin{array}{ll}0 & p\neq j \\ {\rm W}(F)^{\oplus c_j} & p=j,\end{array}\right.
    \]
    noting that the twisting bundle on affine spaces will always be trivial. Moreover, by homotopy invariance, the maps
    \[
    \bigoplus_{p+q=i}\op{H}^{p-j}\bigl(Z,\mathbf{W}(\L_X\otimes\mathscr{N}_Z)\bigr)\otimes \La^q\to \op{H}^{i-j}\bigl(Z\times Y,\mathbf{W}(\L\otimes\op{pr}_1^*\mathscr{N}_Z)\bigr)
    \]
    are also isomorphisms. The claim for $U_{j+1}$ then follows from the 5-lemma.
\end{proof}

\subsection{Leray--Hirsch theorem}
We now turn to the proof of the Leray--Hirsch theorem for $\mathbf{W}$-cohomology. In a nutshell, the theorem states that (under suitable cellularity and freeness conditions) the cohomology of a fiber bundle looks like the cohomology of a product if the generators of the cohomology of the fiber can be lifted to the cohomology of the total space.

Before we state the theorem we recall the appropriate notion of fiber bundle for our setting: A \emph{Zariski-locally trivial fibration} with fiber $X$ is a morphism $p\colon E\to B$ such that for any scheme-theoretic point $b\in B$ there exists a Zariski-open neighbourhood $U$ of $b$ and an isomorphism $E\times_B U\to X\times U$ making the following diagram commutative:
\[
\xymatrix{
  E\times_BU \ar[rr]^\cong \ar[rd]_{p|_U} &&  X\times U \ar[ld]^{{\rm pr}_1} \\
  &U
}
\]

\begin{proposition}[Leray--Hirsch theorem]
  \label{prop:leray-hirsch}
  Let $F$ be a perfect field of characteristic $\neq 2$, let $B$ be a smooth $F$-scheme and let $p\colon E\to B$ be a Zariski-locally trivial fibration with smooth cellular fiber $X$. Assume that there exists a line bundle $\mathscr{L}_E$ on $E$ and a Zariski-covering $\{U_i\}_i$ of $B$ trivializing $p$ such that the following are satisfied:
  \begin{enumerate}
  \item There exists some line bundle $\mathscr{K}$ on $X$ such that ${\rm H}^*(X,{\bf W}(\mathscr{K}))$ is a rank $r$ free ${\rm W}(F)$-module with generators $f_1,\dots,f_r$, and for each trivializing subset $U_i\subseteq B$ there is an isomorphism $\mathscr{L}_E|_{p^{-1}(U_i)}\cong q^*\mathscr{K}$ with $q\colon p^{-1}(U_i)\cong U_i\times X\to X$ the natural projection. 
  \item There exist classes $e_1,\dots,e_r$ in ${\rm H}^*(E,{\bf W}(\mathscr{L}_E))$, such that for each trivializing subset $U_i\subseteq B$ the restrictions of $e_1,\dots,e_r$ along the inclusion  $i\colon p^{-1}(U_i)\hookrightarrow E$ satisfy
    \[i^*(e_1)=q^*(f_1),\dots,i^*(e_r)=q^*(f_r),
    \]
    i.e., these classes form a free basis of the Witt-sheaf cohomology
    \[
      {\rm H}^*(p^{-1}(U_i),{\bf W}(\mathscr{L}_E|_{p^{-1}(U_i)}))\cong {\rm H}^*(p^{-1}(U_i),{\bf W}(q^*\mathscr{K}))
    \]
    as a module over ${\rm H}^*(U_i,{\bf W}(\mathscr{L}))$.  
  \end{enumerate}
  Then for every line bundle $\mathscr{L}$ on $B$, the map
  \[\xymatrix@R-2pc{
  	  \Phi\colon {\rm H}^*\bigl(X,\mathbf{W}(\mathscr{K})\bigr)\otimes_{{\rm W}(F)}{\rm H}^*\bigl(B,\mathbf{W}(\L)\bigr)\ar[r] &{\rm H}^*\bigl(E,\mathbf{W}(\L_E\otimes p^\ast\L)\bigr)\\
  	  f_i\otimes b\ar@{}[r] ^(.53){}="a"^(.7){}="b" \ar@{|->} "a";"b" &p^\ast(b)e_i
  }
  \]
  is an isomorphism of ${\rm H}^*\left(B,\mathbf{W}\right)$-modules. In particular, we have an isomorphism
  \[
    {\rm H}^*\bigl(E,\mathbf{W}(\L_E\otimes p^\ast\L)\bigr)\cong\bigoplus_{i=1}^r{\rm H}^*\bigl(B,\mathbf{W}(\L)\bigr)\cdot e_i
  \]
  of ${\rm H}^*(B,\mathbf{W})$-modules. 
\end{proposition}

\begin{proof}
  We offer two arguments.

  (A) The first argument essentially follows the argument for higher Chow groups in \cite[Theorem~6.3]{krishna}, but it exploits the perfectness of $k$: Since $k$ is perfect, there exists a filtration
  \[
  \emptyset=B_{n+1}\subset B_n\subset\cdots\subset B_1\subset B_0=B
  \]
  of the base scheme $B$ by closed subschemes such that for all $0\leq i\leq n$ the scheme $B_i\setminus B_{i+1}$ is smooth and the given fibration is trivial over it. Then we have the smooth subschemes $U_i=B\setminus B_i$ and $V_i=U_i\setminus U_{i-1}=B_{i-1}\setminus B_i$ and we define $E_i=p^{-1}(U_i)$ and $W_i=p^{-1}(V_i)$. We want to prove by induction on $i$ that the following map is an isomorphism (omitting line bundle decorations for simplicity here):
  \[
  \Phi\colon {\rm H}^*(X,\mathbf{W})\otimes_{{\rm W}(F)}{\rm H}^*(U_i,\mathbf{W})\to {\rm H}^*(E_i,\mathbf{W})
  \]
  The base case is $i=1$, where $E_1=U_1\times X$ and the result follows from the K\"unneth formula in Proposition~\ref{prop:kuenneth} and condition (2) above. In general, by Lemma~\ref{lem:ladder}, we have a commutative diagram:
  \[
  \xymatrix{
    \vdots \ar[d]& \vdots\ar[d] &\\
    \op{H}^*\bigl(X,\mathbf{W}(\L_E)\bigr)\otimes_{{\rm W}(F)} \op{H}^*\bigl(U_i,\mathbf{W}(\L_B)\bigr) \ar[r] \ar[d] &\op{H}^*\bigl(E_i,\mathbf{W}(\L)\bigr) \ar[d]& \\
    \op{H}^*\bigl(X,\mathbf{W}(\L_E)\bigr)\otimes_{{\rm W}(F)} \op{H}^*\bigl(V_{i+1},\mathbf{W}(\L_B\otimes\mathscr{N}_{V_{i+1}})\bigr) \ar[r] \ar[d]& \op{H}^*\bigl(W_{i+1},\mathbf{W}(\L\otimes p^*\mathscr{N}_{V_{i+1}})\bigr) \ar[d] &\\
    \op{H}^*\bigl(X,\mathbf{W}(\L_E)\bigr)\otimes_{{\rm W}(F)} \op{H}^*\bigl(U_{i+1},\mathbf{W}(\L_B)\bigr) \ar[d] \ar[r] & \op{H}^*\bigl(E_{i+1},\mathbf{W}(\L)\bigr)\ar[d] &\\
    \vdots &\vdots& %\phantom{a} \hspace{- 2 cm} .
  }  
  \]
  Then the claim follows from the 5-lemma with horizontal isomorphisms from K\"unneth formula (for the lines with cohomology of $V_{i+1}$) and induction assumption (for the lines with cohomology of $U_i$).

  (B) The other argument allows to remove the perfectness hypothesis. Instead of the filtration from \cite{krishna}, we use the trivializing covering $\{U_i\}_i$, and the proof proceeds by an induction on the number of trivializing opens $U_i$ in the covering. In a nutshell, the K\"unneth formula proves the claim for trivial fibrations, and the assumptions (1) and (2) of the Leray--Hirsch theorem imply that these K\"unneth isomorphisms glue to the global Leray--Hirsch isomorphism.
  
  More precisely: in the base case of a trivial fibration, the condition in (2) implies the claim directly by the K\"unneth formula of Proposition~\ref{prop:kuenneth}. In the general case of a covering by $n$ subsets, we use a commutative ladder of Mayer--Vietoris sequences:
  \[
  \xymatrix{
    \vdots \ar[d]& \vdots\ar[d] &\\
    \op{H}^*\bigl(X,\mathbf{W}(\L_E)\bigr)\otimes_{{\rm W}(F)} \op{H}^*\bigl(B,\mathbf{W}(\L_B)\bigr) \ar[r] \ar[d] &\op{H}^*\bigl(E,\mathbf{W}(\L)\bigr) \ar[d]& \\
    \op{H}^*\bigl(X,\mathbf{W}(\L_E)\bigr)\otimes_{{\rm W}(F)} \op{H}^*\bigl(U_n\sqcup V,\mathbf{W}(\L_B)\bigr) \ar[r]_-{\cong} \ar[d]& \op{H}^*\bigl(p^{-1}(U_n)\sqcup p^{-1}(V),\mathbf{W}(\L)\bigr) \ar[d] &\\
    \op{H}^*\bigl(X,\mathbf{W}(\L_E)\bigr)\otimes_{{\rm W}(F)} \op{H}^*\bigl(U_n\cap V,\mathbf{W}(\L_B)\bigr) \ar[d] \ar[r]_-{\cong} & \op{H}^*\bigl(p^{-1}(U_n\cap V),\mathbf{W}(\L)\bigr)\ar[d] &\\
    \vdots &\vdots& %\phantom{a} \hspace{- 2 cm} .
  }  
  \]
  Here $V=\bigcup_{i=1}^{n-1}U_i$ is the union of $n-1$ trivializing subsets, which satisfies the Leray--Hirsch theorem by the induction assumption. The claim that we actually have a commutative ladder of exact sequences follows by an argument similar to the one proving Lemma~\ref{lem:ladder}, and the freeness assumption for the cohomology of $X$. That the middle vertical map is an isomorphism follows by the inductive assumption (for $V$) and the base case (for $U_n$). That the lower vertical map is an isomorphism follows from the K\"unneth formula and assumption (2), just like in the base case, because $p$ trivializes over $U_n\cap V$. The claim then follows from the 5-lemma. 
\end{proof}

\begin{remark}
  \label{rem:leray-hirsch-pic}
  Assume that for a fibration $X\xrightarrow{i} E\xrightarrow{p} B$ the conditions of Proposition~\ref{prop:leray-hirsch} above are satisfied, and that in addition the conditions for the Leray--Hirsch theorem for Chow groups, cf.\ \cite[Theorem 6.3]{krishna} are also satisfied. The latter implies in particular that the natural map
  \[
  \op{Pic}(X)\oplus\op{Pic}(B)\to \op{Pic}(E)\colon \left(i^\ast(a_j), b\right)\mapsto a_j\otimes p^\ast(b)
  \]
  is an isomorphism, where the $a_j\in \op{Pic}(E)$ restrict to a basis $i^\ast(a_j)$ of $\op{Pic}(X)$. This means in particular, that every line bundle in $\op{Pic}(E)$ is isomorphic to one of the form $\L_{E}\otimes p^\ast\L$ for suitable $\L_{E}\in\op{Pic}(E)$ and $\L\in\op{Pic}(B)$. This means that under these assumptions, the Leray--Hirsch theorem in Proposition~\ref{prop:leray-hirsch} above can be used to describe the total $\mathbf{W}$-cohomology ring. 
\end{remark}

\begin{remark}
  For the present paper, the situation in which we want to apply the Leray--Hirsch theorem above is the case of iterated Grassmannian bundles. Note that if $X$ is a cellular variety, $\EE$ is a rank $n$ vector bundle over $X$, and $0<k<n$ is a natural number, the Grassmannian bundle $\mathscr{G}(k,n,\EE)$ is itself cellular. For any cell $\iota\colon\mathbb{A}^d\to X$, the restriction of the vector bundle $\EE$ along $\iota$ is trivial, hence $\mathscr{G}(k,n,\EE)\times_X\mathbb{A}^d\cong  {\rm Gr}(k,n)\times\mathbb{A}^d$. The cells for $\mathscr{G}(k,n,\EE)$ are then the products of cells in the base with the cells of ${\rm Gr}(k,n)$. In the case of maximal rank flag varieties, these can also be obtained via the Bia\l{}ynicki-Birula method from torus actions on the flag varieties. For a Grassmannian bundle over a cellular variety $B$, we can also choose the cellular filtration for $B$ in the above proof (version A) of the Leray--Hirsch theorem. The reason is again that the restriction of the Grassmannian bundle to the cells of the cellular filtration will be trivial. 
\end{remark}

\begin{remark}
One could probably also formulate the Leray--Hirsch theorem above more generally for $\eta$-invertible theories representable in the stable motivic homotopy category. Parts of the method can already be found in \cite{ananyevskiy} for bundles with oriented Grassmannian fibers. Theorem 4 of loc.~cit.\ is a K\"unneth formula for products with oriented Grassmannians ${\rm SGr}(2,n)$,  and Theorem 5 is the corresponding version for general special linear bundles. Finally, \cite[Theorem 6]{ananyevskiy} is then a description of the cohomology of maximal rank oriented flag varieties for $\eta$-invertible theories, comparable to our Theorem~\ref{thm:maxrank}. 
\end{remark}

\begin{example}
  \label{ex:leray-hirsch-pn}
  We briefly discuss the projective bundle formula (inasmuch as it holds for $\mathbf{W}$-cohomology) as a special case of the above Leray--Hirsch theorem and compare to the results in \cite{fasel:ij}.

  So consider a smooth scheme  $X$ and a rank $r$ vector bundle $\EE$ on $X$ and denote by $\mathbb{P}(\EE):=\op{Proj}\bigl(\op{Sym}(\EE^\vee)\bigr)$ the associated projective bundle. There is a dichotomy in the application of the Leray--Hirsch theorem above to projective bundles.
  \begin{description}
  \item[even-dimensional projective spaces] In this case, the total $\mathbf{W}$-cohomology ring of $\mathbb{P}^{2n}$ is isomorphic to ${\rm W}(F)[{\rm e}_{2n}]/({\rm e}_{2n}^2)$ with ${\rm e}_{2n}$ the Euler class of the tautological rank $2n$-bundle in the twisted $\mathbf{W}$-cohomology. For a vector bundle $\EE$ of rank ${2n+1}$ over $X$, the Euler class of the associated tautological quotient bundle $\QQ$ on $\mathbb{P}(\EE)$ restricts to ${\rm e}_{2n}$ in each fiber, and therefore the Leray--Hirsch theorem is applicable.
  \item[odd-dimensional projective spaces] In this case, the total $\mathbf{W}$-cohomology ring of $\mathbb{P}^{2n-1}$ is isomorphic to ${\rm W}(F)[{\rm R}]/({\rm R}^2)$ with ${\rm R}$ the class of a point. This class fails to be a characteristic class of a vector bundle, making this case more problematic. Indeed, for a vector bundle $\EE$ of rank $2n$ over $X$, the Euler class of $\EE$ is an obstruction for the existence of a cohomology class on $\mathbb{P}(\EE)$ which restricts to the class ${\rm R}$ in each fiber. In a sense, having a cohomology class on $\mathbb{P}(\EE)$ which restricts to the class of a point in each fiber is essentially equivalent to the existence of a non-vanishing section. Therefore, the Leray--Hirsch theorem is only applicable if the Euler class of $\EE$ is trivial. For $n=1$, this means that the projective bundle is already a product. 
  \end{description}
  We will therefore only consider the case of odd rank vector bundles in which the Leray--Hirsch theorem is generally applicable. Note, however, that \cite[Theorems 9.2 and 9.4]{fasel:ij} provide partial descriptions of the cohomology of $\mathbb{P}(\EE)$ in the even-rank case. In particular, the long exact sequence of Theorem 9.4, loc.~cit., shows that the decomposition in the Leray--Hirsch theorem will generally fail in the even-rank case (unless again the Euler class of $\EE$ is trivial).

  In the odd-rank bundle case, we already mentioned that the Euler class
  \[
    {\rm e}(\QQ)\in {\rm H}^{2n}\bigl(\mathbb{P}(\EE),\mathbf{W}(\omega_{\mathbb{P}(\EE)/X})\bigr)
  \]
  of the rank $2n$ quotient bundle $\QQ$ restricts to the Euler class generator for each fiber of $\mathbb{P}(\EE)$. Additionally, we of course need the class $1\in{\rm H}^0\bigl(\mathbb{P}(\EE),\mathbf{W}\bigr)$ which also restricts to a generator of the $\mathbf{W}$-cohomology in degree $0$ of each fiber. For each line bundle $\L$ on $X$, the Leray--Hirsch theorem then provides isomorphisms
  \begin{eqnarray*}
  {\rm H}^i\bigl(X,\mathbf{W}(\L)\bigr)&\xrightarrow{p^\ast}&{\rm H}^i\bigl(\mathbb{P}(\EE),\mathbf{W}(p^\ast\L)\bigr)\\
  {\rm H}^i\bigl(X,\mathbf{W}(\L)\bigr)&\xrightarrow{{\rm e}(\QQ)}&{\rm H}^{i+2n}\bigl(\mathbb{P}(\EE),\mathbf{W}(p^\ast\L\otimes\omega_{\mathbb{P}(\EE)/X})\bigr)
  \end{eqnarray*}
  where the second map is multiplication by the Euler class ${\rm e}(\QQ)$. This recovers exactly the statement of \cite[Theorem 9.1]{fasel:ij}.
\end{example}

\begin{remark}
  A small remark concerning the bundle notation: The bundle denoted by $\mathcal{G}_{\EE}$ in \cite{fasel:ij} is related to the relative bundle of differential forms via $\Omega_{\mathbb{P}(\EE)/X}(1)\cong \mathcal{G}_{\EE}$. For the top exterior power, we then get
\[
\Lambda^r\left(\Omega^1_{\mathbb{P}(\EE)/X}(1)\right)\cong \omega_{\mathbb{P}(\EE)/X}(r) \cong \Lambda^r\mathcal{G}_{\EE}.
\]
A convenient summary of the relevant canonical bundles can be found in \cite[Section 1.1]{fasel:ij}.
\end{remark}

\section{Flag varieties: geometry and Chow rings}
\label{sec:flags}

After setting up the tools for computing Witt-sheaf cohomology, we now turn to computing the specific case of flag varieties (of type A). In the current section we recall some geometric properties of flag varieties relevant to our discussion, such as characteristic classes in Chow rings and Poincar\'e polynomials, while also introducing some notation. In the subsequent three sections, we carry out the main computation of the paper, namely, we compute the $\mathbf{W}$-cohomology of flag varieties.

The computation proceeds in two steps. The first step in Section~\ref{sec:maximal} provides a computation of the maximal rank cases. For this, we use a description of flag varieties as iterated Grassmannian bundles recalled below, and then inductively use the Leray--Hirsch theorem to compute the cohomology. The second step in Section~\ref{sec:sadykov} extends the computation from the maximal rank case to the general case. The key step there is to relate different flag varieties via sphere bundles associated to the tautological subquotient bundles. This requires in particular descriptions of kernels of multiplication by characteristic classes on cohomology rings which we establish in Section~\ref{sec:ann-euler}.

\subsection{Definition and geometry of flag varieties}
Given a partition $\D=(d_1,\dots,d_m)$ of $N=\sum d_i$, denote by $\op{Fl}(\D)$  the variety of flags 
\begin{equation}\label{eq:flagD}
V_\bullet=(0=V_0\subset V_1\subset\cdots\subset V_m=V)
\end{equation}
in a fixed vector space $V$ of dimension $N$, where $\op{rk} V_i/V_{i-1}=d_i$. One way to make more explicit the structure as algebraic variety is the description of the partial flag variety $\Fl(\D)$ as a projective homogeneous variety ${\rm GL}_N/P_\D$, where $P_\D$ is the parabolic subgroup of upper block triangular matrices in which the blocks have the sizes prescribed by the partition $\D$. The parabolic subgroup $P_\D$ is a semidirect product of the Levi subgroup $\prod_{i=1}^m{\rm GL}_{d_i}$ by a unipotent group of strictly upper block triangular matrices.

\begin{remark}
  Since it will appear frequently in Section~\ref{sec:maximal}, we want to briefly explain the terminology ``maximal rank'' for flag varieties. The terminology comes from considering the real flag varieties as homogeneous spaces ${\rm O}(N)/\bigl({\rm O}(d_1)\times\cdots\times{\rm O}(d_m)\bigr)$. Considered as semisimple Lie groups, the rank is $n$ for the orthogonal groups ${\rm O}(2n+1)$ in type $B_n$ and the orthogonal groups ${\rm O}(2n)$ in type $D_n$. We can then consider the discrepancy between the rank of the product ${\rm O}(d_1)\times\cdots\times{\rm O}(d_m)$ and the rank of the full group ${\rm O}(N)$. The difference
  \[
  \op{rk}{\rm O}(N)-\op{rk}\bigl({\rm O}(d_1)\times\cdots\times{\rm O}(d_m)\bigr)
  \]
  is $0$ (and hence the rank of the subgroup is maximal relative to the containing orthogonal group) precisely if at most one of the $d_i$ is odd. Each additional odd $d_i$ increases the difference resp. decreases the rank of the subgroup. This is some justification for the terminology ``maximal rank''.
\end{remark}

\subsection{Tautological vector bundles}

The flag variety $\Fl(\D)$ comes with defining tautological subbundles $\SS_i$, whose fiber over the flag $V_\bullet\in \Fl(\D)$ is $V_i$. There are two natural exact sequences to consider:
\begin{equation}\label{SES}
  \begin{split}
    \xymatrix{			
      0\ar@{->}[r]& \SS_i\ar@{->}[r]^-{}&  \mathbb{A}^N\ar@{->}[r]&
      \QQ_i\ar@{->}[r]&0,}\\
    \xymatrix{
      0\ar@{->}[r]& \SS_{i-1}\ar@{->}[r]^-{}&  \SS_i\ar@{->}[r]&
      \DD_i\ar@{->}[r]&0,}
  \end{split}
\end{equation}
which give rise to the \emph{tautological quotient bundles} $\QQ_i$ and the \emph{tautological subquotient bundles} $\DD_i$.
The ranks of these bundles are as follows:
$$ \op{rk}\DD_i=d_i, \qquad \op{rk}\SS_i=\sum_{j=1}^i d_j,\qquad \op{rk}\QQ_i=\sum_{j=i+1}^m d_j.$$

\subsection{Flag varieties as iterated Grassmannian bundles}

For later use, we want to recall the well-known description of partial flag varieties as iterated Grassmannian bundles. As before, we use the terminology $\mathscr{G}(k,\EE)$ to denote the Grassmannian bundle of rank $k$ subbundles for a rank $n$ vector bundle $\EE$ over some smooth variety $X$. 

Fix $\D=(d_1,\dots,d_m)$, and let $\Fl(\D)$ be the corresponding flag variety. For given $0<k<d_i$, the flag variety $\Fl(\D')$ with $\D'=(d_1,\dots,d_{i-1},k,d_i-k,\dots,d_m)$ can be described as a Grassmannian bundle over $\Fl(\D)$ via an isomorphism
$$ \mathscr{G}(k,\DD_i) \iso \Fl(\D').$$
Here the projection from the Grassmannian bundle to $\Fl(\D)$ is given by forgetting the $k+s_{i-1}$-dimensional subspace $V_i'$, where  $s_i=\sum_{j=1}^{i}d_j$. The fiber of $\Fl(\D')$ over a point of $\Fl(\D)$ given by a flag $0=V_0\subset V_1\subset \cdots\subset V_m=V$ is given by the Grassmannian of $k+s_{i-1}$-dimensional subspaces $V_{i-1}\subset W\subset V_i$, or equivalently, the Grassmannian of $k$-dimensional subspaces of $V_i/V_{i-1}$. 

Using this correspondence, we obtain $\Fl(\D)$ for $\D=(d_1\stb d_m)$ as an $m$-step iterated Grassmannian bundle, starting from $\Gr(d_1,V)$ as follows.  Denote
\[\D^{(i)}=(d_1\stb d_i, \sum_{j=i+1}^m d_j).\]
Then $\Fl(\D^{(1)})=\Gr(d_1,V)$ and $\D^{(m-1)}=\D$. In the $(i-1)$th step, let $\DD_j$ denote the $j$th tautological subquotient bundle over $\Fl(\D^{(i-1)})$. Then 
$$\mathscr{G}(d_{i},\DD_{i})\iso \Fl(\D^{(i)}).$$

By this identification the tautological subquotient bundles $\DD_i'$ and $\DD_{i+1}'$ over $\Fl(\D^{(i)})$ are exactly $\SS_{\DD_i}$ and $\QQ_{\DD_i}$ over $\mathscr{G}(d_{i},\DD_{i})$.

\subsection{Chow rings of partial flag varieties}

The Chow ring of a flag variety has two well-known presentations - one by generators and relations (characteristic classes) and another through an additive basis with structure constants (Schubert classes). We will mostly need the characteristic class picture for the present paper. 

The Chow ring of the partial flag variety $\Fl(\D)$ for $\D=(d_1,\dots,d_m)$ is generated by characteristic classes of the tautological subquotient bundles $\DD_1,\dots,\DD_m$ as follows, see e.g.\ \cite[\S 21]{BottTu} or \cite[Remark 3.6.16]{Manivel}:
\begin{equation}\label{eq:chowringFlD}
 \CH^*\bigl(\Fl(\D)\bigr)=\frac{\Z[\op{c}_j(\DD_i)\mid i=1,\dots,m, \,j=1, \dots, d_i]}{\prod_{i=1}^m \op{c}(\DD_i)-1},
\end{equation}
This presentation is \emph{as graded ring}, with the Chern classes ${\rm c}_j(\DD_i)$ in degree $j$, and the ideal of relations is the \emph{graded ideal} generated by $\prod_{i=1}^m\op{c}(\DD_i)-1$, where $\op{c}(\DD_i)=1+{\rm c}_1(\DD_i)+\cdots+{\rm c}_{d_i}(\DD_i)$ is the total Chern class of $\DD_i$. Note that the Chow ring of the flag varieties is independent of the base field, i.e., the field of definition of the flags. 

Recall that
$$\op{Pic}\bigl(\Fl(\D)\bigr)/ 2=\mathbb{F}_2\bra \ell_1\stb \ell_m\ket/ \sum \ell_i,$$
where $\ell_i$ is the class of the determinant bundle $\det\DD_i$. Using the relation, any element $\L$ of $\op{Pic}\bigl(\Fl(\D)\bigr)/ 2$ can be represented by a subset $S\se \{1\stb m-1\}$ as 
\begin{equation}\label{eq:linebundle}
	\L_S=\otimes_{j\in S}\ell_j.
\end{equation}

For the computation of the kernel and cokernel of multiplication with Euler classes we will use the Poincar\'e polynomials of the Chow rings $\CH^*(\op{Fl}(\D))$ of flag varieties.

\begin{proposition}
  \label{prop:poincare-mod2}
  Let $\D=(d_1,\dots,d_m)$, $N=\sum_{i=1}^m d_i$. Then the Poincar{\'e} polynomial for the Chow ring of the flag variety $\op{Fl}(\D)$ is given by
  \[
  \op{P}\bigl(\CH^*\bigl(\op{Fl}(\D)\bigr),t\bigr)=\frac{\prod_{j=1}^N(1-t^j)}{\prod_{i=1}^m\prod_{j=1}^{d_i}(1-t^j)}.
  \]
\end{proposition}

\begin{proof}
  The flag variety $\Fl(\D)$ can be written as a sequence of iterated Grassmannians (as we recalled above). Then the proof proceeds as in \cite{BottTu}. 
\end{proof}

\begin{remark}\label{rmk:CHfree}
  The Chow ring $\CH^*\bigl(\Fl(\D)\bigr)$ is free as a $\Z$-module, since it has a stratification by affine spaces, the Schubert cells, cf.\ e.g.\ \cite[Examples 1.9.1 and 14.7.16]{Fulton}. In particular, a $\Z$-module basis is given by the Schubert cycles $[\si_I]$, which are indexed by $I\in S_N/(S_{d_1}\times \ldots \times S_{d_m})$.
\end{remark}

\subsection{Poincar\'e duality and class of a point}

For our later computations, we will need some more precise structural statements about Chow rings of flag varieties, in particular concerning Poincar\'e duality and the generator of the top Chow group, the class of a point.

Let $H$ be a graded ring, which is a finitely generated free $\Z$-module on the homogeneous elements $s_\la$: \[H=\Z\bra s_\la: \la\in \La\ket.\] We denote the degree of $s_\la$ by $|\la|$ . We say that $s_\la$ is a \emph{Poincar\'e dual basis of dimension $n$}, if the top degree part of $H$ is generated by a single class $\om_H:=s_{\la_0}$ for some $\la_0\in \La$, and there is a bijection $\iota\colon\La\to \La$  with the property that if $|\la|+|\mu|=n$ (of complementary degree), then
\[
s_\la\cdot s_\mu=\begin{cases}
\om_H\qquad &\text{if } \mu=\iota(\la),\\
0,\qquad &\text{else.}
\end{cases}
\]
\begin{remark}
  \label{rem:ChowPDbasis}
  The most important example of a Poincar\'e dual basis is given by the Schubert cycles in the singular cohomology of complex flag varieties. This can be seen via the Bruhat decomposition of flag varieties: Schubert varieties are the $B^+$-orbit closures whose fundamental classes generate cohomology. The duality $\iota$ is given by taking the classes of the corresponding $B^-$-orbits. For Grassmannians, this is discussed in \cite[Section~4.2.2]{3264}, the general case of flag varieties is in \cite[Example~14.7.16]{Fulton}.
  
  Similarly, for an arbitrary base field $F$ and coefficient field $k$, the Chow ring ${\rm CH}^*(\Fl(\D))\otimes_{\Z}k$ of the flag variety $\Fl(\D)$ over $F$ has the same Poincar\'e dual basis as singular cohomology. One way to see this is to compare the Chow ring of $\Fl(\D)$ over $\mathbb{C}$ to singular cohomology of $\Fl(\D;\mathbb{C})$ via the cycle class map ${\rm CH}^*(\Fl(\D))\to {\rm H}^{2*}(\Fl(\D;\mathbb{C}),\mathbb{Z})$, which is an isomorphism since $\Fl(\D)$ is cellular.  
  Then Poincar\'e duality holds for ${\rm CH}^*(\Fl(\D))$ over arbitrary base fields $F$ because the Chow ring is independent of the base field. Alternatively, that Schubert cycles form a Poincar\'e dual basis can be established by purely algebraic means, cf.~\cite[Example~14.7.16]{Fulton}.
\end{remark}

The following Proposition will allow us to determine the annihilator of a certain element in the Chow ring.

\begin{proposition}
  \label{prop:PDbasis}
  Let $H=\Z\bra s_\la:\la \in \Lambda \ket$ be a graded ring which is a finitely generated free  $\Z$-module on the homogeneous elements $s_\la$. Let $I$ be a homogeneous ideal and let $C=H/I$, $K=\Ann_H(I)$. Assume that
  \begin{itemize}
  \item[i)] $(s_\la,\la\in \La)$ is a Poincar\'e dual basis for $H$ with top degree class $\om_H$,
  \item[ii)] $I=\Z \bra s_\la:\la\in \De\ket $ is a free $\Z$-module on a subset $\De\se \La$,
  \item[iii)] the reductions $(\overline{s}_\ga, \ga\in \La\su \De)$ are a Poincar\'e dual basis for $C$ with top degree class $\om_C$.
  \end{itemize}
  Then there exists a homogeneous element $x\in K$, such that $x\cdot \om_C=\om_H$ and $K$ is a free $C$-module generated by $x$, i.e.,  $K=C\bra x\ket$.
\end{proposition}

\begin{proof}
  Let $\Ga:=\La \su \De$, and denote by $\iota_C\colon\Ga\to\Ga$ and $\iota_H\colon\La\to\La$ the Poincar\'e duality pairings corresponding to $C$ and $H$. Let $C$ be of dimension $m$ and $H$ be of dimension $n$. If $\om_C=\overline s_{\ga_0}$, then  $x:=s_{\iota_H(\ga_0)}$ has degree $n-m$, and $x\cdot \om_C=\om_H$. We will show that the multiplication-by-$x$ map $\mu_x\colon C\to K$ is an isomorphism.
	
  Assume that $\mu_x(y)=x\cdot y=0$ for some $y\in C$. If $y$ is nonzero, then via the Poincar\'e dual basis in $C$, there exists $y^\vee\in C$, such that $y\cdot y^\vee=m\cdot \om_C$ for some nonzero $m\in \Z$. 
  Then we have
  \[
  x\cdot (y\cdot y^\vee)=x\cdot (m\om_C)=m\cdot \om_H,
  \]
  which would imply that $\mu_x(y)=x\cdot y$ is also nonzero, contradicting the assumption $\mu_x(y)=0$. So $y$ is zero, which proves injectivity of $\mu_x$.
	
  To show surjectivity of $\mu_x$, let $\xi\in K$ be an arbitrary homogeneous element of degree $d$. We will show that it is in the ideal $(x)$. For each $\ga\in \Ga$ of degree $d-(n-m)$, we have $|\iota_C(\ga)|=n-d$ and can define $\be_\ga\in \Z$ by \[\xi\cdot s_{\iota_C(\ga)}=\be_\ga\cdot\om_H.\]
We claim that
\[
\xi=\left(\sum_{\ga\in\Ga} \be_\ga \cdot s_\ga\right)\cdot x.
\]
Indeed, it is enough to verify that for all $\la\in \La$ of degree $|\la|=n-d$, i.e., complementary to $|\xi|$, we have
\begin{equation}\label{eq:basis_expansion}
	\xi\cdot s_{\la}=\left(\sum_{\ga\in\Ga} \be_\ga \cdot s_\ga\right)\cdot x\cdot s_\la.
\end{equation}
For $\mu\in \Ga$, this holds by Poincar\'e duality in $C$ and by definition of $\be_\mu$:
\[s_{\iota_C(\mu)}\cdot \left(\sum_{\ga\in\Ga} \be_\ga \cdot s_\ga\right)\cdot x=\be_\mu\cdot \om_C\cdot x=\be_\mu\cdot \om_H=\xi\cdot s_{\iota_C(\mu)}\]
For $\de\in \De$, $s_\de\in I$, so $\xi\cdot s_\de=0$, and so is $x\cdot s_\de=0$. Since $\La=\De\cup \Ga$, \eqref{eq:basis_expansion} holds.
\end{proof}

\begin{proposition}
	\label{prop:kernel-top-chern}
	Let $\D=(d_1,\dots,d_m)$, and let ${\rm c}:={\rm c}_{d_m}(\DD_m)$ be the top Chern class of the last subquotient bundle $\DD_m$ over $\Fl(\D)$. Then the annihilator of $\rm c$ in $\CH^*\bigl(\Fl(\D)\bigr)$ is a principal ideal
	\[\Ann_{\CH^*(\Fl(\D))} ({\rm c})=(x),\]
	where $x=\prod_{j\neq m}{\rm c}_{d_j}(\DD_j)$ is the product of the top Chern classes of all other subquotient bundles.
	
	Moreover, using the notation from Proposition~\ref{prop:PDbasis}, %Lemma~\ref{lem:PDA}, 
	denoting $K:=\ker({\rm c})$ and $C=\coker({\rm c})$, we have that $K$ is a free $C$-module of rank 1.
\end{proposition}

\begin{proof}
	We will apply Proposition \ref{prop:PDbasis} with $I=({\rm c})$. By Remark \ref{rem:ChowPDbasis}, condition i) is satisfied for the fundamental classes of Schubert varieties. Since restriction \[{\rm CH}^*(\Fl(D))\to {\rm CH}^*\bigl(\Fl(D')\bigr)\iso C\] with $D'=(d_1\stb d_m-1)$ maps Schubert classes to Schubert classes, condition iii) is satisfied if condition ii) is, again by Remark \ref{rem:ChowPDbasis}.
	
	To show ii), we use a geometric interpretation of ${\rm c}$: since $\DD_m=\mathbb{A}^N/\SS_{m-1}$, \[{\rm c}=e(\Hom(F_1,\mathbb{A}^N/\SS_{m-1}))=[\si(F_\bullet)],\]
	where $F_1$ denotes the one-dimensional subspace of the reference flag $F_\bullet$ and $[\si(F_\bullet)]$ denotes the Schubert cycle of flags \[\si(F_\bullet):=\{E_\bullet \in \Fl(\D):F_1\se E_{m-1}\}.\] It is smooth, since it can be written as a sequence of Grassmannian bundles over the smooth Schubert cycle of planes containing $F_1$ in the Grassmannian $\Gr_{s_{m-1}}(\mathbb{A}^N)$. Therefore the ideal $I=([\si(F_\bullet)])$ is additively generated by the Schubert cycles contained in $\si(F_\bullet)$, cf.~the proof of Proposition~\ref{prop:ranks}.
\end{proof}

We now recall the standard computation of the class of a point in the Chow ring of a partial flag variety $\Fl(\D)$, cf.~e.g.~\cite[(1.5)]{Fulton92}, which we will need for the annihilator computation in Section~\ref{sec:ann-euler} as well as the enumerative applications in Section~\ref{sec:enumerative}. A given point $x\in \Fl(\D)$, corresponding to a flag $F_\bullet$,  can be considered as an algebraic $0$-cycle $[x]\in{\rm CH}^{d}\bigl(\Fl(\D)\bigr)$ with $d=\dim \Fl(\D)$. Alternatively, this can be considered the fundamental class of the subvariety $\{x\}$ in $\Fl(\D)$. In the statements below, we will use the notation ${\rm c}_{\rm top}(-)$ for the top Chern class of a vector bundle.

\begin{proposition}
  \label{prop:classofpoint}
  Fix a flag $F_\bullet\in \Fl(\D)$. Then the fundamental class of the corresponding point $F_\bullet \in \Fl(\D)$ in ${\rm CH}^{\dim}\bigl(\Fl(\D)\bigr)$ is
  \begin{equation}
    \label{eq:classofpoint}
    [F_\bullet]=\prod_{i=1}^{m-1}{\rm c}_{\rm top}\bigl(\Hom(\DD_i,\A^N/F_i)\bigr),
  \end{equation}
where $\A^N/F_i$ is the quotient of the trivial bundles $\A^N$ and $F_i\inj \A^N$ over $\Fl(\D)$.
\end{proposition}

\begin{proof}
  The bundles $\Hom(\SS_i,\A^N/F_j)$ have tautological sections:
  $$ \varphi_{i,j}\colon \SS_i\to \A^N\to \A^N/F_j.$$
  The zero set of this section is the incidence set $I_{i,j}:=\{V_\bullet\in \Fl(\D):V_i\se F_j\}$. 
  Over $I_{i,j}$, the section $\varphi_{i+1,j+1}$ of the bundle $\Hom(\SS_{i+1},\A^N/F_{j+1})$ descends to a section of $\Hom(\SS_{i+1}/\SS_i,\A^N/F_{j+1})$:
  $$ \psi_{i+1, j+1}\colon \SS_{i+1}/\SS_i\to \A^N\to \A^N/F_{j+1},$$
  since $\SS_i\se F_j\se F_{j+1}$. Note that $\varphi_{1,j}=\psi_{1,j}$ since $\SS_0=(0)$. The zero sets of the sections $\psi_{i+1,i+1}|_{I_{i,i}}$ can be shown to have the expected dimension and 
  $$ [F_\bullet]=\left[\bigcap_{i=1}^{m-1}I_{i,i}\right]=\prod_{i=1}^{m-1} {\rm c}_{\rm top}\bigl(\Hom(\DD_i,\A^N/F_i)\bigr).$$
\end{proof}

\section{Presentation of \texorpdfstring{$\mathbf{W}$}{W}-cohomology and annihilators of Euler classes}
\label{sec:ann-euler}

In this section, we will define rings $W_\D$ which we will in later sections show to be isomorphic to the $\mathbf{W}$-cohomology rings of type A partial flag varieties. For now, we will discuss some of the algebraic properties of these rings. Most importantly, we will describe in Theorem~\ref{thm:annihilator} the annihilators of certain elements that will later correspond to Euler classes of subquotient bundles in the $\mathbf W$-cohomology of $\Fl(\D)$. To establish the algebraic statements, we describe $W_\D$ by iteratively adjoining square roots to elements of a corresponding Chow ring, thus reducing the annihilator computation to the corresponding statement for Chow rings in Proposition~\ref{prop:kernel-top-chern}. From these computations, we will also deduce in Proposition~\ref{prop:WDfree} that the ring $W_\D$ is free as a ${\rm W}(F)$-module.

\subsection{Algebraic presentation of $\mathbf{W}$-cohomology}

The relevant rings $W_\D$ will be defined by generators and relations. For our purposes, we define the rings as algebras over a commutative coefficient ring $R$. Most of the algebraic statements will be established first in the case where $R=\mathbb{Z}$, deducing the general case from that. In the applications to $\mathbf{W}$-cohomology, these coefficient rings will be Witt rings ${\rm W}(F)$ of the base field. 

\begin{definition}
  \label{def:WD}
  Let $R$ be a commutative ring.   Given a partition $\D=(d_1\stb d_m)$ of $N=\sum_{i=1}^md_i$, let $\Si_\D$ denote the commutative graded $R$-algebra given by generators and relations as follows:
  \[
  \Si_\D:= R\Bigg [\op{p}_{2j}(\DD_i),\op e_i \mid i=1\stb m, \;j=1\stb \Bigl\lfloor \frac{d_i}{2}\Bigr\rfloor\Bigg ]\Bigg / \sim.
  \]
  The names for the generators already indicate which cohomology classes they will relate to: the Pontryagin classes ${\rm p}_{2j}(\DD_i)$ in degree $4j$ and Euler classes ${\rm e}_i$ in degree $i$ of the subquotient bundles $\DD_i$ on $\Fl(\D)$.

  The generators are subject to the following relations $\sim$:
  \begin{align}\label{eq:p}
    \prod_{i=1}^m \op{p}(\DD_i)=0,\qquad &\\\label{eq:eodd}
    \op e_i=0, \qquad &\text{if } d_i \text{ is odd, and}\\\label{eq:eeven}
    \op e_i^2=\op{p}_{d_i}(\DD_i), \qquad &\text{if } d_i \text{ is even,}\\
    \prod_{i=1}^m\op{e}_i=0,\qquad &\text{if all }d_i\text{ are even}\label{eq:ealleven}
  \end{align}
  Here ${\rm p}(\DD_i)=1+{\rm p}_2(\DD_i)+\cdots+{\rm p}_{\rm top}(\DD_i)$ is the analogue of the total Pontryagin class of $\DD_i$, and we use the notation
  \[
    {\rm p}_{\rm top}(\DD_i):={\rm p}_{2\lfloor \frac{d_i}{2}\rfloor}(\DD_i)
  \]
  for the top Pontryagin class.
  
    Finally, let 
    \[\La:=\bigwedge\nolimits_{l=q+1}^n\langle \op{R}_l\rangle\]
    be the exterior algebra generated by the classes $\op{R}_{q+1},\dots,\op{R}_n$ with $q=\sum_{i=1}^m\lfloor \frac{d_i}{2}\rfloor$ and $n=\lfloor \frac N 2\rfloor$. The degree of $\op{R}_l$ is $4l-1$ except for the class $\op{R}_n$, which has degree $N-1$. Then we set
    \[
    W_{\D}:=\Si_\D\otimes_{R}\La.
    \]
\end{definition}

\begin{remark}
  There are several coefficient rings $R$ of interest in the above definition. In the case $R=\mathbb{Z}[1/2]$, we recover the known presentation of the half-integer coefficient singular cohomology of the real flag manifolds $\Fl(\D;\mathbb{R})$, cf.~\cite{he} or \cite{matszangosz}; similarly for $R=\mathbb{F}_p$ with an odd prime $p$. In the case $R=\mathbb{Z}$, $\Sigma_\D$ describes the torsion-free part of the singular cohomology of $\Fl(\D;\mathbb{R})$. We will also concentrate on the case $R=\mathbb{Z}$ in the analysis of the annihilators of Euler classes below. However, the freeness result we will prove in the case $R=\mathbb{Z}$, cf.~Proposition~\ref{prop:WDfree}, will imply that our results (about annihilators of ${\rm e}_i$ and freeness as $R$-module) will be true for arbitrary coefficient rings. For our applications, the case $R={\rm W}(F)$ of Witt rings will be most relevant. We will show that the $\mathbf{W}$-cohomology of flag varieties $\Fl(\D)$ is described by $\Si_\D$ with ${\rm W}(F)$-coefficients, cf.~Section~\ref{sec:maximal} and in particular Theorem~\ref{thm:maxrank} for the maximal rank case and Section~\ref{sec:sadykov} and in particular Proposition~\ref{prop:proof-main-thm} for the general case. Our methods also allow to provide new proofs of the known topological cases.
\end{remark}

\begin{remark}
  The grading defined above will correspond to the grading by cohomological degree. In fact, the $R$-algebra has a finer grading by the mod 2 Picard group $\op{Pic}\bigl(\Fl(\D)\bigr)/2\cong \mathbb{Z}/2\mathbb{Z}^{\oplus (m-1)}$. In this grading, the classes ${\rm p}_{2j}(\DD_i)$ have degree $0$, and the classes ${\rm e}_i$ (whenever they are nonzero) have degree $\det\DD_i^\vee$. 
\end{remark}

\begin{remark}\label{rmk:Sigmaodd}
  Let $R=\mathbb{Z}$. For any $D=(d_1,\dots,d_m)$, the subalgebra of $\Si_\D$ generated by the classes ${\rm p}_{2j}(\DD_i)$ (alternatively, the degree $0$ part for the Picard-group grading of the previous remark) is isomorphic to the Chow ring of another flag variety $\Fl(\D')$, for the partition $\D'=\bigl(\lfloor d_1/2\rfloor\stb \lfloor d_m/2\rfloor\bigr)$. The isomorphism $\CH^j\bigl(\Fl(\D')\bigr)\to \Si^{4j}_\D$ sends ${\rm c}_j(\DD_i)$ to ${\rm p}_{2j}(\DD_i)$ and multiplies cohomological  degrees by 4. This follows directly from the presentation of the Chow ring, cf.~\eqref{eq:chowringFlD}. In the case where all $d_i$ are odd, there are no classes ${\rm e}_i$, hence the $\mathbb{Z}$-algebra $\Si_\D$  is isomorphic to the Chow ring of $\Fl(\D')$ with $\D'=\bigl(\lfloor d_1/2\rfloor\stb \lfloor d_m/2\rfloor\bigr)$.

  Note however, that in general $\Si_\D\otimes_{\mathbb{Z}}{\rm W}(F)$ is not isomorphic to the $\mathbf{W}$-cohomology ring of $\Fl(\D)$, unless $\La\cong R$. This happens in the maximal rank cases where at most one $d_i$ is odd. 
\end{remark}

\subsection{Annihilators of Euler classes}

The crucial algebraic input for our computation of ${\rm H}^*(\Fl(\D),\mathbf{W})$ is a description of the kernel and cokernel of multiplication by the Euler class of a subquotient bundle. This section is devoted to proving the following theorem, which describes the annihilator of ${\rm e}_m$ in $\Si_\D$.

In this and the following three subsections, the coefficient ring is $R=\mathbb{Z}$. We will discuss in Section~\ref{sec:freeness} that the main result, Theorem~\ref{thm:annihilator}, is in fact true for arbitrary coefficient rings.

\begin{theorem}\label{thm:annihilator}
  Let $\D=(d_1\stb d_m)$  with $d_m$  even. Then $\Ann {\rm e}_m\se W_{\D}$ is a principal ideal $(x)$, with a parity case distinction for the generator $x$ as follows:
  \begin{itemize}
  \item[(even)] $x=\prod_{j\neq m}\op{e}_j$ if all $d_j$ are even, and
  \item[(odd)] $x=\op{e}_m\cdot\prod_{j\neq m}\op{p}_{\rm top}(\DD_j)$ otherwise.
  \end{itemize}
  In the even case, $x$ is an element of degree $N-d_m$ and in the odd case $x$ is an element of degree $4q-d_m$, where $N=\sum_{i=1}^m d_i$ and $q=\sum_j \lfloor \frac{d_j}{2}\rfloor$.
\end{theorem}

\begin{proof}[Overview of proof]
  Recall that $W_\D$ is defined as tensor product $\Si_\D\otimes_{\Z}\La$. Since $\Lambda$ is a free module and multiplication by $\op{e}_m$ on $W_\D$ is multiplication by $\op{e}_m\otimes 1$ in $\Si_\D\otimes_{\Z}\La$, it is enough to determine $\ker \op{e}_m\se \Si_\D$.
  
  The proof for $\Si_\D$ consists of three steps. In the first step, when all $d_i\in\D$ are odd, we show that \[\Ann_{\Si_\D}\bigl({\rm p}_{\rm top} (\DD_m)\bigr)=(x),\textrm{ for } x=\prod_{j\neq m}{\rm p}_{\rm top}(\DD_j),\] 
  see Proposition~\ref{prop:kernel-top-pontryagin}.

  In the second step, we reduce the (odd) case to the situation without Euler classes, using that we get Euler classes by adjoining square roots of top Pontryagin classes. More precisely, for $\D=(d_1,\dots,d_m)$ and $\D'=(d_1\stb d_{i-1},d_i-1, d_{i+1}\stb d_m)$ for some odd $d_i\geq 3$, we have
  \begin{equation}\label{eq:recursion}
  	\Si_{\D'}=\Si_\D[\e_i]/\bigl(\e_i^2-\p_{\rm top}(\DD_i)\bigr),
  \end{equation}
  by the presentation in Definition~\ref{def:WD}. We then show in Lemma~\ref{lem:adjoiningeuler} that in such a situation we have
  \[\Ann_{\Si_{\D'}}(\e_m)=\e_m\cdot \Ann_{\Si_{\D}}\hspace{-0.1cm}\bigl(\p_{\rm top}(\DD_m)\bigr),\]
  which covers the (odd) case, see Corollary~\ref{cor:oddcase}. 
  
  For the final step, assume that $\D''=(d_1\stb d_m)$ is such that all $d_i$ are even. Then setting $\D=(d_1+1,d_2+1\stb d_m+1)$ and 
  \begin{equation}\label{eq:evenrecursion1}
    A':=\Si_{\D}[\e_1\stb \e_m]/\bigl(\e_i^2-\p_{\rm top}(\DD_i)\mid i=1\stb m\bigr)
  \end{equation}
  we have 
  \begin{equation}\label{eq:evenrecursion2}
    \Si_{\D''}=A'/(\e_1\cdot\ldots\cdot \e_m).
  \end{equation}
  We show in Proposition~\ref{prop:ChernPIQP} that the elements $\p_{\rm top}(\DD_i)\in \Si_{\D}$ satisfy a principal ideal-quotient property (see Definition~\ref{def:PIQP}), which guarantees the equality 
  \[\bigl((\e_1\cdot \ldots \cdot \e_m):(\e_m)\bigr)=(\e_1\cdot \ldots \cdot \e_{m-1})\]
  of ideal in $A'$, with the left-hand side the ideal quotient of two principal ideals in $A'$. The image of the left-hand side under $A'\to \Si_{\D''}$ is mapped to the annihilator, cf.\ Corollary~\ref{cor:PIQP}, proving the (even) case, see Corollary~\ref{cor:evencase}. 
\end{proof}

\begin{remark}
  By symmetry, the statement of Theorem~\ref{thm:annihilator} also holds for other classes ${\rm e}_i$. For every $i$ such that $d_i$ is even (so that ${\rm e}_i$ is nonzero in $\Si_\D$), we have $\ker {\rm e}_i=(x)$, with the obvious modifications for the generator $x$.
\end{remark}

The results necessary for the above proof are now established in the following three subsections.

\subsection{The odd case}

\begin{proposition}\label{prop:kernel-top-pontryagin}
  Let $\D=(d_1\stb d_m)$, with all $d_i$ odd. Then
  \[\Ann_{\Si_\D}\bigl({\rm p}_{\rm top} (\DD_m)\bigr)=(x),\] 
  where $x=\prod_{j\neq m}{\rm p}_{\rm top}(\DD_j)$.
\end{proposition}

\begin{proof}
  Let $\D_0=\bigl(\lfloor d_1/2\rfloor\stb \lfloor d_m/2\rfloor\bigr)$. As mentioned before, the coefficient ring is $R=\mathbb{Z}$. By Remark~\ref{rmk:Sigmaodd}, there is an isomorphism of graded rings multiplying degrees by 4
  \begin{equation}\label{eq:kappa}
  	  \ka\colon\CH^*\bigl(\Fl(\D_0)\bigr)\to \Si_{\D},
  \end{equation}
  which maps  $\op{c}_{j}(\DD_i)$ to $\p_{2j}(\DD_i)$ for all $i,j$ and in particular maps  $\op c:=\op{c}_{\lfloor d_m/2\rfloor}(\DD_m)$ to $\p_{d_m-1}(\DD_m)$. The annihilator $\Ann_{\CH^*(\Fl(\D_0))}(\op c)$ is described by Proposition~\ref{prop:kernel-top-chern} which via the isomorphism $\ka$ gives the desired statement.
\end{proof}

Since the ring $\Si_\D$ for $\D$ odd is obtained from the Chow-ring by adjoining square roots of top Chern classes (via the isomorphism $\ka$ of \eqref{eq:kappa}), we describe how this operation modifies the annihilator.

\begin{lemma}
  \label{lem:adjoiningeuler} 
  Let $A$ be a commutative graded ring. Let $\p\in A_{2k}$ be an even, positive degree element and let $A'=A[\e]/(\e^2-\p)$ be obtained by adjoining a square root $\e$ of $\p$ (of degree $k$). Then
  \[
  \Ann_{A'}(\e)=\Ann_A(\p)\cdot \e,\qquad \Ann_{A'}(y)=\bigl(\Ann_{A}(y)\bigr)_{A'}=\Ann_{A}(y)\oplus \Ann_{A}(y)\cdot \e
  \]
  for all $y\in A$. If $A$ is a free $\Z$-module, then $A'$ is also a free $\Z$-module.
\end{lemma}

\begin{proof}
  Write $A'=A\oplus A\bra \e\ket,$ as a direct sum of $A$-modules, where $A\bra \e\ket=\{a\cdot \e\mid a\in A\}$. In particular, $A'$ is a free $A$-module, which shows the last claim. To understand $\Ann_{A'}(\e)$, let $x=a+b\cdot \e\in A'$. Then $x\cdot\e=(a+b\cdot \e)\cdot\e=a\cdot \e+b\cdot\p$, so $x\in\Ann_{A'}(\e)$ if and only $a=0$ and $b\in\Ann_A(\p)$. As a consequence, we get a direct sum decomposition of $\Ann_{A'}(\e)$ into the terms
  \[\Ann_{A'}(\e)\cap A=(0),\qquad \text{ and }\qquad \Ann_{A'}(\e)\cap A\bra \e\ket=\Ann_A(\p)\cdot \e.\qedhere\]
\end{proof}

\begin{corollary}\label{cor:oddcase}
  Let $\D=(d_1\stb d_m)$ be such that at least one $d_i$ is odd and $d_m$ is even. Then the annihilator of $\e_m$ is a principal ideal
	\[
	\Ann_{\Si_\D}(\e_m)=(x),
	\]
	where $x=\e_m\cdot\prod_{j\neq m}\p_{\rm top}(\DD_j)$.
\end{corollary}

\begin{proof}
  Let $\D_0=(d_1'\stb d_m')$, where $d_i'=d_i+1$ if $d_i$ is even, and $d_i'=d_i$ if $d_i$ is odd (this has the effect of removing all the classes $\e_i$). Then by Definition~\ref{def:WD},
  \[
  \Si_{\D}=\Si_{\D_0}[\e_i]_{i\in E}/\bigl(\e_i^2-\p_{\rm top}(\DD_i)\bigr),
  \]
  where $E=\{i\mid d_i \text{ even}\}$. Let $\D_1=(d_1'\stb d_{m-1}',d_m)$. Then by Lemma~\ref{lem:adjoiningeuler} and Proposition~\ref{prop:kernel-top-chern}:
  \[\Ann_{\Si_{\D_1}}(\e_m)=\Ann_{\Si_{\D_0}}\bigl(\p_{\rm top}(\DD_m)\bigr)\cdot \e_m=\left(x\right),\]
  where $x=\e_m\cdot\prod_{j\neq m}\p_{\rm top}(\DD_j)$. Let $\D_2=(d_1'\stb d_{m-2}',d_{m-1}, d_m)$. Assume that $d_{m-1}'=d_{m-1}+1$, otherwise $\Si_{\D_1}$ and $\Si_{\D_2}$ are equal. Then again, \[\Si_{\D_2}=\Si_{\D_1}[\e_{m-1}]/\bigl(\e_{m-1}^2-\p_{\rm top}(\DD_{m-1})\bigr),\]
  and by Lemma~\ref{lem:adjoiningeuler}
  \[\Ann_{\Si_{\D_2}}(\e_m)=\bigl(\Ann_{\Si_{\D_1}}(\e_m)\bigr)=(x).\]
  as a principal ideal in $\Si_{\D_2}$ now.
  Repeating this process, replacing each $d_i'$ with $d_i$ concludes the proof of the corollary. Note that in all steps of the process there is at least one odd $d_i$ by assumption (and the fact that the modifications leave the odd $d_i$'s unchanged), so that the relation~\eqref{eq:ealleven} never appears.
\end{proof}

\subsection{Interlude: principal ideal--quotient property}

We consider the following situation: let $A$ be a commutative unital ring, $\p_i\in A$, $i=1\stb m$ be elements such that $\Ann_A(\p_m)=(\p_1\cdots\p_{m-1})$, and let
\[
A':=A[\e_1\stb \e_m]/\bigl(\e_i^2-\p_i\mid i=1\stb m\bigr),
\]
and
\[
A''=A'/(\e_1\cdot\ldots\cdot \e_m).
\]
Without additional assumptions on $\p_i$, it is not true that $\Ann_{A''}(\e_m)=(\e_1\cdot \ldots \cdot \e_{m-1})$; for instance if $\p_m=0$, then $\e_m\in \Ann_{A''}(\e_m)$.

In this section, we discuss a property, the \emph{principal ideal--quotient property},  which ensures that an annihilator computation as above holds. This will reduce the proof of Theorem~\ref{thm:annihilator} to checking the principal--ideal quotient property for Chow rings of full flag varieties, which we will then do in the next subsection. The property is formulated in terms of ideal quotients. Recall that for ideals $I,J\subseteq A$, the ideal quotient is defined as
\[
(I:J):=\{a\in A\mid aJ\subseteq I\}.
\]

\begin{definition}\label{def:PIQP}
  A set of elements $x_1\stb x_n$ in a commutative ring satisfies the \emph{principal ideal--quotient property} if 
  \[\bigl((x_I):(x_J)\bigr)=(x_{I\su J})\]
  for all index sets $J\se I\se\{1,\dots,n\}$, where $x_I=\prod_{i\in I}x_i$ with the convention $x_{\emptyset}=1$.
\end{definition}
For principal ideals, we will alleviate notation by writing $(a:b)$ instead of $\bigl((a):(b)\bigr)$.

\begin{lemma}
  \label{lem:piqp}
  Let $A$ be a commutative ring and let
  \[
  A'=A[\e_1\stb \e_m]/\bigl(\e_i^2-\p_i\mid i=1\stb m\bigr)
  \]
  for some elements $\p_i\in A$. Assume that in $A$ the elements $\p_1\stb \p_m$ satisfy the principal ideal--quotient property. Then in $A'$, the elements $\e_1\stb \e_m$ also satisfy the principal ideal--quotient property.
\end{lemma}

\begin{proof}
  The ring $A'$ decomposes as a free $A$-module as
  \begin{equation}\label{eq:directsum}
  A'=\bigoplus_{I\se \{1\stb m\}}A\left\bra \e_I\right\ket,  	
  \end{equation}
  where $\e_I=\prod_{i\in I} \e_i$. Notice that 
  \begin{equation}\label{eq:ringstructure}
    \e_I\cdot \e_J=\e_{I\symmdiff J}\cdot \p_{I\cap J},
  \end{equation}
  where $\symmdiff$ denotes symmetric difference. This way, the ring $A'$ has the structure of a $G$-graded $A$-algebra, where $G$ is the abelian group whose elements are subsets $I\subseteq\{1\stb m\}$ with group structure given by symmetric difference. For this grading, the element $e_I$ has degree $I$. 
  
  Fix $J\se I\se \{1\stb m\}$; the ideal quotient $E:=(\e_I:\e_J)$ in $A'$ is a $G$-homogeneous ideal and therefore decomposes under the direct sum \eqref{eq:directsum}. Let $a\e_K$ be a $G$-homogeneous element in $E$ for some $a\in A$ and $K\se \{1\stb m\}$. To show the principal ideal--quotient property we want to show that $a\e_K\in (\e_{I\su J})$. By definition of $E$, and the decomposition \eqref{eq:directsum}, there exists some $b\in A$ and $L\se \{1\stb m\}$ such that
  \[
  \e_J\cdot a \e_K=b\e_L\cdot \e_I.\]
  By using equation~\eqref{eq:ringstructure} describing the multiplication of Euler class products, this is equal to
  \[a\cdot \p_{J\cap K}\cdot \e_{J\symmdiff K}=b\cdot \p_{L\cap I}\cdot \e_{L\symmdiff I}\]
  From the direct sum decomposition of $A'$, we get that
  \[
  J\symmdiff K=L\symmdiff I,\qquad a\cdot \p_{J\cap K}=b\cdot \p_{L\cap I}
  \]
  which using $J\se I$ implies $J\cap K\se I\cap L$. Using the principal ideal--quotient property for the $\p_i$'s, this implies that $a\in (\p_{L\cap I\su J\cap K})\se (\e_{L\cap I\su J\cap K})$. A simple verification shows that
  \[I\su J\se K\cup (L\cap I\su J\cap K),\] 
  so since $a\in(\e_{L\cap I\su J\cap K})$, we get
  $a\e_K\in (e_{K\cup (L\cap I\su J\cap K)})\se  (\e_{I\su J})$, which is exactly what we wanted to prove.	
\end{proof}

\begin{corollary}\label{cor:PIQP}
  Let $A$ be a ring and let
  \[
  A'=A[\e_1\stb \e_m]/\bigl(\e_i^2=\p_i\bigr)
  \]
  for some elements $\p_i\in A$ satisfying the principal ideal--quotient property. Let $A''=A'/(\e_1\cdot \e_2\cdot \ldots \cdot \e_m)$. Then
  \[\Ann_{A''}(\e_m)=(\e_1\cdot \e_2\cdot \ldots \cdot \e_{m-1})_{A''}\]	
\end{corollary}

\begin{proof}
  In $A'$, by Lemma~\ref{lem:piqp} and the definition of the principal ideal--quotient property, we have
  \[\bigl((\e_1\cdot \ldots \cdot \e_m):(\e_m)\bigr)_{A'}=(\e_1\cdot \ldots\cdot \e_{m-1})_{A'}\]
  The ideal quotient $\bigl(0:(\e_m)\bigr)_{A''}$ is by definition the image of the ideal quotient $\bigl((\e_1\cdot \ldots \cdot \e_m):(\e_m)\bigr)_{A'}$:
  \begin{eqnarray*}
    \{x\in A''\mid x\e_m=0\}&=&\Big\{q(y)\in A''\mid y\in A', y\cdot \e_m\in (\e_1\cdot \ldots \cdot \e_m)\Big\}\\
    &=&q\bigl((\e_1\cdot \ldots \cdot \e_m):\e_m\bigr)_{A'}
  \end{eqnarray*}
  where $q\colon A'\to A''$ is the natural quotient map.
\end{proof}

\begin{proposition}\label{prop:idealintersection}
  Let $A$ be a ring and let $A'=A[\e_1\stb \e_m]/(\e_i^2-\p_i)$ for some elements $\p_i\in A$ satisfying the principal ideal--quotient property. Then 
  \[A\bra \e_I\ket\cap (\e_J)_{A'}=(\p_{J\su I})_A\,\bigl\bra \e_I\bigr\ket.\]
\end{proposition}

\begin{proof}
  It is enough to prove the containment $\se$, since the other containment is immediate from $\p_{J\su I}\e_I=\e_{J\su I}^2\e_I\in (\e_J)_{A'}$.  Assume that $c\e_K\cdot \e_J=b\e_I$ for some $b,c\in A$, $K\se \{1,\dots, m\}$; we will show that $b\in (\p_{J\su I})_A$. Using the ring structure description of equation~\eqref{eq:ringstructure}, we get $c\p_{K\cap J} \e_{J\symmdiff K}=b\e_I$, and therefore $c\p_{K\cap J}=b$ and $J\symmdiff K=I$. From this we obtain $\p_{J\cap K}=\p_{J\su I}$, proving the claim.
\end{proof}

\begin{corollary}
  \label{cor:quotientdecomposition}
  In the situation of Proposition~\ref{prop:idealintersection}, the ring $A''=A'/(\e_1\cdots \e_m)$ decomposes as an $A$-module as
	\[
	A''=\bigoplus_{I\se \{1,\dots,m\}}\bigl(A/(\p_{I^{\rm c}})_A\bigr)\bra \e_I\ket,
	\]
	where $I^{\rm c}$ denotes the complement of $I\se \{1,\dots,m\}$.
\end{corollary}

\begin{proof}
  Consider the quotient map $q\colon A'\to A''=A'/(\e_1\cdots \e_m)_{A'}$. Using the decomposition \eqref{eq:directsum}, $q$ decomposes as a graded homogeneous map over the subsets $I\se \{1\stb m\}$ as follows:
  \[
  A\bra \e_I\ket \to A\bra \e_I\ket / \bigl( A\bra \e_I\ket \cap (\e_1\cdots \e_m)_{A'}\bigr)
  =
  A\bra \e_I\ket \Bigl/ (\p_{I^c})_A\,\bra \e_I\ket=\bigl(A/(\p_{I^c})_A\bigr)\bra \e_I\ket,
  \]
  where in the first equality we use Proposition \ref{prop:idealintersection}.
\end{proof}

\subsection{Principal ideal--quotient property in Chow rings, and proof of the even case}

By Corollary~\ref{cor:PIQP} it is enough to prove that the top Chern classes satisfy the principal ideal--quotient property to conclude the proof of the even case in Theorem~\ref{thm:annihilator}. To prove the statement, we use a splitting-principle-type argument to reduce to the case of complete flag varieties, cf.\ Proposition~\ref{prop:ChernPIQP}. In the case of complete flag varieties, the principal ideal--quotient property will be established by explicit Schubert calculus computations, cf.~Proposition~\ref{prop:PIQPFln}. 

First, we establish one step of the splitting-principle-type reduction for the principal ideal--quotients, which is the setting of the projective bundle theorem.

\begin{lemma}\label{lemma:projective}
  Let $A$ be a commutative ring with unit and \[B=A[x]\bigg/\left(\sum_{i=0}^n c_ix^{n-i}\right)\]
  with $c_i\in A$, $c_0=1$. If $I$ and $J$ are ideals in $A$, then the ideal quotients satisfy
  \[
  (I:J)B=(IB:JB).
  \]
\end{lemma}

\begin{proof}
  If $\al J\se I$ for $\al\in A$, then $\al JB\se \al IB$, which proves the containment $\subseteq$.	

  For the other direction, note that the condition $c_0=1$ implies that $B$ is a free $A$-module of rank $n$ on $1\stb x^{n-1}$. Assume that $\beta\in (IB:JB)$, i.e., $\be J\se IB$ for some $\be=\al_ix^i+\cdots +\al_{n-1}x^{n-1}$, with $\alpha_j\in A$ and $\al_i\neq 0$. We compare coefficients. For $\lambda\in J$ we have $\beta\lambda=\alpha_i\lambda x^i+\cdots+ \alpha_{n-1}\lambda x^{n-1}$, while $\beta\lambda\in IB$ means that $\alpha_j\lambda\in I$ for all $j$ because of the freeness of $B$ as $A$-module. We get that for every $\lambda\in J$, $\alpha_j\lambda\in I$ for all $j$ and so $\alpha_j\in(I:J)$ for all $j$. But this means that $\beta\in (I:J)B$.
\end{proof}

\begin{proposition}\label{prop:ChernPIQP}
  Fix $\D=(d_1,\dots,d_m)$. The top Chern classes $\cc_{d_i}(\DD_i)$ satisfy the principal ideal--quotient property in $\CH^*\bigl(\Fl(\D)\bigr)$.
\end{proposition}

\begin{proof}
  Without loss of generality, it is enough to show that
\begin{equation}\label{eq:chern_quotient}
  \left(\prod_{j=1}^a \cc_{d_j}(\DD_j): \prod_{j=1}^b \cc_{d_j}(\DD_j)\right)=\left(\prod_{j=b+1}^a \cc_{d_j}(\DD_j)\right)
\end{equation}
  for any $1\leq b<a\leq m$. Indeed, permutations of $\D$ yield isomorphic rings.

  For $N=\sum_{i=1}^md_i$, we use $\Fl(N):=\Fl(\D')$ to denote the complete flag variety, corresponding to the partition $\D'=(1,\dots,1)$ of $N$. By the classical splitting principle, the natural projection $\pi\colon \Fl(N) \to\Fl(\D)$ (forgetting the relevant subspaces of a flag) induces an inclusion on the Chow rings
  \[
  \pi^*\colon\CH^*\bigl(\Fl(\D)\bigr)\inj \CH^*\bigl(\Fl(N)\bigr)=\Z[x_1\stb x_N]\bigg/\left(\prod_{i=1}^N(1+x_i)=1\right),
  \]
  where $x_i=\cc_1(\DD_i')$ is the class of the $i$th subquotient line bundle $\DD_i'$ over $\Fl(N)$. In fact, this projection can be written as a sequence of projective bundles over the base space, each satisfying a projective bundle theorem. The induced pullback maps the top Chern classes to
  \[\pi^*\cc_{d_i}(\DD_i)=\prod_{j=s_{i-1}+1}^{s_i} x_j,\]
  where $s_i=\sum_{j=1}^{i}d_i$ with $s_0=0$. 
  
  Since $\pi^*$ is injective and by the above remarks $\pi$ can be written as a sequence of projective bundles, by repeatedly using Lemma \ref{lemma:projective}, it is enough to show that the $\pi^*$-image of \eqref{eq:chern_quotient} holds, namely
\begin{equation}\label{eq:piqpfln}
  \left(\prod_{j=1}^{s_a} x_j: \prod_{j=1}^{s_b} x_j\right)=\left(\prod_{j=s_b+1}^{s_a} x_j\right).
\end{equation}
  This follows from Proposition~\ref{prop:PIQPFln} below.
\end{proof}

As a consequence, we get the description of the annihilator of Euler classes in the even case:

\begin{corollary}
  \label{cor:evencase}
  Let $\D=(d_1\stb d_m)$ with all $d_i$ even. Then the annihilator of $\e_m$ is a principal ideal
  \[
  \Ann_{\Si_\D}(\e_m)=(x),
  \]
  where $x=\prod_{j\neq m}\e_j$.
\end{corollary}

\begin{proof}
  This follows from Corollary~\ref{cor:PIQP}. The base ring $A$ is $\CH^*\bigl(\Fl(\D/2)\bigr)$ with $\D/2=(d_1/2\stb d_m/2)$, and the ${\rm p}_i$ are the top Chern classes $\op{c}_{d_i}(\DD_i)$ in $\CH^*\bigl(\Fl(\D)\bigr)$. The principal ideal--quotient property for the Chow ring follows from Proposition~\ref{prop:ChernPIQP}.
\end{proof}

To conclude the proof of Proposition~\ref{prop:ChernPIQP}, it remains to prove \eqref{eq:piqpfln}, which will finally complete the proof of the (even) case.

\subsubsection{Computations in $\Fl(N)$}\label{sec:piqpfln}

The rest of this section consists of proving Proposition~\ref{prop:PIQPFln}, which establishes the relevant case of the principal ideal--quotient property for
\begin{equation}\label{eq:CHFln}
  \CH^*\bigl(\Fl(N)\bigr)\cong\Z[x_1\stb x_N]\bigg/\left(\prod_{i=1}^N(1+x_i)=1\right).
\end{equation}

We will use the notation $[i]:=\{1\stb i\}$ for initial subsets of $[N]$.\footnote{And we are aware that this conflicts with the usual use of $[i]$ to denote the ordinal $0<1<\cdots<i$.} Using the product notation from Definition~\ref{def:PIQP}, we have $x_{[i]}=\prod_{j=1}^ix_j$.

To establish the principal ideal--quotient property, we will consider the special Schubert varieties $\si_i:=\si_{w_i}$ in the complete flag variety $\Fl(N)$, corresponding to the word $w_i=t_1t_2\cdots t_i\in S_N$, with $t_i$ the elementary transpositions. These Schubert varieties are smooth. Note that the class of $\si_i$ in $\CH^*\bigl(\Fl(N)\bigr)$ is $x_{[i]}$, cf.\ \cite{Fulton92}. 

\begin{proposition}
  \label{prop:PIQPFln}
With the above notation, for every $1\leq j<i\leq m$, the ideal quotient is a principal ideal:
  \[
  \left(x_{[i]}: x_{[j]}\right) = \left( x_{[i]\su [j]}\right).
  \]
\end{proposition}

\begin{proof}
  The containment $\left(x_{[i]}: x_{[j]}\right)\supseteq  \left( x_{[i]\su [j]}\right)$ clearly holds.

  For the other containment, we need some preparation. We consider the inclusions $f\colon \sigma_i\hookrightarrow \sigma_j$ and $g\colon \sigma_j\hookrightarrow \Fl(N)$ of special Schubert varieties in $\Fl(N)$. These are proper morphisms of smooth varieties. The induced pushforward maps $f_*$ and $g_*$ are injective (the basis of Schubert cells is mapped to a part of the Schubert cell basis), and we have $x_{[j]}=g_*(1)$ and $x_{[i]}=(g\circ f)_*(1)$. Moreover, for any $\alpha\in\CH^*(\sigma_i)$, the projection formula implies
  \[
  g_*\bigl(g^*(g_*(\alpha)^\vee)\cdot \alpha\bigr)=g_*(\alpha)^\vee\cdot g_*(\alpha)=1,
  \]
  and injectivity of $g_*$ implies $g^*\bigl(g_*(\alpha)^\vee\bigr)=\alpha^\vee$. In particular, the restriction morphism $g^*\colon\CH^*\bigl(\Fl(N)\bigr)\to\CH^*(\sigma_i)$ is surjective. Analogous statements are true for the morphisms $f$ and $g\circ f$. 

  Now assume that $\al\in (x_{[i]}:x_{[j]})$, i.e., there exists $\beta\in \CH^*\bigl(\Fl(N)\bigr)$ such that $\alpha\cdot g_*(1)=\beta\cdot (g\circ f)_*(1)$. From this, the projection formula implies
  \[
  g_*\bigl(g^*(\alpha)\bigr)=\alpha\cdot g_*(1) = \beta\cdot (g\circ f)_*(1) = g_*\bigl(g^*(\beta)\cdot f_*(1)\bigr).
  \]
  By definition of $x_{[i]\setminus[j]}$, we have
  \[
  g_*\bigl(g^*(x_{[i]\setminus[j]})\bigr)=x_{[i]\setminus[j]}\cdot g_*(1)=x_{[i]\setminus[j]}\cdot x_{[j]}=x_{[i]}=(g\circ f)_*(1),
  \]
  where in the first equality we have used the projection formula. From the injectivity of $g_*$, we get 
  $g^*(x_{[i]\setminus[j]})=f_*(1)$, and we can conclude $g_*\bigl(g^*(\alpha)\bigr)=g_*\bigl(g^*(\beta\cdot x_{[i]\setminus[j]})\bigr)$. Since $g_*$ is injective (as noted above), we get $g^*(\alpha)=g^*(\beta\cdot x_{[i]\setminus[j]})$. In particular,
  \[
  \alpha-\beta x_{[i]\setminus[j]}\in \ker g^*.\]
  In Corollary~\ref{cor:AnnFln} below, we will show $\ker g^*\se(x_{[i]\su[j]})$, which finally implies $\alpha\in (x_{[i]\su[j]})$, as required. 
\end{proof}

First, we use Schubert calculus to identify the ranks of the relevant ideals as $\Z$-modules.

\begin{proposition}\label{prop:ranks}
  Denote by $\si_i:=\si_{w_i}$ the special Schubert variety corresponding to $w_i=t_1 t_2\cdots t_i\in S_{N}$. Then the principal ideal of $\CH^*\bigl(\Fl(N)\bigr)$ generated by the class $[\si_i]$ has a $\Z$-module basis:
  \[
  \bigl([\si_i]\bigr)=\Z\bra [\si_w]\mid w\leq w_i \ket.
  \]
  Consequently, $\bigl([\si_i]\bigr)$ is a free $\Z$-module of rank $(N-i)\cdot (N-1)!$.
\end{proposition}

\begin{proof}
  The special Schubert variety $\si_i\subseteq\Fl(N)$ consists of those flags $\{F_\bullet\mid L\se F_i\}$ whose $i$th term contains a given line $L$. The variety $\si_i$ is smooth and has a cellular structure given by the Schubert cells contained in it. In particular, its Chow ring is generated by the cell closures. 

We denote by $f\colon \si_i\hookrightarrow \Fl(N)$ the inclusion, which is a proper morphism of smooth varieties (in particular $f_*$ and $f_!$ coincide). The restriction $f^*\colon\CH^*\bigl(\Fl(N)\bigr)\to \CH^*(\si_i)$ along $f$ is surjective, which was noted previously in the proof of Proposition~\ref{prop:PIQPFln}. Using the projection formula, we can write elements of the ideal $\bigl([\si_i]\bigr)$ as $\la\cdot[\si_i]=f_*\bigl(f^*(\lambda)\bigr)$ for $\la\in\CH^*\bigl(\Fl(N)\bigr)$. Since $f^*$ is surjective, the ideal $\bigl([\si_i]\bigr)$ is equal to $\op{Im} f_*$, which is exactly the $\Z$-module of Schubert cycles contained in $\si_i$.
  
  For the second part, note that in one-line notation $v=[v(1)\stb v(N)]$, we can write $w_i=[2,3\stb i, 1, i+1\stb N]$, cf.\ e.g.\ Fulton \cite{Fulton92}. Using the Bruhat order, one obtains that $v\leq w_i$ if and only if $v^{-1}(1)\geq i$. The number of these permutations is easily seen to be $(N-i)\cdot (N-1)!$.
\end{proof}

The following corollary will be used later on, to prove that $W_\D$ is free.
\begin{corollary}
  \label{cor:quotientfree}
  For a subset $I\subseteq[m]$ define ${\rm c}_I=\prod_{i\in I}\op{c}_{d_i}(\DD_i)$. Then the quotient $\CH^*\bigl(\Fl(\D)\bigr)/({\rm c}_I)$ is a free $\Z$-module.
\end{corollary}

\begin{proof}
  Using the pullback 
  \[
  \pi^*\colon\CH^*\bigl(\Fl(\D)\bigr)\inj \CH^*\bigl(\Fl(N)\bigr)\cong\Z[x_1\stb x_N]\bigg/\left(\prod_{i=1}^N(1+x_i)=1\right)
  \]
  induced by $\pi\colon \Fl(N)\to \Fl(\D)$, it is enough to prove that $\CH^*\bigl(\Fl(N)\bigr)/(x_1\cdots x_j)$ is a free $\Z$-module. More precisely, since $\bigl(\prod_{i=1}^m\op{c}_{d_i}(\DD_i)\bigr)=(x_1\cdots x_j)\cap\op{Im}(\pi^*)$, the pullback map $\pi^*$ induces an injective morphism
  \[
  \CH^*\bigl(\Fl(\D)\bigr)\bigg/\left(\prod_{i=1}^m\op{c}_{d_i}(\DD_i)\right)\inj \CH^*\bigl(\Fl(N)\bigr)/(x_1\cdots x_j),
  \]
  hence freeness of the submodule follows from freeness of $\CH^*\bigl(\Fl(N)\bigr)/(x_1\cdots x_j)$. We can identify $(x_1\cdots x_j)=\bigl([\sigma_j]\bigr)$. Since $\CH^*\bigl(\Fl(N)\bigr)$ has a $\Z$-module basis given by Schubert classes, by Proposition~\ref{prop:ranks} and $([\si_j])$ is a submodule generated by Schubert classes, the quotient $\CH^*\bigl(\Fl(N)\bigr)/\bigl([\si_j]\bigr)$ is also a free $\Z$-module.
\end{proof}

To conclude the proof of Proposition~\ref{prop:PIQPFln}, it remains to show that $\ker g^*\se (x_{[i]\su [j]})$. In the next proposition and its corollary, we show that $\ker g^*=(x_{[N]\su [j]})$, immediately implying this statement. This concludes the proof of Propositions~\ref{prop:PIQPFln} and \ref{prop:ChernPIQP}, and thus the proof of Theorem~\ref{thm:annihilator}.

\begin{proposition}
  \label{prop:FlnAnn}
  Using the notation introduced before Proposition~\ref{prop:PIQPFln}, we have the following equality in the Chow ring of $\Fl(N)$:
  \[
  \Ann(x_{[i]})=(x_{[N]\su [i]}).
  \]
\end{proposition}

\begin{proof}
	Since $x_{[N]}=\prod_{i=1}^n x_i=0$, $(x_{[N]\su [i]})\se \Ann(x_{[i]})$. To show that we have equality, we will use the following general fact: if $M\se N$ is an inclusion of submodules of a finitely generated $\mathbb{Z}$-module, and if for all primes $p\in \Z$, the mod $p$ reductions $M\otimes_\Z \mathbb{F}_p$ and $N\otimes_\Z\mathbb{F}_p$ have the same (finite) dimension over $\mathbb{F}_p$, then $M=N$.

  Consider the multiplication 
  \[
  x_{[i]}\colon\CH^*\bigl(\Fl(N)\bigr)\to \CH^*\bigl(\Fl(N)\bigr).
  \]
  From Proposition~\ref{prop:ranks}, we get that the image has rank $(N-i)\cdot (N-1)!$, so the kernel $\Ann(x_{[i]})$ has rank $i\cdot (N-1)!$. Proposition~\ref{prop:ranks} also implies that the ideal $(x_{[N]\su [i]})$ is contained in $\Ann(x_{[i]})$ and has rank $i\cdot (N-1)!$. This means, in particular, that the claimed equality is true in $\CH^*\bigl(\Fl(N)\bigr)/p$, for every prime $p$. Consequently, the inclusion $(x_{[N]\su [i]})\subseteq\Ann(x_{[i]})$ is an equality integrally, as claimed.
\end{proof}

\begin{corollary}
  \label{cor:AnnFln}
  Let $g\circ f\colon\si_i\inj \si_j\inj \Fl(n)$ be the inclusion of the special Schubert varieties for some $i>j$. Then 
  \[\ker g^*\se (x_{[i]\su [j]}).\]
\end{corollary}

\begin{proof}
  We note that all morphisms of varieties are proper morphisms of smooth varieties. First, $g_*(1)=x_{[j]}$, see e.g.\ Fulton \cite{Fulton92}, and $g_*$ is injective. Second, 
  \[\ker g^*=\Ann\bigl(g_*(1)\bigr)=\Ann(x_{[j]})=(x_{[N]\su [j]}),\] 
  where we use projection formula $g_*\bigl(g^*(y)\bigr)=g_*(1)\cdot y$ and Proposition~\ref{prop:FlnAnn}. Finally, $(x_{[N]\su [j]})\se(x_{[i]\su [j]})$ is obvious, since $x_{[i]\su[j]}$ divides $x_{[N]\su[j]}$.
\end{proof}

\subsection{Freeness of \texorpdfstring{$W_\D$}{WD}}
\label{sec:freeness}

From this proof we also obtain that $W_\D$ is free for arbitrary coefficient rings $R$. We begin with the special case of integral coefficients. 

\begin{proposition}
  \label{prop:WDfree}
  For the coefficient ring $R=\mathbb{Z}$ and any $\D=(d_1\dots,d_m)$, the ring $W_\D$ defined in Definition~\ref{def:WD} is a free $\Z$-module.
\end{proposition}

\begin{proof}
  Since $W_\D=\Si_\D\otimes_\Z\La$, and $\La$ is free, cf.\ the discussion in Appendix A2.3 of \cite{eisenbud:commalg}, it is enough to prove that $\Si_\D$ is free for any $\D=(d_1,\dots,d_m)$.
	
  We first consider the case where all $d_i\in \D$ are odd. By Remark~\ref{rmk:Sigmaodd}, $\Si_\D\cong\CH^*(\Fl_{\D'})$ for $\D'=(\lfloor d_1/2\rfloor\stb \lfloor d_m/2\rfloor)$, and by Remark~\ref{rmk:CHfree}, $\CH^*(\Fl_{\D'})$ is free, which settles this case. Now we consider the case where there is at least one odd $d_i$. Every such $\D$ is obtained by adjoining square roots for some of the top Pontryagin classes. Therefore, we can use the all-odd case above, combined with a repeated application of Lemma~\ref{lem:adjoiningeuler} (as described in the proof of Corollary~\ref{cor:oddcase}). This shows that $\Si_\D$ is free whenever there is at least one odd $d_i$.
	
	For the case where all $d_i$ are even, using earlier notation, let $A$ be the Chow ring $\CH^*\bigl(\Fl(\D')\bigr)$ for $D'=(d_1/2\stb d_m/2)$, let $A'$ be the ring with all the Euler classes adjoined, and $A''=A'/(\e_1\ldots \e_m)$. Then, as an $A$-module, by Corollary~\ref{cor:quotientdecomposition}, 

\[
A''=\bigoplus_{I\se [m]}\bigl(A/(\p_{I^{\rm c}})\bigr)\bra \e_I\ket,
\]
where $I^{\rm c}$ denotes the complement of $I\se [m]$. Since $\CH^*(\Fl_\D)/(\op{c}_{I^{\rm c}})$ is torsion-free for all $I\se [m]$ by Corollary~\ref{cor:quotientfree}, we get that $A''$ is also a free $\Z$-module.
\end{proof}

With the freeness established in the case $R=\mathbb{Z}$, we can now generalize the freeness statement to arbitrary coefficient rings. 

\begin{theorem}
  \label{thm:WDfree}
  For any commutative coefficient ring $R$ and any $\D=(d_1,\dots,d_m)$, the $R$-algebra $W_\D$ defined in Definition~\ref{def:WD} is free as an $R$-module. The statement of Theorem~\ref{thm:annihilator} remains true for arbitrary coefficient rings $R$. 
\end{theorem}

\begin{proof}
By definition $W_{D,\mathbb{Z}}:=\Si_{\D,\mathbb{Z}}\otimes_{\mathbb{Z}}\Lambda_{\mathbb{Z}}$ and $W_{D,R}:=\Si_{\D,R}\otimes_{R}\Lambda_{R}\cong W_{D,\mathbb{Z}}\otimes_{\mathbb{Z}}R$. By Proposition~\ref{prop:WDfree}, $W_{D,\mathbb{Z}}$ is a free $\mathbb{Z}$-module, therefore $W_{D,R}$ is a free $R$-module. 

It remains to describe the annihilators of Euler classes. In the case $R=\mathbb{Z}$, these are described by Theorem~\ref{thm:annihilator}. For the computation in the general case, we make some observations concerning the kernel-cokernel sequence.

Let $H$ be a commutative $\mathbb{Z}$-algebra whose underlying abelian group is free and finitely generated, and let $\e\in H$ be an element. Consider the associated kernel-cokernel exact sequence
\[
0\to K\hookrightarrow  H\xrightarrow{\e} H\twoheadrightarrow C\to 0.
\]
If we assume that $C$ is free as abelian group, then so are the image of $\e$, and $K$. As a consequence, this exact sequence is preserved by arbitrary tensor products, so that for every $\mathbb{Z}$-algebra $R$, we get a similar exact sequence
\[
0\to K\otimes_{\mathbb{Z}}R\hookrightarrow  H\otimes_{\mathbb{Z}}R\xrightarrow{\e\otimes 1} H\otimes_{\mathbb{Z}}R\twoheadrightarrow C\otimes_{\mathbb{Z}}R\to 0.
\]

We now want to apply this argument in the case where $H=W_\D$ and $\e={\rm e}_i(\DD_i)$ is some Euler class for a subquotient bundle of even rank $d_i$. By Proposition~\ref{prop:WDfree}, $H$ is a free $\mathbb{Z}$-module. We still need to show that the module $C=H/(\e)$ is a free $\mathbb{Z}$-module.
\begin{description}
  \item[odd]In the case where some $d_j$ is odd, we write $H=A[\e]/(\e^2-\p)\cong A\oplus A\langle \e\rangle$ for ${\rm p}={\rm p}_{d_i}(\DD_i)$, and see that the quotient $H/(\e)$ is $A/(\p)$. But $A$ is an algebra over its Pontryagin subalgebra, which by Remark~\ref{rmk:Sigmaodd} is isomorphic to $\CH^*\bigl(\Fl(\D')\bigr)$ for $\D'=(\lfloor d_1/2\rfloor,\dots,\lfloor d_m/2\rfloor)$. The latter is free, and so is its quotient mod ${\rm c}$ (where ${\rm c}={\rm c}_{\lfloor d_i/2\rfloor }(\DD_i)$) by  Proposition~\ref{prop:kernel-top-chern}, hence
\[
A/(\p)\cong\bigl(\CH^*\bigl(\Fl(\D')\bigr)/({\rm c})\bigr)\otimes_{\CH^*(\Fl(\D'))}A
\]
is a free $\mathbb{Z}$-module. 
\item[even]In the case where all $d_i$ are even, we use the decomposition
\[
A''=\bigoplus_{I\subseteq[m]}\bigl(A/(\p_{I^{\rm c}})\bigr)\langle\e_I\rangle
\]
as a module over $A=\CH^*\bigl(\Fl(\D')\bigr)$. As above, $A''/(\p)\cong A/(\p)\otimes_AA''$, and we again get that $A''/(\p)$ is a free $\mathbb{Z}$-module.
\end{description}
Consequently, in both cases, all terms in the kernel-cokernel exact sequence are free abelian groups. Therefore, the sequence remains exact after tensoring with $R$, and we get
\[
\op{Ann}_{H\otimes_{\mathbb{Z}}R} (\e\otimes 1)=\op{Ann}_H(\e)\otimes_{\mathbb{Z}}R.
\]
This proves the second claim about the annihilator computation. 
\end{proof}

\section{Witt-sheaf cohomology of Grassmannian bundles and maximal rank flag varieties}
\label{sec:maximal}

In this section we prove  the main theorem in the special situation of maximal rank flag varieties. The proof uses the description of flag varieties as iterated Grassmannian bundles recalled in Section~\ref{sec:flags} in combination with the Leray--Hirsch theorem, cf.\ Proposition~\ref{prop:leray-hirsch}.

\subsection{Witt-sheaf cohomology for Grassmannian bundles}

We first formulate explicitly how the Leray--Hirsch theorem, cf.\ Proposition~\ref{prop:leray-hirsch}, provides formulas for the $\mathbf{W}$-cohomology of Grassmannian bundles. To fix notation, let $X$ be a scheme, $\EE\to X$ be a rank $n$ vector bundle, and denote by $\mathscr{G}(k,\EE)$ the Grassmannian bundle of rank $k$ subbundles of $\EE$. As for the projective spaces discussed in Example~\ref{ex:leray-hirsch-pn}, there are two cases to consider, depending on the parity of the dimension $k(n-k)$ of the fiber ${\rm Gr}(k,n)$ and the resulting difference in cohomology ring structure of the fiber. Recall from ~\cite{realgrassmannian} that for $k(n-k)$ even, the total $\mathbf{W}$-cohomology ring of the Grassmannian ${\rm Gr}(k,n)$ is generated by Pontryagin classes and Euler class of the tautological sub- and quotient bundles (if applicable). In particular, for a rank $n$ vector bundle $\EE$ on $X$, the characteristic classes of the tautological sub- and quotient bundles on the associated Grassmannian bundle $\mathscr{G}(k,\EE)$ restrict to generators of the $\mathbf{W}$-cohomology in each fiber. In these cases, the Leray--Hirsch theorem is applicable and completely describes the $\mathbf{W}$-cohomology of the Grassmannian bundle. This is the case we will discuss in the results below, cf.\ Proposition~\ref{prop:gr-bundle}. On the other hand, if $k(n-k)$ is odd, then the $\mathbf{W}$-cohomology of ${\rm Gr}(k,n)$ has a class ${\rm R}$ in degree $n-1$ which fails to be a characteristic class. In this case, the Leray--Hirsch theorem is not applicable, and likely fails as in the case of odd-dimensional projective bundles discussed in Example~\ref{ex:leray-hirsch-pn}. 

The results below, Propositions~\ref{prop:pic-gr-bundle} and \ref{prop:gr-bundle}, describe the Picard group and $\mathbf{W}$-cohomology for Grassmannian bundles $\mathscr{G}(2k,\EE)$ in the case where $\EE$ has even rank. Similar results are true by similar arguments (and used in the proof of Theorem~\ref{thm:maxrank}) for Grassmannian bundles $\mathscr{G}(2k,\EE)$ and $\mathscr{G}(2k+1,\EE)$ when $\EE$ has odd rank. 

\begin{proposition}
  \label{prop:pic-gr-bundle}
  Let $X$ be a smooth scheme and let $\EE$ be a rank $2n$ vector bundle on $X$. Denote by $\mathscr{G}(2k,\EE)$ the Grassmannian bundle of rank $2k$ subbundles of $\EE$ over $X$, and denote by $\ell_S={\rm c}_1(\SS_{\EE})$ and $\ell_Q={\rm c}_1(\QQ_{\EE})$ the first Chern classes of the tautological sub- and quotient bundles on $\mathscr{G}(2k,\EE)$, respectively. Then the Picard group of the Grassmannian bundle is described as follows:
  $$
  \op{Pic}\bigl(\mathscr{G}({2k},\EE)\bigr)\cong \Bigl(p^*\bigl(\op{Pic}(X)\bigr)\oplus \Z\bra \ell_S,\ell_Q\ket\Bigr)\Big/ \Bigl(\ell_S+\ell_Q=p^*\bigl({\rm c}_1(\EE)\bigr)\Bigr).
  $$
\end{proposition}

\begin{proof}
  While the claim can certainly be proved more directly, we deduce it from the Leray--Hirsch theorem for Chow groups, cf.\ \cite[Theorem 6.3]{krishna}, as preparation for our arguments in Witt-sheaf cohomology. The Grassmannian bundle is Zariski-locally trivial with smooth fiber. The well-known computation of ${\rm CH}^\ast\bigl({\rm Gr}(2k,2n)\bigr)$, recalled in Section~\ref{sec:flags}, implies that the Chow ring of the fiber is generated by characteristic classes of tautological sub- and quotient bundles on ${\rm Gr}(2k,2n)$. Since the restrictions of $\SS_\EE$ and $\QQ_\EE$ to any Grassmannian fiber of $\mathscr{G}(2k,\EE)$ are the appropriate tautological sub- and quotient bundles of the fiber, respectively, we find that the characteristic classes of $\SS_{\EE}$ and $\QQ_{\EE}$ restrict to a $\mathbb{Z}$-basis of the Chow ring ${\rm CH}^*\bigl({\rm Gr}(2k,\EE_x)\bigr)$ for every point $x\in X$. In degree 1, the isomorphism
  \[
    {\rm CH}^*\bigl({\rm Gr}(2k,2n)\bigr)\otimes_{\mathbb{Z}}{\rm CH}^*(X)\xrightarrow{\cong}{\rm CH}^*\bigl(\mathscr{G}(2k,\EE)\bigr)
  \]
  of \cite[Theorem 6.3]{krishna} induces an isomorphism $\op{Pic}\bigl({\rm Gr}(2k,2n)\bigr)\oplus\op{Pic}(X)\xrightarrow{\cong}\op{Pic}\bigl(\mathscr{G}(2k,\EE)\bigr)$.

  The claim follows from this, using that the Chern classes of tautological bundles on the fiber and on the Grassmannian bundle correspond under the above isomorphism, and the exact sequence $0\to \SS_\EE\to p^\ast(\EE)\to \QQ_\EE\to 0$ of vector bundles on $\mathscr{G}(2k,\EE)$ implies $\ell_S+\ell_Q=p^\ast{\rm c}_1(\EE)$.
\end{proof}

As discussed in Remark~\ref{rem:leray-hirsch-pic}, every line bundle on $\mathscr{G}(2k,\EE)$ is then a tensor product of a line bundle pulled back from the base $X$ and a power of a determinant bundle $\det \SS_{\EE}$ (or $\det \QQ_{\EE}$). This explains the possible twists available on $\mathscr{G}(2k,\EE)$ and implies that the Leray--Hirsch theorem of Proposition~\ref{prop:leray-hirsch} can be used to compute the total $\mathbf{W}$-cohomology ring for Grassmannian bundles. 

\begin{proposition}
  \label{prop:gr-bundle}
  Let $X$ be a smooth scheme and let $\EE\to X$ be a vector bundle of rank $2n$. Let $\pi\colon\mathscr{G}(2k,\EE)\to X$ denote the Grassmannian bundle associated to $\EE$ with its short exact sequence  $\SS_{\EE}\to \pi^*\EE\to \QQ_{\EE}$ of tautological bundles over $\mathscr{G}(2k,\EE)$. Then the total $\mathbf{W}$-cohomology ring $\bigoplus_\L \op{H}^*\bigl(\mathscr{G}(2k,\EE),\mathbf{W}(\L)\bigr)$ of the Grassmannian bundle is isomorphic to
  \begin{equation}
    \label{eq:Grassmannbundle}
    \bigoplus_\L \op{H}^*\bigl(X,\mathbf{W}(\L)\bigr)\Bigl[{\rm p}_{2i}(\SS_{\EE}),{\rm p}_{2j}(\QQ_{\EE}), {\rm e}(\SS_{\EE}),{\rm e}(\QQ_{\EE})\Bigr]\Big/\sim,
  \end{equation}
where the relations $\sim$ are the following:
\begin{equation}\label{eq:rel1}
	{\rm p}(\SS_{\EE}){\rm p}(\QQ_{\EE})=\pi^*{\rm p}(\EE),\qquad {\rm e}(\SS_{\EE}){\rm e}(\QQ_{\EE})=\pi^*{\rm e}(\EE),
\end{equation}
\begin{equation}\label{eq:rel2}
	{\rm e}(\SS_{\EE})^2={\rm p}_{2k}(\SS_{\EE}),  \qquad {\rm e}(\QQ_{\EE})^2={\rm p}_{2(n-k)}(\QQ_{\EE}).
\end{equation}
  Here ${\rm p}_{2i}(\SS_{\EE})$ with $i=1\stb k$ and ${\rm p}_{2j}(\QQ_{\EE})$ with $j=1\stb n-k$ are the (even) Pontryagin classes of the tautological sub- and quotient bundle, respectively. They are elements of the untwisted cohomology of the appropriate degree, i.e.,   $p_{2i}(\SS_{\EE})\in  \op{H}^{4i}\bigl(\mathscr{G}(2k,\EE),\mathbf{W}\bigr)$ and $p_{2j}(\QQ_{\EE})\in\op{H}^{4j}\bigl(\mathscr{G}(2k,\EE),\mathbf{W}\bigr)$. The classes ${\rm e}(\SS_{\EE})$ and ${\rm e}(\QQ_{\EE})$ are Euler classes of sub- and quotient bundle, and they are elements of the twisted cohomology groups, i.e., ${\rm e}(\SS_{\EE})\in \op{H}^{2k}\bigl(\mathscr{G}(2k,\EE),\mathbf{W}\bigl(\det(\SS_{\EE})^\vee\bigr)\bigr)$ and ${\rm e}(\QQ_\EE)\in \op{H}^{2(n-k)}\bigl(\mathscr{G}(2k,\EE),\mathbf{W}\bigl(\det(\QQ_{\EE})^\vee\bigr)\bigr)$. 
\end{proposition}

\begin{proof}
  The argument is a standard one, see e.g.\ \cite[Prop. 23.2]{BottTu}, we adapt it to the setting of $\mathbf{W}$-cohomology with twists. For typographical reasons, we introduce the shorthand notation
  $$\op{H}^*_\oplus(Y):=\bigoplus_\L \op{H}^*\bigl(Y,\mathbf{W}(\L)\bigr).$$
  We first introduce notation to write down a ring homomorphism. We denote by
  \[
  \pi^*\colon {\rm H}^*_\oplus(X)\to {\rm H}^*_\oplus\bigl(\mathscr{G}(2k,\EE)\bigr)
  \]
  the pullback along the projection $\pi$, and we denote by
  \[
  \kappa\colon {\rm H}^*_\oplus(\op{B}{\rm GL}_{2k})\otimes_{{\rm W}(F)}{\rm H}^*_\oplus(\op{B}{\rm GL}_{2(n-k)})\to {\rm H}^*_\oplus\bigl(\mathscr{G}(2k,\EE)\bigr)
  \]
  the pullback along the map classifying the pair  $(\SS_{\EE}, \QQ_{\EE})$ of vector bundles. Recall from \cite[Proposition 4.5]{realgrassmannian} that the Witt-sheaf cohomology of $\op{B}{\rm GL}_m$ is a polynomial ${\rm W}(F)$-algebra generated by even Pontryagin classes, plus an Euler class in case $m$ is even. The map $\kappa$ maps the universal characteristic classes in the cohomology of $\op{B}{\rm GL}_{2k}$ and $\op{B}{\rm GL}_{2(n-k)}$ to the characteristic classes of the bundles $\SS_{\EE}$ and $\QQ_{\EE}$, respectively. Now we consider the homomorphism of ${\rm W}(F)$-algebras
  \[
  f=\pi^*\otimes \ka\colon {\rm H}^*_\oplus(X)\left[{\rm p}_2^{(s)},\dots,{\rm p}_{2(k-2)}^{(s)},{\rm e}_{2k}^{(s)},{\rm p}_2^{(q)},\dots,{\rm p}_{2(n-k)-2}^{(q)},{\rm e}_{2(n-k)}^{(q)}\right]\to {\rm H}^*_\oplus\bigl(\mathscr{G}(2k,\EE)\bigr),
  \]
  where the superscripts denote the characteristic classes the corresponding elements will map to.

  We claim that this homomorphism is surjective, mapping the characteristic classes ${\rm p}_{2i}^{(s)}$ to ${\rm p}_{2i}(\SS_{\EE})$, ${\rm p}_{2j}^{(q)}$ to ${\rm p}_{2j}(\QQ_{\EE})$ (and similarly for the Euler classes), and has kernel exactly described by the equations \eqref{eq:rel1} and \eqref{eq:rel2}.
  
  We first note that these relations actually hold in ${\rm H}^*_\oplus\bigl(\mathscr{G}(2k,\EE)\bigr)$. The first set of relations \eqref{eq:rel1} follows from the Whitney sum formula, and \eqref{eq:rel2} already holds in the cohomology rings of the classifying spaces $\op{B}{\rm GL}_{2k}$ resp. $\op{B}{\rm GL}_{2(n-k)}$. Therefore, $f$ factors through the quotient by these relations, which is the ring described in \eqref{eq:Grassmannbundle}.

  It remains to show that $f$ induces an isomorphism of $\op{H}^*_\oplus(X)$-algebras from the algebra described in \eqref{eq:Grassmannbundle} to ${\rm H}^*_\oplus\bigl(\mathscr{G}(2k,\EE)\bigr)$. Since both the relevant ranks $2k$ and $2n$ are even, the Pontryagin and Euler classes of the tautological bundles $\SS_{\EE}$ and $\QQ_{\EE}$ multiplicatively generate $\op{H}^*_\oplus\bigl(\Gr(2k,2n)\bigr)$ by \cite[Theorem 6.4]{realgrassmannian}. Then by taking suitable monomials of characteristic classes of $\SS_{\EE}$ and $\QQ_{\EE}$, we obtain a set of elements in $\op{H}_\oplus^*\bigl(\mathscr{G}(2k,\EE)\bigr)$ which restricts to a basis in the ${\bf W}$-cohomology of the fibers of $\pi$. By the K\"unneth formula, cf.~Proposition~\ref{prop:kuenneth}, for an open subset $U\subseteq X$ over which $\EE$ trivializes, the characteristic-class monomials in $\op{H}_\oplus^*\bigl(\mathscr{G}(2k,\EE)\bigr)$ induce a ${\rm H}^*_\oplus(U,{\bf W})$-basis of ${\rm H}^*_\oplus(\pi^{-1}(U),{\bf W})$. Thus, we can apply the Leray--Hirsch theorem, cf.\ Proposition~\ref{prop:leray-hirsch}, to see that ${\rm H}^*_\oplus\bigl(\mathscr{G}(2k,E)\bigr)$ is a free $ \op{H}_\oplus^*(X)$-module on the same generators as for $\op{H}^*_\oplus\bigl(\Gr(2k,2n)\bigr)$. Since the domain of $f$ is also a free $\op{H}^*_\oplus(X)$-module on these elements, $f$ is an isomorphism. 
\end{proof}

\subsection{Maximal rank flag varieties}

We are now ready to establish the formula for $\mathbf{W}$-cohomology of maximal rank flag varieties. 

\begin{theorem}
  \label{thm:maxrank}
  For $\D=(2d_1\stb 2d_m)$ and $\D=(2d_1\stb 2d_{m-1},2d_m+1)$, the total $\mathbf{W}$-cohomology ring $\bigoplus_\L \op{H}^*\bigl(\op{Fl}(\D),\mathbf{W}(\L)\bigr)$ is isomorphic to the following commutative $\Z\oplus \op{Pic}\bigl(\Fl(\D)\bigr)/2$-graded ${\rm W}(F)$-algebra:
  $$\frac{\op{W}(F)\big[\op{p}_{2j}(\DD_i)\big]}{\prod_{i=1}^m \op{p}(\DD_i)=1}\bigotimes \op{W}(F)\big[\op{e}(\DD_i)\big]\bigg/\sim
  $$
  The generators are (even) Pontryagin classes ${\rm p}_{2j}(\DD_i)\in {\rm H}^{4j}\bigl(\Fl(\D),\mathbf{W}\bigr)$ (for $j=1\stb d_i$) of the rank $2d_i$ subquotient bundles $\DD_i$ (for $i=1\stb m$), as well as their Euler classes ${\rm e}(\DD_i)\in {\rm H}^{2d_i}\bigl(\Fl(\D),\mathbf{W}(\det \DD_i^\vee)\bigr)$. These classes satisfy the following relations $\sim$: for the odd-rank exception case, we have ${\rm e}(\DD_{2d_m+1})=0$, otherwise we have ${\rm e}(\DD_i)^2=\op{p}_{2d_i}(\DD_i)$, and there is an additional relation in the case when all $d_i$ are even:
  $$ \prod_{i=1}^m\op{e}(\DD_i)=0.$$
\end{theorem}

\begin{proof}
  We use the description of flag varieties as iterated Grassmannian bundles, together with an induction on $m$. Starting point is the case $m=2$ of Grassmannians in which the claim follows from \cite[Theorem 6.4]{realgrassmannian}.

  For the inductive step, we consider the case where all entries of $\D$ are even. The case where the last entry is odd is similar. Now it suffices to show the claim of the theorem for the case $\D=(2d_1,\dots,2d_m)$ under the assumption that the theorem is known for the case $\D'=\bigl(2d_1,\dots,2d_{m-2},2(d_{m-1}+d_m)\bigr)$. In this case, $\Fl(\D)$ is isomorphic to the Grassmannian bundle $\mathscr{G}(2d_{m-1},\EE)$, where $\EE$ is the quotient bundle of rank $2(d_{m-1}+d_m)$  over $\Fl(\D')$. We can apply Proposition~\ref{prop:gr-bundle} in this situation to get a description of ${\rm H}^*_\oplus\bigl(\Fl(\D)\bigr)$ as algebra over ${\rm H}^*_\oplus\bigl(\Fl(\D')\bigr)$ generated by Pontryagin and Euler classes of the subquotient bundles $\DD_{m-1}$ and $\DD_m$. Using the inductive assumption, we find that the cohomology ring is indeed generated by Pontryagin and Euler classes of all the subquotient bundles. It remains to rewrite the relations. The relations ${\rm e}(\DD_i)^2={\rm p}_{2d_i}(\DD_i)$ follow from relations~\eqref{eq:rel2} of Proposition~\ref{prop:gr-bundle}. Moreover, we have
  \[
  \prod_{i=1}^m{\rm e}(\DD_i)=\pi^\ast\bigl({\rm e}(\EE)\bigr)\prod_{i=1}^{m-2}{\rm e}(\DD_i)=0,
  \]
  where the first equality follows from relation~\eqref{eq:rel1} of Proposition~\ref{prop:gr-bundle}, and the second equality follows from the inductive assumption by pullback along $\pi\colon\Fl(\D)\to\Fl(\D')$. Similarly, the relation for total Pontryagin classes follows
  \[
  \prod_{i=1}^m{\rm p}(\DD_i)=\pi^\ast\bigl({\rm p}(\EE)\bigr)\prod_{i=1}^{m-2}{\rm p}(\DD_i)=1,
  \]
  using the Pontryagin class part of relation~\eqref{eq:rel1} of Proposition~\ref{prop:gr-bundle} and the inductive assumption. This concludes the proof.
\end{proof}

\begin{corollary}
  \label{cor:free-chow-w}
  For $\D=(2d_1\stb 2d_m)$ and  $\D=(2d_1\stb 2d_{m-1},2d_m+1)$, the total $\mathbf{W}$-cohomology ring ${\rm H}^*_\oplus\bigl(\Fl(\D)\bigr)$ is free as a ${\rm W}(F)$-module.

  For $\D=(d_1\stb d_m)$,  there is an isomorphism
  \[
  {\rm CH}^*\bigl(\Fl(\D)\bigr)\otimes_{\mathbb{Z}}{\rm W}(F)\cong {\rm H}^{4*}\bigl(\Fl(2\D),\mathbf{W}\bigr)
  \]
  sending the Chern classes ${\rm c}_i(\DD_j)$ of the subquotient bundle $\DD_j$  (of rank $d_j$) to the Pontryagin classes ${\rm p}_{2i}(\EE_j)$ of the subquotient bundle $\EE_j$ (of rank $2d_j$). 
\end{corollary}

\begin{proof}
  The second claim follows directly from the fact that the presentations for the Chow rings of the flag variety for $\D$ and the $\mathbf{W}$-cohomology rings of the flag variety for the doubled partition $2\D$ are the same (except for a degree shift and different coefficient rings). The bundle $\DD_j$ of rank $d_j$ has nontrivial Chern classes ${\rm c}_1,\dots,{\rm c}_{d_j}$ while the bundle $\EE_j$ of rank $2d_j$ has nontrivial Pontryagin classes ${\rm p}_2,{\rm p}_4,\dots,{\rm p}_{2d_j}$; and for both rings, the only nontrivial relations are the respective Whitney sum formulas, $\prod{\rm c}(\DD_j)=1$ on the Chow side and $\prod{\rm p}(\EE_j)=1$ on the $\mathbf{W}$-cohomology side.

    The freeness follows from the presentation, cf.~Theorem~\ref{thm:WDfree}, or the isomorphism in the second statement, but can alternatively also be seen inductively: in every application of the Leray--Hirsch theorem for Grassmannian bundles, the total $\mathbf{W}$-cohomology ring is a free module over the $\mathbf{W}$-cohomology of the base, and in the basic case of Grassmannians, the freeness follows from \cite[Theorem 6.4]{realgrassmannian}. 
\end{proof}

\begin{remark}
  It should be pointed out that this method only works in the maximal rank case, because the Leray--Hirsch theorem doesn't apply to arbitrary Grassmannian bundles. The conditions of Proposition~\ref{prop:leray-hirsch} are satisfied generally for Grassmannian bundles with fiber ${\rm Gr}(k,n)$ with $k(n-k)$ even. In the case where $k(n-k)$ is odd, the Leray--Hirsch theorem will only apply under very restrictive hypotheses which will generally not be satisfied. In the maximal rank case of flag varieties ${\rm Fl}(\D)$ with $\D=(d_1,\dots,d_m)$ such that at most one of the $d_i$ is odd, the Grassmannian bundles will have fibers ${\rm Gr}(k,n)$ with $k(n-k)$ even. As soon as at least two of the $d_i$ are odd, there will also be Grassmannian bundles with $k(n-k)$ odd and the Leray--Hirsch theorem will not be applicable.
\end{remark}

\begin{remark}
  The $\mathbf{W}$-cohomology of the oriented flag varieties $\op{SFl}(2,\dots,2)$ and $\op{SFl}(2,\dots,2,1)$ has been worked out by Ananyevskiy \cite{ananyevskiy}, for precise statements see the theorem statement in the introduction and Remark 12 of Section 10 of loc.~cit. These oriented flag varieties can be defined as quotients of $\op{SL}_{2n}$ resp. $\op{SL}_{2n+1}$ by a subgroup (almost a parabolic) containing $\op{SL}_2^{\times n}$. Alternatively, they can be understood as $\mathbb{A}^1$-universal coverings of the respective flag varieties $\op{Fl}(2,\dots,2)$ and $\op{Fl}(2,\dots,2,1)$ considered here. In particular, they have trivial Picard groups and thus there is only untwisted Witt-sheaf cohomology. In \cite{ananyevskiy}, the generators are the Euler classes of the rank 2 subquotient bundles, comparable to our Theorem~\ref{thm:maxrank}. Note, however, that the Euler classes for $\op{Fl}(\D)$ are elements in twisted Witt-sheaf cohomology groups, whereas there are no nontrivial twists in the case $\op{SFl}$. We refer to \cite{ananyevskiy} for a reinterpretation of the presentation of cohomology rings in terms of Weyl group coinvariants. 
\end{remark}

As a consequence of the maximal rank computation in Theorem~\ref{thm:maxrank}, we also obtain a general formula for Witt-sheaf cohomology of maximal rank flag bundles over smooth schemes. We formulate the result for the case where all entries in $\D$ are even, a similar statement holds for the other maximal rank case.

\begin{corollary}\label{cor:splitting}
  Fix $\D=(2d_1,\dots,2d_m)$. Let $X$ be a smooth scheme and let $\EE$ be a vector bundle of rank $2\sum d_i$ over $X$. Let $\pi\colon\mathscr{F}l(\D,\EE)\to X$ denote the $\D$-flag bundle associated to $\EE$. Then the total $\mathbf{W}$-cohomology of $\mathscr{F}l(\D,\EE)$ is the ${\rm H}^*_\oplus(X)$-algebra generated by the even Pontryagin and Euler classes of the tautological subquotient bundles $\DD_{i,\EE}$ subject to the following relations:
  \begin{eqnarray*}
    {\rm e}^2(\DD_{i,E})&=&{\rm p}_{2d_i}(\DD_{i,\EE}),\\
    \prod_{i=1}^m{\rm e}(\DD_{i,\EE})&=&\pi^*{\rm e}(\EE),\\
    \prod_{i=1}^m{\rm p}(\DD_{i,\EE})&=&\pi^*{\rm p}(\EE).
  \end{eqnarray*}
  The total $\mathbf{W}$-cohomology ring ${\rm H}^*_\oplus\bigl(\mathscr{F}l(\D,\EE)\bigr)$ is free as ${\rm H}^*_\oplus(X)$-module.
\end{corollary}

\begin{remark}
  \label{rem:splitting-principle}
  This result can be used for splitting-principle arguments. It allows to prove identities for cohomology classes in ${\rm H}^*\bigl(\Fl(\D),\mathbf{W}\bigr)$ for a partition $\D=(2d_1,\dots,2d_m)$ with $N=2\sum_{i=1}^m d_i$ by reduction to the flag variety ${\Fl}(\D_{\rm split})$ for
  \[
  \D_{\rm split}:=(\underbrace{2,\dots,2}_N)
  \]
  Analogous to the description of flag varieties as iterated Grassmannian bundles, the projection $\Fl(\D_{\rm split})\to \Fl(\D)$ is an iterated flag variety bundle with fibers flag varieties for partitions $(2,\dots,2)$ of length $d_1,\dots,d_m$. By Corollary~\ref{cor:splitting}, these flag variety bundles induce injective pullback morphisms in $\mathbf{W}$-cohomology. To prove an identity for cohomology classes, it thus suffices to prove that identity after pullback to $\Fl(\D_{\rm split})$. A similar splitting principle for $\eta$-inverted theories and the oriented flag varieties is discussed in \cite[Section 9]{ananyevskiy}.
\end{remark}

\section{An elementary computation of the \texorpdfstring{$\mathbf{W}$}{W}-cohomology rings}
\label{sec:sadykov}

\newcommand{\DFl}{\widetilde{\Fl}}

In this section, we will now go beyond the maximal rank case and establish formulas for the $\mathbf{W}$-cohomology of general type A flag varieties similar to the Borel--Cartan formulas in topology, cf.~\cite{he}, \cite{matszangosz}. The argument is an elementary one along the lines of \cite{sadykov} for Grassmannians, using spherical bundles associated to subquotients of the tautological filtration which relate different flag varieties via localization sequences. 

The proof proceeds by transforming a maximal rank sequence $\D_0$ to another sequence $\D_1=(d_1\stb d_m)$, by repeatedly applying the following two operations:
\begin{itemize}
	\item take an arbitrary permutation $\D':=\pi.\D$ of $\D=(d_1\stb d_m)$, $\pi\in S_m$,
	\item for $d_1, d_2$ even, replace $\D=(d_1\stb d_m)$ by $\D'=(d_1-1, d_2+1,d_3\stb d_m)$.
\end{itemize}
It is not hard to see that any sequence $\D_1$ can be obtained from a maximal rank sequence $\D_0$ using these two operations. We will describe how the $\mathbf{W}$-cohomology changes under these two operations. For the proof we will consider \emph{decomposition varieties} $\DFl(\D)$, which are $\mathbb{A}^1$-equivalent replacements of the flag varieties $\Fl(\D)$. The upside is that the tautological exact sequences split over $\DFl(\D)$, contrary to the case of flag varieties.
To simplify notation, in the rest of this section, we denote 
\begin{equation}\label{eq:notation}
	\op{H}_{\EE}^*(X):=\op{H}^*\bigl(X;\mathbf{W}(\det \EE)\bigr),\quad
	\op{H}_\oplus^*(X):= \bigoplus_{\L}\op{H}_{\L}^*(X),\quad
	\op{H}^*(X):=\op{H}_{\mathscr{O}}^*(X).
\end{equation}

\subsection{The decomposition variety}
Fix $\D=(d_1,\dots,d_m)$. Let $\widetilde{\Fl}(\D)$ denote the \emph{decomposition variety of $V$} defined as open subscheme  of a product $\prod_{i=1}^m \Gr_{d_i}(V)$ of Grassmannians containing the points
$$ V_\bullet=(V_1,V_2\stb V_m)\in \prod_{i=1}^m \Gr_{d_i}(V),$$
given by $d_i$-dimensional subspaces $V_i\subseteq V$ (for $i=1,\dots,m$) such that $V=\bigoplus V_i$. The decomposition variety is a homogeneous space under the obvious action of ${\rm GL}(V)$, and can be identified as
\[
\widetilde{\Fl}(\D)\cong{\rm GL}_n/({\rm GL}_{d_1}\times\cdots\times{\rm GL}_{d_m}),
\]
since the stabilizer of a decomposition $V=V_1\oplus\cdots\oplus V_m$ is exactly the block-diagonal subgroup ${\rm GL}_{d_1}\times\cdots\times{\rm GL}_{d_m}$ in ${\rm GL}_N$.

This space comes with projections $\pi_i(V_\bullet):=V_i$ to the $i$-th factor ${\rm Gr}_{d_i}(V)$. These projections induce tautological bundles $\DD_i:=\pi_i^*(\SS_i)$ over $\DFl(\D)$ via pulling back the $i$-th tautological subbundle $\SS_i$ along $\pi_i$. There is a map
$$ \pi\colon\widetilde{\Fl}(\D)\to \Fl(\D)\colon (V_1\stb V_m)\mapsto (V^1\stb V^m),$$
where $V^i=\oplus_{j=1}^i V_j$. Note that $\pi^* \DD_i\cong \DD_i$ and $\pi^*\SS_i\cong\oplus_{j=1}^i \DD_j$.

\begin{proposition}
  The projection map $\pi\colon\DFl(\D)\to \Fl(\D)$ is an $\mathbb{A}^1$-homotopy equivalence. In particular, it induces an isomorphism in $\mathbf{W}$-cohomology. 
\end{proposition}

\begin{proof}
  As mentioned above, $\GL(V)$ acts on both spaces, in a manner compatible with $\pi$. Actually, $\pi$ can be explicitly identified as the quotient map
  $$ \pi\colon\GL(V)\Big/\prod_{i=1}^m\GL(d_i)\to \GL(V)\Big/\GL(d_1\stb d_m),$$
  where $\GL(d_1\stb d_m)$ denotes the parabolic subgroup of block upper triangular matrices of the specified block sizes. The fiber of this map is $\GL(d_1\stb d_m)/\prod_{i=1}^m\GL(d_i)$. This is a unipotent group of strictly upper block triangular matrices whose underlying space is isomorphic to $\mathbb{A}^r$ for suitable $r$ (and therefore $\mathbb{A}^1$-contractible). Consequently, $\pi$ is a vector bundle torsor and therefore an $\mathbb{A}^1$-weak equivalence.

  The second claim follows from $\mathbb{A}^1$-invariance of $\mathbf{W}$-cohomology. 
\end{proof}

The advantage of working with $\DFl(\D)$ as opposed to $\Fl(\D)$ is that the bundles $\DD_i$ are subbundles of the trivial bundle $V$ as opposed to subquotients. This induces geometric isomorphisms as described below.

\begin{lemma}\label{lem:flexchange}
  For any permutation $\sigma\in \Sigma_m$ and $\D=(d_1\stb d_m)$, there are isomorphisms of decomposition varieties
  $$ \widetilde{\Fl}(\sigma\cdot\D)\cong \widetilde{\Fl}(\D).$$
\end{lemma}

\begin{proof}
  On points, i.e., decompositions $V\cong V_1\oplus\cdots\oplus V_m$, the map is given as follows:
  \[
  V_\bullet\mapsto \sigma\cdot V_\bullet:=\bigl(V_{\sigma(1)},V_{\sigma(2)}\stb V_{\sigma(m)}\bigr).
  \]
  To see that this is a morphism of varieties $\DFl(\D)\to \DFl(\sigma\cdot\D)$, one can use the description as homogeneous spaces. The morphism
  \[
  \widetilde{\Fl}(\D)\cong{\rm GL}_n/\bigl({\rm GL}_{d_1}\times\cdots\times{\rm GL}_{d_m}\bigr)\to{\rm GL}_n/\bigl({\rm GL}_{d_{\sigma(1)}}\times\cdots\times{\rm GL}_{d_{\sigma(m)}}\bigr)\cong \widetilde{\Fl}(\sigma\cdot\D)
  \]
  is then given by sending the standard basis vectors with indices $1+\sum_{i=1}^{j-1} d_i,\dots,\sum_{i=1}^jd_i$ to the standard basis vectors with indices $1+\sum_{i=1}^{\sigma(j-1)} d_{\sigma(i)},\dots,\sum_{i=1}^{\sigma(j)}d_{\sigma(i)}$. This will exactly permute the order of blocks as required and therefore map the subgroup ${\rm GL}_{d_1}\times\cdots\times{\rm GL}_{d_m}$ to the subgroup ${\rm GL}_{d_{\sigma(1)}}\times\cdots\times{\rm GL}_{d_{\sigma(m)}}$. The required morphism is the one induced on the quotients. 

  Its inverse is similarly given by $V_\bullet'\mapsto \sigma^{-1}\cdot V_\bullet'$.
\end{proof}

\begin{proposition}
  \label{prop:move1}
  Fix $\D=(d_1,\dots,d_m)$ and $\D'=(d_1,\dots,d_i-1,d_{i+1}+1,\dots,d_m)$. Let $q\colon\mathscr S\to \DFl(\D)$ and $p\colon\mathscr S'\to \DFl(\D')$ denote the sphere bundles 	
  $$\mathscr{S}:=\DD_{i}^\vee\setminus z\bigl(\DFl(\D)\bigr),\qquad  \mathscr{S}':=\DD_{i+1}'\setminus z\bigl(\DFl(\D')\bigr).$$ 
  where $z$ denotes the zero section. Then there is an isomorphism of varieties:
  $$ f\colon\mathscr{S}\to \mathscr{S}',$$
  whose inverse we denote $g\colon\mathscr{S}'\to \mathscr{S}$. Moreover, under these isomorphisms, we have the following isomorphisms of pullbacks of tautological vector bundles 
  \begin{equation}
    \label{eq:pullback}
    g^\ast q^\ast \DD_{i}\cong p^\ast \DD'_{i}\oplus \mathscr{O},\qquad   
    f^\ast p^\ast \DD'_{i+1}\cong q^\ast \DD_{i+1}\oplus \mathscr{O}.
  \end{equation}
\end{proposition}

\begin{proof}
  Fix a nondegenerate symmetric bilinear form $B$ on $V$, or alternatively an isomorphism $V\cong V^\vee$. We will describe the morphisms $f$ and $g$ by what they do on points. Recall that points of $\DFl(\D)$ are decompositions $V_\bullet=(V_1\oplus\cdots\oplus V_m=V)$, and the bundle $\DD_i^\vee$ restricted to $V_\bullet$ is just the subspace $V_i^\vee$. The points of the sphere bundle $\mathscr{S}$ are then given by pairs of a decomposition $V=V_1\oplus\cdots\oplus V_m$ together with a nonzero vector in the subspace $V_i^\vee$, alternatively a nonzero linear form on $V_i$. 
  
  With this notation, we can now define 
  $$
  f\colon\mathscr S\to \mathscr S'\colon (V_\bullet,\varphi\in V_i^\vee)\mapsto (V_\bullet', w\in V_{i+1}'),
  $$
  where $V_\bullet'$ is the decomposition obtained from $V_\bullet$ by taking 
  $$ V_\bullet'=\Bigl(V_1\stb V_{i-1},\ker \varphi, L\oplus V_{i+1},V_{i+2}\stb V_m\Bigr),$$
  where $L=(\ker \varphi)^\perp \cap V_i$ and $w:=\varphi^{-1}(1)\cap L$.

  The inverse map $g$ is defined as 
  $$
  g\colon \mathscr{S}'\to\mathscr{S}\colon (V_\bullet', w\in V_{i+1}')\mapsto (V_\bullet,\varphi\in V_i^\vee),
  $$ 
  where $V_\bullet$ is the decomposition obtained from $V_\bullet'$ by taking
  $$ V_\bullet=\Bigl(V_1'\stb V_{i-1}', V_i'\oplus \bra w\ket, w^\perp \cap V_{i+1}',V'_{i+2}\stb V_m\Bigr)$$
  and define the form $\varphi\in V_i^\vee$ by the requirements $\varphi(w):=1$ and $\varphi(V_i')\equiv 0$.

  It is then clear from the definitions that these are inverses of each other.

The tautological bundles split as described in \eqref{eq:pullback}, which follows from the form of the maps $f$ and $g$: the bundles $L$ and $\bra w\ket$ are the trivial bundles $\mathcal{O}$ in the statement.
\end{proof}
\begin{remark}
	The previous proof can be translated to a homogeneous space formulation along the lines of \cite[Proposition 6.1]{realgrassmannian}. 
\end{remark}
\begin{remark}
  Over the flag varieties $\Fl(\D)$ and $\Fl(\D')$, isomorphisms $f\colon\PP \DD_{i}^\vee\to \PP \DD_{i+1}'$ exist between the projectivizations of the bundles. 
\end{remark}

\begin{example}
  We want to point out that this result specifically also holds in the case $d_i=1$. The simplest example is given by $\op{Fl}(1,1)\cong\mathbb{P}^1$. The complement of the zero section in the canonical line bundle is $\mathbb{A}^2\setminus\{0\}$, and this can also be interpreted as the complement of the zero section of the trivial rank 2 bundle on $\op{Fl}(0,2)=\op{pt}$.
\end{example}

The induction step for extending the computation to the non-maximal rank case proceeds by a diagram-chase on the Gysin sequences \eqref{eq:gysin} (on page~\pageref{eq:gysin}) of the sphere bundles $\mathscr{S}_i\iso \mathscr{S}'_{i+1}$ via this isomorphism, together with the description of the $\mathbf{W}$-cohomology of the sphere bundle that we gave in Proposition~\ref{prop:spherebundle}. For later, we spell out the form of the Gysin sequences in the following proposition:

\begin{proposition}
  \label{prop:locgrass}
  For every twist $\L$, there are two localization sequences:
  \[
  \xymatrix{
&    \op{H}^{*-1}(\mathscr{S}_{i}) \ar[r]^-{\partial}\ar[d]^{\iso} &
     \op{H}^{*-d_i}_{\DD_i}\bigl(\DFl(\D)\bigr) \ar[r]^-{\op{e}_i} &
     \op{H}^{*}\bigl(\DFl(\D)\bigr) \ar[r]^-{q_i^*} &
     \op{H}^*(\mathscr{S}_{i}) \ar[d]^{\iso}&\\
&   \op{H}^{*-1}(\mathscr{S}'_{i+1}) \ar[r]^-{\partial}&
     \op{H}^{*-d_{i+1}-1}_{\DD_{i+1}'}\bigl(\DFl(\D')\bigr) \ar[r]^-{\op{e}_{i+1}'} &
     \op{H}^*\bigl(\DFl(\D')\bigr) \ar[r]^-{p_{i+1}^*} &
     \op{H}^*(\mathscr{S}'_{i+1}) &\phantom{a} \hspace{-1.8 cm } ,
  }
  \]
  where $\op{H}^*(X)=\op{H}^*\bigl(X;\mathbf{W}(\L)\bigr)$, $\op{H}^*_{\EE}(X)=\op{H}^*\bigl(X;\mathbf{W}\bigl(\L\otimes \det(\EE)\bigr)\bigr)$,  and $\op{e}_i= \op{e}(\DD_i)$, $\op{e}_{i+1}'= \op{e}(\DD_{i+1}')$.
\end{proposition}

It should be noted that the boundary maps and multiplication with the Euler classes change the twists in the Witt-sheaf cohomology.

\subsection{Induction step: $\mathbf{W}$-cohomology of the sphere bundle}

In this section, we discuss the behaviour of $\mathbf{W}$-cohomology under the change from $\D=(d_1,\dots,d_m)$ to $\D'=(d_1-1,d_2+1,d_3,\dots,d_m)$. For the Grassmannian special case, cf.~\cite{sadykov} in a singular cohomology situation or \cite[Section 6.2]{realgrassmannian} in a Witt-sheaf cohomology situation. For readability,  we will use the simplified notation of \eqref{eq:notation}, cf.\ page~\pageref{eq:notation}.

The induction step consists of two parts. Assume that the cohomology of $\Fl(\D)$ is described by the presentation in Definition~\ref{def:WD}, i.e., we have an isomorphism ${\rm H}^*_\oplus\bigl(\Fl(\D)\bigr)\iso W_\D$. The following result then describes the $\mathbf{W}$-cohomology of the sphere bundle $\mathscr{S}_i$ associated to $\DD_i^\vee$ as $\op{coker}\rm{e}_i\otimes \La[R]$, combining Proposition~\ref{prop:spherebundle} and the annihilator computations of Theorem~\ref{thm:annihilator} in Section~\ref{sec:ann-euler}. In the next subsection, we will then deduce from this description of the sphere bundle cohomology that the cohomology $\Fl(\D')$ also satisfies the presentation from Definition~\ref{def:WD}. 

\begin{proposition}
  \label{prop:sphere}
  Fix $\D=(d_1,\dots,d_m)$ with $d_i$ even, and assume that ${\rm H}^*_\oplus\bigl(\Fl(\D)\bigr)\iso W_\D$, with $W_\D$ as in Definition~\ref{def:WD}. Denoting by ${\rm e}_i$ the endomorphism of ${\rm H}^*_\oplus\bigl(\Fl(\D)\bigr)$ given by multiplication with the Euler class ${\rm e}_i={\rm e}(\DD_i^\vee)$ of the bundle $\DD_i^\vee$, we have $\ker {\rm e}_i=(x)$ with $x\in {\rm H}^*_\oplus\bigl(\Fl(\D)\bigr)$ as in Theorem~\ref{thm:annihilator}. Then the $\mathbf{W}$-cohomology of the sphere bundle $\mathscr{S}_i$ associated to $\DD_i^\vee$ is
  $$
  \bigoplus_\L \op{H}^*\bigl(\mathscr{S}_i,\mathbf{W}(\L)\bigr)\cong \op{coker} {\rm e}_i \otimes \Lambda[R],
  $$
  where $R$ is an element satisfying $\partial R=x$. The element $R$ has degree $N-1$ if all $d_i$ are even and $4q-1$ otherwise, where $q=\sum_{i=1}^m \lfloor \frac{d_i}{2}\rfloor$ and $N=\sum_{i=1}^md_i$.
\end{proposition}

\begin{proof}
  Introduce the shorthand notation $C:=\coker \op{e}_i$ for the whole cokernel and $C^d:=\coker \op{e}_i\cap \op{H}^d$. Theorem~\ref{thm:annihilator} shows that $\ker\e_i$ is a principal ideal, and an argument as in the proof of Theorem~\ref{thm:WDfree} shows that $\ker\e_i$ is indeed a free $C$-module of rank 1.  Therefore, the conditions of Proposition~\ref{prop:spherebundle} are satisfied, so as a $C$-module for an element $R\in \op{H}^r$ in the trivial twist:
  \[\op{H}^*_{\oplus}(\mathscr{S}_i)=C\bra 1,R\ket,\qquad r:=\deg R=\deg x+d_i-1.\]
  It remains to verify $R^2=0$. We will prove this using the Gysin sequence \eqref{eq:Gysinkercoker} by showing that 
  \begin{equation}\label{eq:partialR2ker}
    \partial R^2\in \ker \op{e}_i\cap \op{H}^{2r-d_i+1}_{\DD_i}=(0)	
  \end{equation} 
  and that in the trivial twist $C^{2r}=(0)$. 
  To show \eqref{eq:partialR2ker}, since $\ker \op{e}_i=C\bra x\ket$ by Theorem~\ref{thm:annihilator} and $\deg x=r+1-d_i$, it is enough to show that in the trivial twist $C^r=(0)$.
	
  If all $d_i$ are even, then $\deg x=N-d_i$ and $r=N-1$. In the trivial twist $C$ is generated by Pontryagin classes of degrees divisible by 4, so since $r$ is odd, $C^r=(0)$ and $C^{2r}=(0)$.
  
  If not all $d_i$ are even, then $\deg x=4q-d_i$ and $r=4q-1$. The list of odd degree generators is $R_{l}$ of degrees $4l-1$, $l=q+1\stb n$. This already implies that $C^{4l-1}=(0)$ for $l\leq q$, so that $C^r=(0)$. To land in degree $2r=8q-2$, by looking at residues modulo 4, one must take the product of at least two $R_l$'s. However 
  $$\deg R_iR_j\geq (4q+4)+(4q+8)-2=8q+10.$$ 
  Therefore $C^{2r}=(0)$ as well.	
\end{proof}

\subsection{End of proof of Theorem~\ref{thm:main}}
We will use the notation
\[
{\rm p}_{\rm top}(\EE):={\rm p}_{2\lfloor \frac{\op{rk} \EE}{2}\rfloor}(\EE)=\left\{\begin{array}{ll}
	{\rm p}_{\op{rk} \EE}(\EE) & \op{rk} \EE \textrm{ even}\\
	{\rm p}_{\op{rk} \EE-1}(\EE) & \op{rk} \EE\textrm{ odd}.
\end{array}\right.
\]
to denote the top Pontryagin class of a vector bundle $\EE$. Note that in $\mathbf{W}$-cohomology the odd Pontryagin classes are zero, cf.~\cite[Proposition 4.5]{realgrassmannian}. 

\begin{proposition}
  \label{prop:proof-main-thm}
  Fix partitions
  \[
  \D=(d_1,\dots,d_m)\quad \textrm{ and } \quad \D'=(d_1,\dots,d_i-1,d_{i+1}+1,\dots,d_m)
  \]
  of $N=\sum d_i$, and assume $d_i,d_{i+1}$ are even. Assume that the total $\mathbf{W}$-cohomology ring for $\op{Fl}(\D)$ satisfies the description of Theorem~\ref{thm:main}. Then the total $\mathbf{W}$-cohomology ring $\bigoplus_{\L}{\rm H}^*\bigl(Fl(\D'),\mathbf{W}(\L)\bigr)$ is also described by the presentation in Definition~\ref{def:WD}, and Theorem~\ref{thm:main} holds for $\op{Fl}(D')$. 
\end{proposition}

\begin{proof}
  We use the simplified notation of \eqref{eq:notation}.
  Since $\op e_{i+1}'=0$, the lower exact sequence of Proposition~\ref{prop:locgrass} reduces to the short exact sequence
  \[\xymatrix{
    0 \ar[r]^-{}&
    \op{H}_{\L}^*\bigl(\op{Fl}(\D')\bigr) \ar[r]^-{p_{i+1}^*} &
    \op{H}_{\L}^*(\mathscr{S}'_{i+1}) \ar[r]^-{\partial} &
    \op{H}_{\L\otimes \DD_{i+1}'}^{*-d_{i+1}}\bigl(\op{Fl}(\D')\bigr) \ar[r]^-{} &
    0.
  }\]
  By Propositions~\ref{prop:move1} and \ref{prop:sphere},
  $$\op{H}_\oplus ^*(\mathscr{S}_{i+1}')\simeq\op{H}_\oplus^*(\mathscr{S}_i)\simeq \op{coker} {\rm e}_i \otimes \Lambda[R],$$
  with $R\in \op{H}^{4q-1}$ or $R\in \op{H}^{N-1}$, depending on the parity of the other $d_j$'s. There are three types of generators of $\op{coker} {\rm e}_i$:
  \begin{itemize}
  \item characteristic classes $\op{p}_j(\DD_i)$ and $\op{e}_l$ for $l\neq i, i+1$,
  \item all $R_j$'s,
  \item an additional characteristic class ${\rm e}_{i+1}$, which we treat separately.
  \end{itemize}
  We will show that the subalgebra generated by characteristic classes except $e_{i+1}$, $R$ and the $R_j$'s is the image of $p_{i+1}^*$. Indeed, by \eqref{eq:pullback}, all polynomials in characteristic classes of the bundles $\DD_l$ are in the image of $p_{i+1}^*$; for instance
  \[
p_{i+1}^*{\rm p}_j(\DD_i'\oplus \mathcal{O})=q_i^*{\rm p}_j(\DD_i)
  \] 
  and similarly, this follows for the rest of the Pontryagin classes by naturality of pullback. For Euler classes $p_{i+1}^*{\rm e}(\DD_l)=q_i^*{\rm e}(\DD_l)$ for $l\neq i,i+1$. 
  
  Since $R_j\in \op{H}^*(\mathscr{S}_{i+1}')$, the boundary $\partial R_j$ lives in $\op{H}_{\DD_{i+1}'}^{*}\bigl(\op{Fl}(\D')\bigr)$, which equals $(0)$, since $e_{i+1}'=0$ and $e_{i}'=0$ (both $d_i'$ and $d_{i+1}'$ are odd), similarly for the additional generator $R$. So all $R$'s are also in the image of $p_{i+1}^*$.
  
  By degree considerations, the exceptional characteristic class ${\rm e}_{i+1}$ necessarily maps to 1 via $\partial$. Finally, by the derivation property, ${\rm e}_{i+1}\cdot \la$ maps to $\la$ for any $\la\in \op{Im}p_{i+1}^*$. This allows us to conclude the proof because $\op{Im} p_{i+1}^*$ coincides with the description of $\op{H}_{\oplus}^*\bigl(\op{Fl}(\D')\bigr)$ in Theorem~\ref{thm:main}, and $p_{i+1}^*$ is injective.
\end{proof}

\section{Chow--Witt rings and singular cohomology}
\label{sec:chowwittsingular}

In this section, we draw some consequences from Theorem~\ref{thm:main}. On the one hand, we can give a description of the $\mathbf{I}$-cohomology and Chow--Witt ring of partial flag varieties. On the other hand, we can also obtain some new consequences on the integral singular cohomology of real partial flag manifolds. 

\subsection{Relationship to Chow--Witt rings}
\label{sec:relationship}
We first formulate the description of the total $\mathbf{I}$-cohomology ring of the partial flag varieties, which is a direct consequence of the description of the total $\mathbf{W}$-cohomology ring in Theorem~\ref{thm:main}. 

\begin{theorem}
  \label{thm:icohomology}
  For any partition $\D=(d_1,\dots,d_m)$, there is an isomorphism of graded  $\op{W}(F)$-modules
  \[
  \bigoplus_{q,\L}\op{H}^q\bigl(\op{Fl}(\D),\mathbf{I}^q(\L)\bigr)\cong \bigoplus_{\L}\op{Im}(\beta_{\L})\oplus\bigoplus_{q,\L} \op{H}^q\bigl(\op{Fl}(\D),\mathbf{W}(\L)\bigr).
  \]
  Here $\beta_{\L}\colon {\rm Ch}^q\bigl(\Fl(\D)\bigr)\to {\rm H}^{q+1}\bigl(\Fl(\D),\mathbf{I}^{q+1}(\L)\bigr)$ is the twisted Bockstein map, cf.~Section~\ref{sec:prelims}. 
\end{theorem}

\begin{proof}
  This follows directly from Theorem~\ref{thm:main} (which implies in particular that Witt-sheaf cohomology is a free $\op{W}(F)$-module) using Lemma~\ref{lem:wsplit}. 
\end{proof}

Recall from Proposition~\ref{prop:fiberprod} that the Chow--Witt groups for a projective homogeneous variety for a reductive group have a pullback description:
\[
\xymatrix{
	&\widetilde{\op{CH}}^q\bigl(\Fl(\D),\L\bigr) \ar[r] \ar[d] 
	& \ker \partial_{\L}\subseteq \op{CH}^q\bigl(\Fl(\D)\bigr)\ar[d] &
	\\
	&\op{H}^q\bigl(\Fl(\D),\mathbf{I}^q(\L)\bigr) \ar[r] &
	\op{Ch}^q\bigl(\Fl(\D)\bigr)&\phantom{a} \hspace{-2 cm} .
}
\]
Therefore, Theorem~\ref{thm:icohomology} reduces the computation of the Chow--Witt ring of a flag variety $\Fl(\D)$ to a description of the maps in the pullback diagram. The right-hand vertical morphism is well-understood from classical intersection theory (plus some possibly less well-known information on Steenrod squares as in \cite{duan:zhao}, \cite[Proposition 6.3]{matszangosz}). For the lower vertical morphism, we have the following partial description of the reduction morphism which follows from the case of Grassmannians, cf.~\cite{realgrassmannian}. 

\begin{proposition}
  \label{prop:reduction}
  Fix $\D=(d_1,\dots,d_m)$. 
  Let $\L$ be any line bundle on the flag variety $\op{Fl}(\D)$, and set $N=\sum_{i=1}^md_i$, $n=\sum_{i=1}^m\lfloor\frac{d_i}{2}\rfloor$. As before, denote by $\DD_j$ the subquotient bundle of rank $d_j$ over $\op{Fl}(\D)$. The reduction morphism is given as follows
  \begin{eqnarray*}
    \rho\colon \op{H}^q\bigl(\op{Fl}(\D),\mathbf{I}^q(\L)\bigr)&\to& \op{Ch}^q\bigl(\op{Fl}(\D)\bigr)\colon\\
    \op{p}_i(\DD_j) & \mapsto & \overline{\op{c}}_i(\DD_j)^2 \\
    \op{e}(\DD_j) & \mapsto & \overline{\op{c}}_{d_j}(\DD_j) \\
    \beta_{\L}(x) & \mapsto & \op{Sq}^2_{\L}(x) \hspace{0.9 cm}.
  \end{eqnarray*}
  The reduction morphism is injective on the image of $\beta_{\L}$. In particular, multiplication with classes in $\op{Im}\beta_{\L}$ can be determined by reduction to $\op{Ch}^*\bigl(\op{Fl}(\D)\bigr)$ where computations can be done via classical Schubert calculus.
\end{proposition}

It remains to describe the reduction of the classes $\op{R}_j$. At this point, these are only well-defined generators in $\mathbf{W}$-cohomology, but the reduction map $\rho\colon{\rm H}^q(X,\mathbf{I}^q)\to {\rm Ch}^q(X)$ doesn't factor through ${\rm H}^q(X,\mathbf{W})$. To compute the reductions of the classes $\op{R}_j$, we therefore first need to lift them along the projection ${\rm H}^q\bigl(\Fl(\D),\mathbf{I}^q\bigr)\to {\rm H}^q\bigl(\Fl(\D),\mathbf{W}\bigr)$. Since the $\op{R}_j$ aren't characteristic classes, it is difficult to determine lifts in a canonical way.

As an alternative, we can define the following extension of the reduction map defined on $\mathbf{W}$-cohomology. To fix notation, recall that the Steenrod square ${\rm Sq}^2$ makes ${\rm Ch}^q(X)$ into a complex
\[
\cdots\xrightarrow{\op{Sq}^2}{\rm Ch}^{q-1}(X)\xrightarrow{\op{Sq}^2}{\rm Ch}^{q}(X)\xrightarrow{\op{Sq}^2}{\rm Ch}^{q+1}(X)\xrightarrow{\op{Sq}^2}\cdots
\]
of $\mathbb{F}_2$-vector spaces, whose cohomology is an algebraic version of Bockstein cohomology which we denote by ${\rm H}^q_{\op{Sq}^2}(X)$.

\begin{proposition}
  For any smooth $F$-scheme $X$,  the reduction map $\rho\colon{\rm H}^q(X,\mathbf{I}^q)\to {\rm Ch}^q(X)$ factors through a well-defined homomorphism 
\[
\overline{\rho}\colon {\rm H}^q(X,\mathbf{W})\to {\rm H}^q_{\op{Sq}^2}(X).
\]
\end{proposition}

\begin{proof}
  We first note that the commutative diagram
\[
\xymatrix{
  {\rm Ch}^{q-1}(X) \ar[r]^{\beta} \ar[d]_= & {\rm H}^q(X,\mathbf{I}^q) \ar[d]_\rho \\
  {\rm Ch}^{q-1}(X) \ar[r]_{\op{Sq}^2}  & {\rm Ch}^q(X) 
}
\]
(which appears as a lower-right triangle in the key diagram on p.~\pageref{keydiagram}) implies that the image of the reduction map $\rho\colon{\rm H}^q(X,\mathbf{I}^q)\to {\rm Ch}^q(X)$ is contained in $\ker\left(\op{Sq}^2\colon{\rm Ch}^q(X)\to {\rm Ch}^{q+1}(X)\right)$, using that $\op{Sq}^2\circ\op{Sq}^2=0$. We can now extend the diagram to the right to have exact rows:
\[
\xymatrix{
&  {\rm Ch}^{q-1}(X) \ar[r]^{\beta} \ar[d]_= & {\rm H}^q(X,\mathbf{I}^q) \ar[r] \ar[d]_\rho & {\rm H}^q(X,\mathbf{W}) \ar[r] \ar@{.>}[d]^{\overline{\rho}} & 0&\\
 & {\rm Ch}^{q-1}(X) \ar[r]_{\op{Sq}^2}  & \ker\op{Sq}^2 \ar[r] & {\rm H}^q_{\op{Sq}^2}(X) \ar[r] & 0&\phantom{a}\hspace{-2 cm}.
}
\]
The top row is the B\"ar sequence, and the bottom row is the exact sequence defining ${\rm H}^q_{\op{Sq}^2}(X)$. The exactness of the top row implies the existence of the well-defined map $\overline{\rho}\colon{\rm H}^q(X,\mathbf{W})\to {\rm H}^q_{\op{Sq}^2}(X)$. 
\end{proof}

In the particular case of flag varieties $\Fl(\D)$, the $\mathbf{I}$-cohomology splits as described in Theorem~\ref{thm:icohomology}. The following proposition shows that the reduction morphism is compatible with this splitting, mapping the image of $\beta$ to the image of $\op{Sq}^2$, and Witt-sheaf cohomology to Bockstein cohomology.

\begin{proposition}
  Let $X$ be a smooth $F$-scheme such that the $\mathbf{W}$-cohomology of $X$ is free as a ${\rm W}(F)$-module. Then the reduction morphism $\rho\colon \op{H}^n\bigl(X,\mathbf{I}^n(\L)\bigr)\to \op{Ch}^n(X)$ fits into a commutative diagram of compatible splittings
  \[\xymatrix{
    &	\op{H}^n\bigl(X,\mathbf{I}^n(\L)\bigr)\ar[r]^-{\iso}\ar[d]_{\rho}& \op{Im}\beta_{\L}\oplus \op{H}^n\bigl(X,\mathbf{W}(\L)\bigr)\ar[d]^-{(\rho_I,\overline{\rho})}&\\
    &	\Ch^n(X)\ar[r]_-{\iso}& \op{Im}\Sq^2_{\L}\oplus {\rm H}^n_{\op{Sq}^2_{\L}}(X)%& \phantom{a} \hspace{-2 cm} ,
  }
  \]
  where $\rho_I=\rho|_{\op{Im} \be_{\L}}$ is an isomorphism. 
\end{proposition}

\begin{proof}
  To see the final isomorphism claim, we note that $\op{Sq}^2_{\L}=\rho\circ\beta_{\L}$, hence $\rho_I$ surjects onto $\op{Im}\op{Sq}^2_{\L}$. The reduction map $\rho\colon{\rm H}^q(X,\mathbf{I}^q)\to {\rm Ch}^q(X)$ is injective on the image of $\beta$, cf.~Lemma~\ref{lem:wsplit}. 

  Combining this with the exact B\"ar sequence
\[
\cdots\to \op{H}^q(X,\mathbf{I}^{q}) \xrightarrow{\rho} \op{Ch}^q(X) \xrightarrow{\beta} \op{H}^{q+1}(X,\mathbf{I}^{q+1})\to\cdots
\]
we see that a class $\sigma\in{\rm Ch}^q(X)$ lifts along $\rho$ to ${\rm H}^q(X,\mathbf{I}^q)$ if and only if $\op{Sq}^2(\sigma)=0$. This means that the image of the reduction map $\rho\colon{\rm H}^q(X,\mathbf{I}^q)\to {\rm Ch}^q(X)$ is exactly the kernel of $\op{Sq}^2\colon{\rm Ch}^q(X)\to {\rm Ch}^{q+1}(X)$. The surjectivity of $\rho$ to $\ker \Sq^2$ also implies that the induced map $\overline{\rho}\colon {\rm H}^q(X,\mathbf{W})\to {\rm H}^q_{\op{Sq}^2}(X)$ is also surjective, which implies that we can choose a splitting $\ker\op{Sq}^2\cong \op{Im}\op{Sq}^2\oplus\op{H}^q_{\op{Sq}^2}(X)$ compatibly with the splitting ${\rm H}^q(X,\mathbf{I}^q)\cong\op{Im}\beta\oplus {\rm H}^q(X,\mathbf{W})$. Using these splittings, the reduction map $\rho\colon{\rm H}^q(X,\mathbf{I}^q)\to\ker\op{Sq}^2$ is the sum of the isomorphism $\rho_I\colon \op{Im}\beta\xrightarrow{\cong}\op{Im}\op{Sq}^2$ and the induced morphism $\overline{\rho}\colon{\rm H}^q(X,\mathbf{W})\to {\rm H}^q_{\op{Sq}^2}(X)$.
\end{proof}

The result applies in particular to $X=\Fl(\D)$ since the $\mathbf{W}$-cohomology of flag varieties is free, cf.~Theorems~\ref{thm:untwisted} and \ref{thm:main}. As a consequence, for the description of the reductions of the classes ${\rm R}_j\in{\rm H}^*\bigl(\Fl(\D),\mathbf{W}\bigr)$, it suffices to determine their images under the induced reduction map $\overline{\rho}\colon{\rm H}^q\bigl(\Fl(\D),\mathbf{W}\bigr)\to {\rm H}^q_{\op{Sq}^2}\bigl(\Fl(\D)\bigr)$ which is the natural reduction map ${\rm W}(F)\to {\rm W}(F)/{\rm I}(F)\cong\mathbb{Z}/2\mathbb{Z}$ on each summand.

When it comes to computing $\overline{\rho}(\op{R}_j)$, given that the ${\rm W}(F)$-rank of ${\rm H}^q\bigl(\Fl(\D),\mathbf{W}\bigr)$ and the $\mathbb{Z}/2\mathbb{Z}$-rank of $\op{H}^q_{\op{Sq}^2}\bigl(\Fl(\D)\bigr)$ are equal, we now see that the possible images in $\op{H}^q_{\op{Sq}^2}\bigl(\Fl(\D)\bigr)$ are very much restricted. We can first divide out the contributions from the ${\rm W}(F)$-subalgebra generated by Pontryagin classes, where the reductions are known from Proposition~\ref{prop:reduction}. Then with the exception of $\op{R}_{N-1}$, the $\op{R}_j$ classes are ordered by strictly increasing cohomological degree.
 So we can determine the reduction of the $\op{R}_{j}$-classes as follows: for every cohomological degree $q$, consider the quotient of $\op{H}^q_{\op{Sq}^2}\bigl(\Fl(\D)\bigr)$ modulo the subgroup of classes in the image of the ${\rm W}(F)$-subalgebra of $\mathbf{W}$-cohomology generated by Pontryagin classes and $\op{R}_j$-classes of smaller index and possibly $\op{R}_{N-1}$. This is necessarily isomorphic to $\mathbb{Z}/2\mathbb{Z}$. Using the images of a ${\rm W}(F)$-basis of ${\rm H}^q\bigl(\Fl(\D),\mathbf{W}\bigr)$ in $\op{H}^q_{\op{Sq}^2}\bigl(\Fl(\D)\bigr)$, the generator of this $\mathbb{Z}/2\mathbb{Z}$ has a unique representative in $\op{H}^q_{\op{Sq}^2}\bigl(\Fl(\D)\bigr)$ which is $\overline{\rho}(\op{R}_j)$. This procedure allows to inductively compute the reductions of the $\op{R}_j$ classes.

\subsection{Relationship to singular cohomology}

Our computations also have consequences for the integral cohomology of real flag manifolds. Recall from \cite{4real} that there is a real cycle class map
\[
\bigoplus_{q,\L}\op{H}^q\bigl(X,\mathbf{I}^q(\L)\bigr)\to \bigoplus_{q,\L}\op{H}^q\bigl(X(\mathbb{R}),\mathbb{Z}(\L)\bigr)
\]
from the $\mathbf{I}$-cohomology of a real variety $X/\mathbb{R}$ to the singular cohomology of the manifold $X(\mathbb{R})$ of real points. Here $\op{H}^q\bigl(X(\mathbb{R}),\mathbb{Z}(\L)\bigr)$ denotes integral singular cohomology with coefficients in the local system corresponding to the line bundle $\L$ on $X$. The real cycle class map is compatible with pushforward, the ring structure and characteristic classes. Moreover, for cellular real varieties, the real cycle class map is an isomorphism. As a direct consequence, we get the following result: 

\begin{theorem}
  \label{thm:realflag}
  Let $\D=(d_1,\dots,d_m)$. For all line bundles $\L$ over the real flag manifold $\op{Fl}(\D,\mathbb{R})$, all torsion in $\op{H}^*\bigl(\op{Fl}(\D,\mathbb{R}),\mathbb{Z}(\L)\bigr)$ is of order 2. The torsion is given by the image of the twisted Bockstein map $\beta_{\L}$. The torsion-free quotient of the integral cohomology is isomorphic as graded-commutative $\mathbb{Z}$-algebra to $\op{H}^*\bigl(\op{Fl}(\D),\mathbf{W}(\L)\bigr)$ which is described in Theorem~\ref{thm:main}.  
\end{theorem}

While the additive structure of the integer coefficient cohomology of certain partial flag manifolds (projective spaces, Grassmannians, maximal rank flag manifolds \cite{matszangosz} and most recently complete flag manifolds \cite{yang}) are known, the above result in full generality appears not to be known. Knowing that all torsion is 2-torsion allows to completely determine the additive structure of cohomology (at least in principle), and we will discuss the example case of the complete flag manifolds in the next subsection.

\subsection{The case of complete flag manifolds}
 Based on the fact that all torsion is 2-torsion, the additive structure of the cohomology of complete flag manifolds $\Fl(\mathbb{R}^N)$ can be described by their Poincar\'e polynomials. We denote the Poincar\'e polynomials of the free part and torsion part of $\op{H}^*\bigl(X;\Z(\L)\bigr)$ by ${\rm P}^{\L}_0(X,t)$ and ${\rm P}^{\L}_{\op{Tor}}(X,t)$ respectively. By Poincar\'e polynomial of the torsion part, we mean the Poincar\'e polynomial of the ranks of the $\Z/2\Z$ summands.

 \begin{theorem}\label{thm:completeflag}
The Poincar\'e polynomials of $\op{H}^*(\Fl(\R^N);\Z)$ are as follows:
\begin{equation}\label{eq:WPoincare}
	P_{0}\bigl(\Fl(2k+1),t\bigr)= \prod_{i=1}^k (1+t^{4i-1}),\qquad P_{0}\bigl(\Fl(2k),t\bigr)= (1+t^{2k-1})\prod_{i=1}^{k-1} (1+t^{4i-1}),
\end{equation}
\begin{equation}\label{eq:PoincareTor1}
P_{\op{Tor}}\bigl(\Fl(2k+1),t\bigr)=\frac{t}{t+1}\left(
\frac{\prod_{j=1}^{2k+1} (1-t^j)}{(1-t)^{2k+1}}-\prod_{i=1}^k (1+t^{4i-1})
\right),
\end{equation}
\begin{equation}\label{eq:PoincareTor0}
P_{\op{Tor}}\bigl(\Fl(2k),t\bigr)=\frac{t}{t+1}\left(
\frac{\prod_{j=1}^{2k} (1-t^j)}{(1-t)^{2k}}-(1+t^{2k-1})\prod_{i=1}^{k-1} (1+t^{4i-1})
\right).
\end{equation}
For any non-trivial twist $\L\in \op{Pic}\bigl(\Fl(N)\bigr)/2$, the $\mathbf{W}(\L)$-cohomology is 0, and the entire $\Z(\L)$-cohomology consists of 2-torsion, whose ranks are given by
\begin{equation}\label{eq:PoincareTortwisted}
 P_{\op{Tor}}^{\mathscr L}\bigl(\Fl(N),t\bigr)=\frac{t}{t+1}\frac{\prod_{j=1}^N (1-t^j)}{(1-t)^N}.
\end{equation}
\end{theorem}

 \begin{proof}
By Proposition~\ref{prop:poincare-mod2}, the Poincar\'e polynomial of the Chow ring mod 2 of $\Fl(N)$ is
$$P_2\bigl(\Fl(N),t\bigr)= \frac{\prod_{j=1}^N (1-t^j)}{(1-t)^N}.$$
For $\Fl(N)$, the $\mathbf{W}$-cohomology is an exterior algebra on generators $x_3,x_7\ldots $ cf.\ Theorem~\ref{thm:main} and Theorem~\ref{thm:untwisted}, whose Poincar\'e polynomials are exactly given by \eqref{eq:WPoincare}.
For any space whose integer coefficient cohomology consists of 2-torsion, $P_{Tor}^{\mathscr L}=\frac{t}{t+1}(P_2-P_0^{\mathscr L})$, \cite[Lemma 24.1]{Borel}. Since for flag varieties all torsion is 2-torsion by Theorem~\ref{thm:realflag}, we get the Poincar\'e polynomial of the ranks of $\Z/2\Z$'s  \eqref{eq:PoincareTor1}, \eqref{eq:PoincareTor0} and \eqref{eq:PoincareTortwisted}.
\end{proof}

\begin{example}
	For $\Fl(3)$, we have
	$$ P_2\bigl(\Fl(3),t\bigr)=(1-t)(1-t^2)(1-t^3)/(1-t)^3=(1+t)(1+t+t^2)=1+2t+2t^2+t^3,$$
	$$ P_{0}\bigl(\Fl(3),t\bigr)=1+t^3,$$
	$$ P_{\op{Tor}}\bigl(\Fl(3),t\bigr)=\frac{t}{t+1}(1+2t+2t^2+t^3-1-t^3)=2t^2.$$
Since for all twists $\mathbf{W}(\L)$-cohomology is 0:
$$ P_{\op{Tor}}^{\L}\bigl(\Fl(3),t\bigr)=\frac{t}{t+1}(1+2t+2t^2+t^3)=t+t^2+t^3.$$
\end{example}

\begin{example}
	For $\Fl(4)$, we have
	$$ P_2\bigl(\Fl(4),t\bigr)=(1+t)(1+t+t^2)(1+t+t^2+t^3)=1+3t+5t^2+6t^3+5t^4+3t^5+t^6,$$
	$$ P_{0}\bigl(\Fl(4),t\bigr)=(1+t^3)^2=1+2t^3+t^6,$$
	$$ P_{\op{Tor}}\bigl(\Fl(4),t)=\frac{t}{t+1}(3t+5t^2+4t^3+5t^4+3t^5)=
t(3t+2t^2+2t^3+3t^4),
	$$
	in accordance with \cite[8.1]{casianstanton}, see also \cite[Section 3.10]{matszangosz}. Since for all twists $\mathbf{W}(\L)$-cohomology is 0:
$$ P_{\op{Tor}}^{\L}\bigl(\Fl(4),t\bigr)=\frac{t}{t+1}(1+3t+5t^2+6t^3+5t^4+3t^5+t^6)=
t+2t^2+3t^3+3t^4+2t^5+t^6.
$$	
\end{example}

\section{Flags incident to hypersurfaces}
\label{sec:enumerative}

A classical problem in enumerative geometry is the count of lines on a cubic surface - its history goes back to the 19th century. For a nice and thorough account of the history of this problem, see \cite{dolgachev}. The solution over the complex field is due to Cayley and Salmon. The real case is less straightforward, as the number of lines now depends on the surface, and its solution is due to Schl\"afli \cite{schlafli}. B.\ Segre  \cite{segre} classified the lines into two types. Topological interpretations of these results were developed later by Finashin and Kharlamov \cite{finashinkharlamov1}, and Okonek and Teleman \cite{okonekteleman}. More recently Kass and Wickelgren \cite{kasswickelgren} gave arithmetic extensions of these results, giving solutions over general fields and in particular also in characteristics other than $0$.

Over the complex field the `modern' solution of the problem is the computation of a Chern class of a certain bundle over a Grassmannian $\Gr_2(\C^4)$. Over the reals - as demonstrated by Finashin, Kharlamov \cite{finashinkharlamov1}, Okonek and Teleman \cite{okonekteleman} - a lower bound is obtained by a similar computation of an Euler class over the Grassmannian $\Gr_2(\R^4)$. Finashin and Kharlamov also gave a series of similar enumerative problems over larger Grassmannians \cite{finashinkharlamov2}. Other enumerative problems can be formulated over general Grassmannians; topological lower bounds in real Schubert calculus problems have been obtained in \cite{FeherMatszangoszschubert}, \cite{heinhillarsottile} and \cite{FeherMatszangosz}. Arithmetic results on Schubert calculus in Chow--Witt rings have been developed in \cite{schubert}.

Underlying the Chern class formulation of the lines on the cubic surface is the question of counting $2$-dimensional linear subspaces in $\mathbb{A}^4$ contained in a hypersurface $X\se \mathbb{A}^4$ defined by a degree 3 homogeneous polynomial $f$ on $\mathbb{A}^4$. Indeed, after projectivization such subspaces are just lines on the cubic surface. On the other hand, the polynomial $f$ provides a section $s_f$ of the bundle $\op{Sym}^3(S^\vee)$ on ${\rm Gr}(2,4)$ via restriction. A subspace $W\se \mathbb{A}^4$ is contained in $X$ exactly if the restriction of $f$ to $W$ is zero, so the zeros of the section $s_f$ correspond exactly to the lines contained in the cubic surface. The fundamental class of such a zero-locus is described by the Euler class of the bundle in case the section $s_f$ is generic, which holds in case $X$ is smooth. 

\subsection{Enumerative problems in flag varieties}
In the context of the paper, we want to consider similar enumerative problems over flag varieties:

\begin{question}
  \label{quest}
  Given a partition $\D=(d_1,\dots,d_m)$ of $N=\sum_{i=1}^md_i$, how many flags $F_\bullet$ in $\mathbb{A}^N$ of type $\D$ exist, so that $F_1\se S_1$, $F_2\se S_2\stb F_m\se S_m$, where each $S_i\subseteq\mathbb{A}^N$ is a general  hypersurface defined by a homogeneous polynomial of degree $\ell_i$?
\end{question}

We are interested in such questions when the expected dimension of the set of solutions is zero, so that for suitably generic problems, the number of solutions is finite.

Over the complex numbers, such enumerative problems can be formulated in terms of intersection multiplicities for subvarieties in a smooth variety $X$, or in terms of top Chern classes of suitable bundles $\EE$ arising from the enumerative problem. We can extract numbers, the solutions to the enumerative problem, from such an intersection product or top Chern class, whenever they are elements in a top Chow group $\op{CH}^{\dim X}(X)$ which is isomorphic to $\Z$, e.g.\ for Grassmannians or flag varieties. In this case, a natural choice of isomorphism ${\rm CH}^{\dim X}(X)\cong\mathbb{Z}$ is given by mapping 1 to the class of a point. 

Similarly, in real enumerative or refined enumerative contexts, the problems can be formulated in terms of Euler classes of suitable bundles, e.g.\ as in \cite{kasswickelgren}, or in terms of intersection products in Chow--Witt rings. In the latter case, to extract quadratic forms counting the solutions of the enumerative problem requires that the intersection product is an element in a top Chow--Witt group $\widetilde{\rm CH}^{\dim X}(X)$ which is isomorphic to ${\rm GW}(F)$, the Grothendieck--Witt ring of the base field. This is the case for Grassmannians ${\rm Gr}(k,n)$ whenever $n$ is even, cf.~\cite{realgrassmannian}, or for maximal rank flag varieties, cf.~Theorem~\ref{thm:maxrank}. A noteworthy difference in the Chow--Witt situation is that the choice of isomorphism $\widetilde{\rm CH}^{\dim X}(X)\cong{\rm GW}(F)$ depends very much on the choice of orientation; there is no natural ``class of a point'' without a choice of orientation of the tangent bundle.\footnote{Over the real numbers, this choice only affects a sign and could be seen as less significant. However, over arbitrary base fields, the Grothendieck--Witt ring might have many units, resulting in a potentially rather high number of different possibilities for classes of points.} If such a choice is made, the resulting quadratic forms provide the solutions to the enumerative problems: the rank recovers the number of solutions over algebraically closed fields, and the signature recovers signed counts from real enumerative geometry which provide in particular lower bounds for the enumerative problem over the real numbers.

The purpose of the present section is to sketch this procedure in a simple case, to give an idea how Euler class computations for relevant bundles on flag varieties or intersection-product computations could be done, and to illustrate the structure of Witt-sheaf cohomology and Chow--Witt rings of flag varieties described in the paper. To fix the setting, we want to consider an instance of Question~\ref{quest} where the resulting quadratic form will not be a multiple of the hyperbolic plane. To get this, several additional constraints arise and we summarize the conditions:\footnote{Admittedly, what we are considering is not the utmost generality. The first two conditions can be generalized to relatively orientable vector bundles by taking twisted coefficients and assuming that $\det \EE=\det TX$, but we will only consider this simpler setting.}
\begin{itemize}
	\item $X$ should be orientable,
	\item $\EE$ should be orientable,
	\item $\EE$ should be of even rank (otherwise ${\rm e}(\EE)=0$) and
	\item $\op{rk}(\EE)=\dim X$.	
\end{itemize}

\subsubsection{The case of flag varieties}
We are interested in flag varieties $\op{Fl}(\D)$, and bundles of the form $\EE=\bigoplus_{i\in I}\op{Sym}^{\ell_i}(\SS_i^\vee)$ which (as in the lines on smooth cubics case) control when the $i$-th subspace in the flag lies on a hypersurface of degree $\ell_i$.

To find a setting in which the above conditions on orientability etc are satisfied, first let $\SS\to X$ be an arbitrary bundle of rank $s$ and let $\EE:=\op{Sym}^{\ell}(\SS)$. Then by the proof of \cite[Lemma 3.1.2]{finashinkharlamov2}, 
\begin{equation}\label{eq:rkw1Sym}
  {\rm c}_1(\EE)=\binom{\ell+s-1}{s}\cdot {\rm c}_1(\mathcal{S}),\qquad \op{rk}(\EE)=\binom{\ell+s-1}{s-1}.
\end{equation}
So -- unless $\mathcal{S}$ itself is orientable -- $\EE$ is orientable iff $\binom{\ell+s-1}{s} c_1(\SS)\in\op{Pic}(X)/2$ is zero, i.e.\ $\binom{\ell+s-1}{s}$ is even.

We now want to restrict to the simplest (non-Grassmannian) case of such an enumerative problem in which only two incidence conditions appear, i.e., we want to compute the Euler class of a bundle 
$$\EE=\op{Sym}^{\ell_1}(\SS_1^\vee)\oplus \op{Sym}^{\ell_2}(\SS_2^\vee).$$
over a two-step flag variety $\Fl(2,2,d)$. The conditions listed above translate to the following numerical conditions on $d$ and $(\ell_1,\ell_2)$:
\begin{itemize}
	\item $d$ must be even,
	\item $\binom{\ell_1+1}{2}$ and $\binom{\ell_2+3}{4}$ must be even, cf.\ \eqref{eq:rkw1Sym},
	\item $\ell_1+1$ and $\binom{\ell_2+3}{3}$ must be even, cf.\ \eqref{eq:rkw1Sym},
	\item $\ell_1+1+\binom{\ell_2+3}{4}=4(d+1)$.
\end{itemize}
The two parity conditions on $\ell_2$ show that $\ell_2=5$ is the smallest possible choice, and the remaining smallest choices are $\ell_1=3$ and $d=14$.
In particular, the smallest problem satisfying these conditions can be formulated as follows:
\begin{question}
	How many two-step flags $A_2\se A_4\se \mathbb{A}^{18}$ exist, s.t.\ $A_2$ lies on a degree 3 hypersurface, and $A_4$ lies on a degree 5 hypersurface?
\end{question}

Essentially, the answer is obtained by computing the Euler number of the vector bundle $\op{Sym}^{3}(\SS_1^\vee)\oplus \op{Sym}^{5}(\SS_2^\vee)$ on $\op{Fl}(2,2,14)$. In the next two subsections we will describe the computations to solve this problem over the real and complex numbers, and sketch the refined enumerative solution over general fields. Some general statements concerning Euler classes of symmetric and tensor powers of vector bundles which we'll need for the computations will be postponed to the last subsection.

\subsection{Solution of the problem over complex and real numbers}

\begin{theorem}\label{thm:enumerativeexample}
	There are at least $\mathcal{M}_\R=24681637575$ real flags $(U^2\se W^4\se \R^{18})$ so that $U^2$ lies on a generic degree $3$ hypersurface $H_3$ and $W^4$ lies on a generic degree 5 hypersurface $H_5$.
\end{theorem}
The proof of this theorem relies on computing the coefficient of the class of a point in the Euler class
$$ {\rm e}\bigl(\op{Sym}^3(\SS_1^\vee)\bigr)\cdot {\rm e}(\op{Sym}^5\bigl(\SS_2^\vee)\bigr)$$
in the top-degree singular cohomology of the real flag manifold $\Fl(2,2,14;\mathbb{R})$.\footnote{Similar formulas will hold in Chow--Witt-theory or Witt-sheaf cohomology, we will discuss these in the next subsection.}
This consists of two parts -- expressing the Euler class in terms of the variables $x_i={\rm p}_2(\DD_i)$ and computing the coefficient of the class of a point -- we carry them out in this order. 

For the computation of $\op{Sym}^5(\SS_2^\vee)$, we can use the splitting principle and compute $\op{Sym}^5$ of a direct sum of two rank 2 vector bundles $A$ and $B$, with Euler classes denoted by $a$ and $b$, respectively. Then
\begin{equation}\label{eq:Symdirectsum}
	{\rm e}\bigl(\op{Sym}^k(A\oplus B)\bigr)=\prod_{i=0}^k{\rm e}\left(\op{Sym}^i(A)\otimes \op{Sym}^{k-i}(B)\right).
\end{equation}

  We will give the computation $e_i:={\rm e}\left(\op{Sym}^{5-i}(A)\otimes \op{Sym}^i(B)\right)$ in Example~\ref{ex:Sym5rk2} below, after recalling some relevant Euler class formulas from \cite{levine:schur}. Using these computations, the product of the $e_i$'s for $i=0,1,2$ is
	$$ p(a,b):=15a^3\bigl((9a^4-40a^2b^2+16b^4)3a^2\bigr)\bigl((64a^4-20a^2b^2+b^4)b\bigr),$$
	%SAGE:
	%vars('a b')
	%f = lambda a, b : 15*a^3*((9*a^4-40*a^2*b^2+16*b^4)*3*a^2)*((64*a^4-20*a^2*b^2+b^4)*b)
	%expand(-f(a,b)*f(b,a))
	and then the product for $i=3,4,5$ is $-p(b,a)$ by symmetry. Then 
by \eqref{eq:Symdirectsum} the product over $i=0\stb 5$ is $-p(a,b)p(b,a)$, which is
\begin{equation}
  \begin{split}
    18662400a^{22}b^6 - 508680000a^{20}b^8 + 4194860400a^{18}b^{10}\\ - 14714257500a^{16}b^{12} + 22941470025a^{14}b^{14} - 14714257500a^{12}b^{16} \\+ 4194860400a^{10}b^{18} - 508680000a^8b^{20} + 18662400a^6b^{22}.
  \end{split}
  \label{eq:Sym5}
\end{equation}

We now describe the class of a point in the Chow ring of the complex flag variety $\op{Fl}(1,1,n)$ and the corresponding description of the class of a point in the singular cohomology of the real flag varieties $\op{Fl}(2,2,2n)$. 

\begin{proposition}
  \label{prop:fl11n}
  In the complex case:
  $$ \op{H}_{\rm sing}^*\bigl(\Fl(1,1,n;\mathbb{C});\mathbb{Z}\bigr)=\Z\left[ x_1^ix_2^j\mid i,j\leq n+1\right]\bigg/\left(\sum_{i=0}^dx_1^ix_2^{d-i}=0, \quad d>n\right),$$
  where $x_i={\rm c}_1(\DD_i)$ and the class of a point is (up to sign)
  $$[pt]= -x_1^{n+1}x_2^n=x_1^{n}x_2^{n+1}.$$
\end{proposition}

\begin{proof}
  More generally, the cohomology ring $\op{H}^*\bigl(\Fl(1^k,n-k)\bigr)$ has the following presentation:
  $$ \Z\big[x_1\stb x_k,c_1\stb c_{n-k}\big]\bigg/\left(
  c_i=(-1)^ih_i(x),\sum_{i=1}^k e_i(x)c_{d+(n-k)-i}=0
  \right)$$
  and the class of a point by Proposition~\ref{prop:classofpoint} is
  \begin{equation}\label{eq:eulerclassofpoint}
    {\rm c}_{\rm top}\bigl(\Hom(\DD_1,\C^{n+1})\bigr)\cdot {\rm c}_{\rm top}\bigl(\Hom(\DD_2,\C^{n})\bigr)= \prod_{i=1}^{n+1}(0-x_1)\prod_{i=1}^n(0-x_2)=-x_1^{n+1}x_2^n
  \end{equation}
\end{proof}

\begin{corollary}\label{cor:monomialvanishing}
  In the notation of Proposition~\ref{prop:fl11n}
  $$ x_1^ix_2^{2n+1-i}=0$$
  for all $i\neq n, n+1$.
\end{corollary}

\begin{proposition}\label{prop:realpoint}
  In the real case:
  $$\op{H}_{\rm sing}^*\bigl(\Fl(2,2,2n;\mathbb{R});\mathbb{Q}\bigr)=\Q\big[x_1^ix_2^j\mid i,j\leq n+1\big]\bigg/\left(\sum_{i=0}^dx_1^ix_2^{d-i}=0, \quad d>n\right)$$
  where $x_i={\rm p}_2(\DD_i)$ and the class of a point is (up to sign)
  $$[pt]= -x_1^{n+1}x_2^n=x_1^{n}x_2^{n+1}.$$
\end{proposition}

\begin{proof}
  For the rational coefficient cohomology ring, see \cite{he}. For the statement about the class of the point, the only difference from the complex case in Proposition~\ref{prop:classofpoint} is that the ranks are doubled. However, for a rank two bundle $\DD$ by Corollary~\ref{cor:tensorrk2}:
  $$ {\rm e}\bigl(\Hom(\DD,\A^{2k})\bigr)=\bigl(-{\rm p}_2(\DD)\bigr)^k,$$
  so the computation is formally the same as \eqref{eq:eulerclassofpoint}.
\end{proof}

\begin{proof}[Proof of Theorem~\ref{thm:enumerativeexample}]
  According to Proposition~\ref{prop:realpoint}, we have to determine the coefficient of 
  \[[pt]=x_1^{7}x_2^{8}\in {\rm H}^*_{\rm sing}\bigl(\Fl(2,2,14);\Q\bigr)\] in the product of $\op{e}\bigl(\op{Sym}^3(\DD_1^\vee)\bigr)=3x_1$ and ${\rm e}\bigl(\op{Sym}^5(\DD_1^\vee\oplus \DD_2^\vee)\bigr)$. The latter is given by Equation~\ref{eq:Sym5} with $x_1=a^2$ and $x_2=b^2$ if $a={\rm e}(\DD_1^\vee)$ and $b={\rm e}(\DD_2^\vee)$ (the top Pontryagin classes are the squares of the Euler classes). By Corollary~\ref{cor:monomialvanishing}, all coefficients besides $x_1^{8}x_2^{7}$ and $x_1^{7}x_2^{8}$ vanish, so we are interested in the coefficient of $x_1^{7}x_2^{8}$ in 
  \begin{eqnarray*}&&{\rm e}\bigl(\op{Sym}^3(\DD_1^\vee)\bigr){\rm e}\bigl(\op{Sym}^5(\DD_1^\vee\oplus \DD_2^\vee)\bigr)\\
    &=&3x_1\cdot(22941470025x_1^{7}x_2^{7} - 14714257500x_1^{6}x_2^{8}+\cdots)=24681637575x_1^{8}x_2^{7},
  \end{eqnarray*}
  which up to sign is the class of a point. The sign corresponds to a different orientation, which can be determined via a careful tracking of the orientations of the symmetric power and tensor product bundles which appeared so far. Since there is no additional benefit to this tedious calculation, we will omit this. 
%	8227212525$$ <- is just difference, multiply by 3!!!!
\end{proof}

\begin{proposition}\label{prop:upperbound}
  The number of complex flags $(U^2\se W^4\se \C^{18})$ so that $U^2$ lies on a generic degree $3$ hypersurface $H_3$ and $W^4$ lies on a generic degree 5 hypersurface $H_5$ is
  \[
  \mathcal{M}_\C=1731448582963698760147916022054375.
  \]
\end{proposition}

\begin{proof}
  Finashin and Kharlamov \cite[p.190]{finashinkharlamov2} showed that there are \[\mathcal{N}_5^\C=64127725294951805931404297113125\] many 4-subspaces $W^4\se \C^{18}$, which lie on a generic degree 5 hypersurface $H_5$. For each such subspace $W$, the number of 2-subspaces $U^2\se W^4$ which lie on a degree 3 hypersurface $H_3\cap W^4$ is 27 by the classical result. The product of these two numbers is equal to $\mathcal{M}_\C$.
\end{proof}

\begin{remark}
  The lower bound to the real problem $\mathcal{M}_\R$ in Theorem~\ref{thm:enumerativeexample} can also be obtained by a proof similar to Proposition~\ref{prop:upperbound}. By computing the Euler class of $\op{Sym}^5(S^\vee)$ over $\Gr_4(\R^{18})$, there are at least $822721252$ 4-subspaces $W^4\se \R^{18}$ on a generic degree 5 hypersurface. For each such $W$, there are at least 3 many 2-subspaces $U^2\se W$ lying on a generic degree 3 hypersurface $W\cap \op{H}_3$. In particular, the product of these two numbers is exactly the lower bound $\mathcal{M}_\R$ in Theorem~\ref{thm:enumerativeexample}.	
\end{remark}

\subsection{Sketch of refined enumerative interpretation}

\newcommand{\M}{\mathcal{M}}

We will now sketch a formulation of the above results in terms of refined enumerative geometry, similar to the result of \cite{kasswickelgren}.

We consider the flag variety $\Fl(\D)$ for $\D=(2,2,14)$ over a field $K$ of characteristic different from $2$, which has dimension 60. In the previous section, we computed the Euler numbers of the vector bundle
$$\EE=\op{Sym}^3(\SS_1^\vee)\oplus \op{Sym}^5(\SS_2^\vee)$$
over $\Fl(\D)$. Recall that the situation was chosen in such a way that $\Fl(\D)$ and $\EE$ are both orientable.

Over the complex numbers, the Euler class was $\M_{\mathbb{C}}\cdot[\op{pt}]$ in Proposition~\ref{prop:upperbound}, i.e., the Euler class in $\op{CH}^{\rm 60}\bigl(\Fl(\D)\bigr)$ is a multiple of the class of the point, and the multiplicity is the number of incident flags over $\mathbb{C}$. Over the real numbers, the result was $\M_{\mathbb{R}}$ in Theorem~\ref{thm:enumerativeexample}. The interpretation here is that the real flag manifold $\Fl(\D;\mathbb{R})$ is an orientable manifold, i.e., ${\rm H}^{60}_{\rm sing}\bigl(\Fl(\D;\mathbb{R}),\mathbb{Z}\bigr)\cong \mathbb{Z}$. Choosing an orientation amounts to a choice of such isomorphism, and the image of the Euler class under the isomorphism is $\M_{\mathbb{R}}$. In this case, the choice of isomorphism only affects a sign, and the resulting number $\M_{\mathbb{R}}$ can be interpreted as a lower bound of the number of zeroes of a generic section of $\EE$, i.e., a lower bound on the number of real solutions of the enumerative problem. 

Now for a general base field, our orientability assumptions imply that the Euler class of $\EE$ is an element of $\widetilde{\rm CH}^{60}\bigl(\Fl(\D)\bigr)$, and the latter group is isomorphic to the Grothendieck--Witt ring $\op{GW}(K)$ of the base field. While the choice of orientation mattered little over $\mathbb{R}$ and not at all over $\mathbb{C}$, it matters now, e.g., since $\op{GW}(K)$ may have many units, depending on the base field $K$. In the Grassmannian case, there are preferred choices of orientations, cf.~\cite{FeherMatszangoszschubert} or \cite{schubert}, but for now we'll not go into that for flag varieties. As in \cite{kasswickelgren}, the Euler class of $\EE$ can be written as a sum of local indices, and it is likely that under the isomorphism $\widetilde{\rm CH}^{60}\bigl(\Fl(\D)\bigr)\cong \op{GW}(K)$ corresponding to a suitable orientation, the following formula holds in $\op{GW}(K)$, (compare to similar formulas in Chow--Witt Schubert calculus \cite{schubert}):
\[
\sum_{\substack{F_\bullet \in \Fl(\D):\\
    F_1\se \op{H}_3, F_2\se \op{H}_5}}
\op{ind}F_\bullet=\frac{\M_\C +\M_\R}{2}\bra 1\ket+\frac{\M_\C -\M_\R}{2}\bra-1\ket.
\]
Here, $\M_\C$ and $\M_\R$ are the numbers of complex and real solutions computed in Theorem~\ref{thm:enumerativeexample} and Proposition~\ref{prop:upperbound}, respectively, and $\op{ind}F_\bullet$ is the local index for a zero of a generic section of $\EE$ at a point $F_\bullet$ of $\Fl(\D)$.

The right-hand side is a very explicit quadratic form encoding correctly the enumerative results over the complex and real numbers. It also seems likely that the phenomenon in the Grassmannian case that the right-hand side is always of the form $a\bra 1\ket+b\bra-1\ket$ persists in the case of type A flag varieties. The left-hand side is still very inexplicit (and actually implicitly requires choices of orientations of bundles and local charts as well). Comparing to other Schubert calculus situations, one could try to rewrite the Euler class computation as an intersection problem. In the smooth cubic line count, the computation of the Euler class can be rewritten as 4-fold self-intersection of the Schubert variety corresponding to ${\rm c}_1(\SS^\vee)$. Ignoring the singularities, the fundamental class for the Schubert variety is essentially given by a choice of orientation of the normal bundle, and the local index at an intersection point is the rank 1 quadratic form $\bra u\ket$ where $u$ is the unit comparing the orientations on the tangent bundle and its decomposition into normal bundles to the intersecting Schubert varieties. We hope to make this explicit in the sequel to this paper which will deal with Schubert varieties in flag varieties and their fundamental classes, choices of orientations etc.

\subsection{Computing Euler classes}

The final subsection now provides some general computations of Euler classes of symmetric powers of rank 2 bundles, which is used in the proof of Theorem \ref{thm:enumerativeexample}. In particular, the definition of $p(a,b)$ is based on Example \ref{ex:Sym5rk2}. These computations are mainly based on the following proposition of Levine, cf.\ \cite[Theorem 8.1, Proposition 9.1]{levine:schur}.

\begin{proposition}\label{prop:levine}
  Let $A,B$ be rank 2 bundles with Euler classes $a={\rm e}(A)$, $b={\rm e}(B)$. Then
  $$ {\rm p}({\rm Sym}^mA)=\prod_{i=0}^{\lfloor m/2\rfloor}\bigl(1+(m-2i)^2a^2\bigr),\qquad {\rm p}(A\otimes B)=1+2(a^2+b^2)+(a^2-b^2)^2,$$
  $$ {\rm e}({\rm Sym}^{2r+1}A)=(2r+1)!! a^r,\qquad {\rm e}(A\otimes B)=a^2-b^2.$$
\end{proposition}
We will now consider some higher rank analogues, following the outline already given in \cite[Remark 9.2]{levine:schur}.

\subsubsection{Tensor products}

In order to discuss Euler classes of tensor products we recall orientations for tensor product bundles, cf.\ \cite[Section 9]{levine:schur}.

\begin{definition}
  \label{def:tensor-orientation}
  Let $A,B$ be vector bundles with orientations $\omega_A\colon \det A\xrightarrow{\cong}\L^{\otimes 2}$ and $\omega_B\colon \det B\xrightarrow{\cong}\mathscr{M}^{\otimes 2}$. Then we define the orientation on $A\otimes B$ to be
  \[
  \omega_A\wedge\omega_B\colon \det(A\otimes B)\xrightarrow{\varphi_{AB}^{-1}} \det A\otimes \det B\xrightarrow{\omega_A\otimes\omega_B} \L^{\otimes 2}\otimes\mathscr{M}^{\otimes 2}\xrightarrow{\cong}(\L\otimes\mathscr{M})^{\otimes 2}.
  \]
  Here, the isomorphism $\varphi_{AB}\colon\det A\otimes \det B\to \det(A\otimes B)$ is induced by the bilinear map 
  \[	\begin{split}
    \varphi_{AB}\colon\bigwedge^k A\times \bigwedge^l B&\to \bigwedge^{k\cdot l}(A\otimes B)\\
    (a_1\wedge\ldots \wedge a_k,b_1\wedge\ldots \wedge b_l)&\mapsto a_1\otimes b_1\wedge \ldots \wedge a_1\otimes b_l\wedge\ldots
    \wedge a_k\otimes b_1\wedge \ldots \wedge a_k\otimes b_l
  \end{split}
  \]
  taken in the lexicographical order, and $k=\op{rk}A$ and $l=\op{rk} B$. The last isomorphism simply exchanges two tensor factors. 
\end{definition}

Note that this orientation and subsequently the Euler class depends on the order of $A$ and $B$. The following proposition describes how the Euler classes for $A\otimes B$ and $B\otimes A$ are related. 

\begin{proposition}
  If $A$ and $B$ are oriented bundles, then with the induced orientations:
  \[{\rm e}(A\otimes B)=\langle-1\rangle^N{\rm e}(B\otimes A),\]
  where $N=\binom{\operatorname{rk} A}{2}\cdot \binom{\operatorname{rk} B}{2}$.
\end{proposition}

\begin{proof}
 Consider the following diagram:
  \[\xymatrix{
    \det (A\otimes B)\ar[d]^{\det F} \ar[r]^{\varphi_{AB}^{-1}}& \det A\otimes\det B\ar[d]^{(-1)^N\cdot F'}\ar[rr]^-{\omega_A\otimes \omega_B}&&(\L\otimes\mathscr{M})^{\otimes 2}\ar[d]^{(-1)^N(F'')^{\otimes 2}}&\\			
    \det(B\otimes A) \ar[r]_{\varphi_{BA}^{-1}} & \det B \otimes \det A\ar[rr]_-{\omega_B\otimes \omega_A}&& (\mathscr{M}\otimes\L)^{\otimes 2}& \phantom{a} \hspace{-0.5 cm},
  }
  \]
  where $F\colon A\otimes B\to B\otimes A$, $F'\colon \det A\otimes \det B\to \det B\otimes \det A$ and $F''\colon \L\otimes\mathscr{M}\to\mathscr{M}\otimes\L$ are the appropriate flip maps. Commutativity of the left square is a consequence of the following property of lexicographic orderings: if $M$ and $N$ are finite sets, then the relative sign of the lexicographic ordering of their product $M\times N$ with respect to the lexicographic ordering of $N\times M$ is $(-1)^{\binom{|M|}{2}\cdot \binom{|N|}{2}}$ (after identifying $(m,n)$ with $(n,m)$).\footnote{It is a nice combinatorics exercise to check that the number of inversions is equal to the number of pairs $(\{m\neq m'\},\{n\neq n'\})$.} The commutativity of the right square follows: it arises from an obviously commutative square with $\L^{\otimes 2}\otimes\mathscr{M}^{\otimes 2}\xrightarrow{\cong} \mathscr{M}^{\otimes 2}\otimes\L^{\otimes 2}$ on the right-hand side by multiplying the vertical maps with $(-1)^N$ and making an even number of flips of tensor factors in the right-hand vertical map. The claim then follows from Lemmas~\ref{lem:euler-iso} and \ref{lem:change-orientation}.
\end{proof}

Next we will give a formula for Euler classes of tensor products involving the Pontryagin classes of the bundles. First, we have to introduce some notation. Given indeterminates $v=(v_i)_{i\in \Z}$ and a partition $\la$, introduce the notation for the determinant:
\begin{equation}\label{eq:JacobiTrudidet}
	\De_{\la}(v):=\det(v_{\la_i+j-i})_{i,j=1\stb \ell(\la)}.
\end{equation}
The (second) Jacobi--Trudi formula \cite[I.(3.5)]{MacDonald} expresses the Schur polynomial in variables $x_i$ in terms of elementary symmetric polynomials $e_j(x_1\stb x_n)$:
$$ s_\la(x_1\stb x_n)=\De_{\la^T}(e_1\stb e_n),$$
where $e_j=0$ for $j>n$ and $j<0$.

For Pontryagin classes of even rank bundles we have a splitting principle, see Corollary~\ref{cor:splitting} and Remark~\ref{rem:splitting-principle}. In particular, this enables the use of the calculus of elementary symmetric polynomials and Schur polynomials. We give universal identities between characteristic classes by dealing with the universal case, e.g.\ by working over large enough Grassmannians, or by applying Corollary~\ref{cor:splitting} to the case at hand.

\begin{proposition}[Cauchy identity]\label{prop:Cauchy}
  Let $A$ and $B$ be vector bundles of rank $2m$ and $2n$, respectively. Then 
  $$ {\rm e}(A\otimes B)=\sum_{\la\se m\times n} (-1)^{|\tilde{\la}|}\De_{\la^T}\bigl({\rm p}(A)\bigr)\De_{\tilde{\la}^T}\bigl({\rm p}(B)\bigr),$$
  where $\la\se m\times n$ are partitions fitting into an $m\times n$ box, $\tilde{\la}$ denotes the dual partition to $\la\se m\times n$ and $\mu^T$ denotes the transpose of a partition $\mu$. For the definition of the determinant $\De_{\la}$, see Definition \eqref{eq:JacobiTrudidet}, ${\rm p}(A)=\bigl({\rm p}_{2i}(A)\bigr)_{i\in \Z}$ and ${\rm p}(B)=\bigl({\rm p}_{2i}(B)\bigr)_{i\in \Z}$ (with ${\rm p}_{<0}=0$).
\end{proposition}

\begin{proof}
  Write ${\rm p}(A)=\prod_{i=1}^m(1+a_i^2)$ and ${\rm p}(B)=\prod_{j=1}^n(1+b_j^2)$. Then by Proposition~\ref{prop:levine}, we have 
  $$ {\rm e}(A\otimes B)=\prod_{i=1}^m\prod_{j=1}^n(a_i^2-b_j^2)$$
  and substituting $\al_i=a_i^2$ and $\be_i=-b_i^2$ in the classical Cauchy identity \cite[I.4. Ex.\ 5]{MacDonald}
  $$\prod_{i,j}(\al_i+\be_j)=\sum_{\la\se m\times n} s_\la(\al)s_{\tilde{\la}}(\be)$$
  proves the claim. Indeed, since $e_j(a_1^2\stb a_m^2)=\p_j(A)$ and $e_j(b_1^2\stb b_n^2)=\p_j(B)$, so by \eqref{eq:JacobiTrudidet}:
  \[
  e(A\otimes B)=\prod_{i=1}^m\prod_{j=1}^n(a_i^2-b_j^2)=\sum_{\la\se m\times n} s_\la(a_i^2)s_{\tilde{\la}}(-b_j^2)=\sum_{\la\se m\times n} \De_{\la^T}\bigl({\rm p}(A)\bigr)(-1)^{|\tilde \la|}\De_{\tilde{\la}^T}\bigl({\rm p}(B)\bigr)
  \]
\end{proof}

\begin{corollary}
  \label{cor:tensorrk2}
  In the situation of Proposition~\ref{prop:Cauchy}, if $n=1$, with ${\rm e}(B)=b$, then
  $${\rm e}(A\otimes B)=\sum_{k=0}^m(-1)^k{\rm p}_{2(m-k)}(A)b^{2k}.$$
\end{corollary}

\begin{proof}
  This is a special case of Proposition~\ref{prop:Cauchy} when $n=1$, the sum goes over the one-column partitions $1^i\se m\times 1$ and $\De_{i}\bigl(\p(A)\bigr)=\p_i(A)$.

  Alternatively, it can also be seen directly. Using the splitting principle, cf.\ Corollary~\ref{cor:splitting} and Remark~\ref{rem:splitting-principle}, 
  we can assume $A=A_1\oplus\cdots \oplus A_m$ with rank $2$ bundles $A_i$ with Euler class ${\rm e}(A_i)=a_i$. By multiplicativity of the Euler class and Proposition~\ref{prop:levine}, we have
  $$ {\rm e}(A\otimes B)=\prod_{i=1}^m(a_i^2-b^2)=\sum_{k=0}^{m} e_{m-k}(a_i^2)(-b^2)^k,$$
  where $e_j(x_i)$ is the $j$th elementary symmetric polynomial in variables $x_i$. The claim follows from the splitting principle, expressing the Pontryagin classes as elementary symmetric polynomials in the Pontryagin roots, as in \cite[Proposition 4.8]{chowwitt}. 
\end{proof}

\subsubsection{Symmetric powers}

Now we turn to tensor products of symmetric powers. Introduce the notation for the ``quadratic specialization'' of Schur polynomials; for $M\in \Z$, let
$$ q_\la(M):=s_\la\bigl(M^2, (M-2)^2, \ldots\bigr),$$
where the last element is 1 or 0 depending on the parity of $M$. 
\begin{theorem}
  Let $A,B$ be rank 2 bundles with Euler classes $a={\rm e}(A)$, $b={\rm e}(B)$. Let 
  $$ E=\sum_{\la\se m\times n}(-1)^{|\tilde{\la}|}q_\la(M)q_{\tilde{\la}}(N)a^{2|\la|}b^{2|\tilde{\la}|},$$
where $m=\lfloor\frac{M+1}{2}\rfloor$, $n=\lfloor\frac{N+1}{2}\rfloor$. Then 
  $$ {\rm e}\bigl({\rm Sym}^M(A)\otimes \op{Sym}^N(B)\bigr)=\begin{cases}
		E\qquad &\text{if both }M, N \text{ are odd},\\
		E\cdot N!!\cdot b^n\qquad &\text{if }M \text{ is even and }  N \text{ is odd},\\
		E\cdot M!!\cdot a^m\qquad &\text{if }N \text{ is even and }  M \text{ is odd},\\
		0\qquad &\text{if both }M, N \text{ are even}.
  \end{cases}
  $$
\end{theorem}

\begin{proof}
By Proposition~\ref{prop:levine}: 
  \[
  {\rm p}\bigl(\op{Sym}^M (A)\bigr)=\prod_{i=0}^{\lfloor M/2\rfloor }\bigl(1+(M-2i)^2a^2\bigr),
  \]
so
\[
\De_{\la^T}\Bigl({\rm p}_1\bigl(\op{Sym}^M (A)\bigr)\stb {\rm p}_m(\op{Sym}^M (A)\bigr)\Bigr)=s_\la\bigl(M^2a^2, (M-2)^2a^2,\ldots\bigr)=q_\la(M) a^{2|\la|}
\]
and a similar computation gives 
\[\De_{\la^T}\Bigl({\rm p}\bigl(\op{Sym}^N(B)\bigr)\Bigr)=q_{\la}(N)b^{2|\la|}.\] If $M,N$ are odd, then $\op{Sym}^M(A)$ and $\op{Sym}^N(B)$ are even rank bundles (of ranks $M+1$ and $N+1$), so Cauchy's identity (Proposition~\ref{prop:Cauchy}) directly applies, and we get the claim.

When $M$ is even, for rank 2 bundles $A$, the bundle $\op{Sym}^M(A)$ splits as direct sum of a one-dimensional representation and another representation which is obtained from $\op{Sym}^{M-1}(A)$ by a sign change in some direct summands, see the proof of \cite[Theorem 8.1]{levine:schur}. This influences the sign of the Euler class of the respective direct summand, see \cite[Theorem 7.1]{levine:schur}, but doesn't affect the Euler class of the symmetric power because the Euler class enters quadratically. In particular, we get an equality of Euler classes 
\[
  \op{e}\left(\op{Sym}^M(A)\otimes \op{Sym}^N(B)\right)=\op{e}\left(\op{Sym}^{M-1}(A)\otimes \op{Sym}^N(B)\oplus \op{Sym}^N(B)\right).
\]
Then we can apply the already established case when both $M-1$ and $N$ are odd, and Proposition~\ref{prop:levine} applies to the second term. The other odd-even case follows by a symmetric argument. Finally, when both $M,N$ are even, then $\op{Sym}^M(A)\otimes\op{Sym}^N(B)$ splits off a rank one trivial bundle, so the Euler class is 0.
\end{proof}

\Yboxdim{3pt}
\begin{example}\label{ex:Sym5rk2}
For ${\rm e}\bigl(\op{Sym}^3(A)\otimes \op{Sym}^2(B)\bigr)$, we have $m=2$ and $n=1$. The quadratic specialization of Schur polynomials are
$$ q_{\yng(1)}(3)=10,\qquad q_{\yng(1,1)}(3)=9,\qquad q_{\yng(1)}(2)=4,\qquad q_{\yng(2)}(2)=16,$$
so 
$$ {\rm e}\bigl(\op{Sym}^3(A)\otimes \op{Sym}^2(B)\bigr)=(9a^4-40a^2b^2+16b^4)3a^2.$$

For ${\rm e}\bigl(\op{Sym}^4(A)\otimes \op{Sym}^1(B)\bigr)$, we have $m=2$ and $n=1$. The quadratic specialization of Schur polynomials are
$$ q_{\yng(1)}(4)=20,\qquad q_{\yng(1,1)}(4)=64,\qquad q_{\yng(1)}(1)=1,\qquad q_{\yng(2)}(1)=1,$$
so
$$ {\rm e}\bigl(\op{Sym}^4(A)\otimes \op{Sym}^1(B)\bigr)=(64a^4-20a^2b^2+b^4)b.$$
\end{example}

\end{document}